\documentclass[
a4paper,reqno]{amsart}
\usepackage[T1]{fontenc}
\usepackage[english]{babel}
\usepackage{amssymb,bbm,enumerate}
\usepackage{color}
\usepackage[linktocpage=true,colorlinks=true, linkcolor=blue, citecolor=red, urlcolor=green]{hyperref}
\usepackage[all]{xy}
\usepackage{tikz-cd}

\usepackage{float}
\restylefloat{table}

\usepackage{tikz}

\renewcommand{\ker}{\Ker}

\newcommand{\mc}[1]{\mathcal{#1}}
\newcommand{\mf}[1]{\mathfrak{#1}}
\newcommand{\mb}[1]{\mathbb{#1}}

\newcommand{\id}{\mathbbm{1}}

\DeclareMathOperator{\Mat}{Mat}
\DeclareMathOperator{\Hom}{Hom}

\DeclareMathOperator{\im}{Im}

\DeclareMathOperator{\Der}{Der}
\DeclareMathOperator{\Ker}{Ker}

\DeclareMathOperator{\sign}{sign}

\DeclareMathOperator{\CEnd}{CEnd}
\DeclareMathOperator{\RCEnd}{RCEnd}
\DeclareMathOperator{\CDer}{CDer}
\DeclareMathOperator{\RCDer}{RCDer}
\DeclareMathOperator{\CHom}{CHom}
\DeclareMathOperator{\RCHom}{RCHom}

\DeclareMathOperator{\Cas}{Cas}

\theoremstyle{plain}
\newtheorem{theorem}{Theorem}[section]
\newtheorem{lemma}[theorem]{Lemma}
\newtheorem{proposition}[theorem]{Proposition}
\newtheorem{corollary}[theorem]{Corollary}

\theoremstyle{definition}
\newtheorem{definition}[theorem]{Definition}
\newtheorem{example}[theorem]{Example}

\theoremstyle{remark}
\newtheorem{remark}[theorem]{Remark}

\numberwithin{equation}{section}

\definecolor{light}{gray}{.9}

\begin{document}

\title[LCAd cohomology]{Cohomology of Lie conformal algebroids}

\author{Alberto De Sole}
\address{Dipartimento di Matematica  \& INFN, Sapienza Universit\`a di Roma,
P.le Aldo Moro 5, 00185 Rome, Italy}
\email{desole@mat.uniroma1.it}
\urladdr{www1.mat.uniroma1.it/\$$\sim$\$desole}
\author{Jiefeng Liu}
\address{School of Mathematics and Statistics, Northeast Normal University, Changchun, 130024, China}
\email{liujf534@nenu.edu.cn}
\author{Daniele Valeri}
\address{Dipartimento di Matematica \& INFN, Sapienza Universit\`a di Roma,
P.le Aldo Moro 5, 00185 Rome, Italy}
\email{daniele.valeri@uniroma1.it}



\begin{abstract}
We study Lie conformal algebroids (LCAd) and their representations using the language of $\lambda$-brackets and Lie conformal algebras. 
We describe several general constructions, such as the LCAd of conformal derivations $\CDer(A)$ of a differential algebra $A$,
the gauge LCAd $\mc G(A,M)$ associated to a differential algebra $A$ and its module $M$,
the current LCAd $\widehat{F}$ of a Lie algebroid $F$,
and the LCAd structure of the space $\Omega(\mc V)$ of K\"ahler differentials over a Poisson vertex algebra (PVA) $\mc V$.
We develop the cohomology theories of LCAd 
and we relate them to the corresponding cohomology theories of PVA.
In particular, we find an isomorphism between the cohomology of a PVA $\mc V$ with coefficients in a module $M$
and the corresponding cohomology of the LCAd $\Omega(\mc V)$ with coefficients in the same module.
\end{abstract}


\maketitle


\section{Introduction}\label{sec:intro}

Lie conformal algebras and Poisson vertex algebras are algebraic structures that naturally arise in the formalism of integrable systems, conformal field theory, and the representation theory of infinite-dimensional Lie algebras.

Poisson vertex algebras (PVA), also known as Coisson algebras in the language of \cite{BD}, 
are obtained as the classical limit of vertex algebras and 
provide an algebraic framework for the theory of Hamiltonian evolutionary partial differential equations \cite{BSK}.
They generalize the notion of Poisson algebras in the context of Lie conformal algebras (LCA), encoding the local Poisson bracket of local functionals via the PVA $\lambda$-bracket. It is therefore not surprising that they have become central objects in the algebraic theory of integrable systems and the formal calculus of variations.

Cohomology theories for PVA were systematically developed by the first author and Kac \cite{DK11a,DK11b}. In these papers it was introduced the variational PVA complex and its covering complex, called the basic complex.
The variational PVA cohomology has deep relations both with the deformation theory of PVA and with the theory of integrable systems. As an example, given a differential algebra
$\mc V$  endowed with two compatible PVA structures, the Lenard-Magri scheme of integrability can be infinitely extended, provided that the first PVA cohomology of $\mc V$ vanishes for one of the two structures.

In \cite{BDSHK1,BDSHK2} cohomology theories for both vertex algebras (VA) and PVA were developed through an operadic approach,
leading to the so called classical cohomology complex for PVA.
In those and the subsequent papers  \cite{BDHKV,BDK,BDK2} computational techniques and the connection between the classical and the variational PVA complexes were also explored. These developments have led to a coherent picture of the deformation theory of PVA, and
have clarified the structure and applications of these cohomologies, placing them in close relation to the cohomologies of LCA and VA \cite{BDK2,BKV,DK09}.

On the other hand, the theory of LCA, initiated in \cite{DAK}, has evolved into a broader framework encompassing their modules and interactions with differential structures. This leads naturally to the notion of a Lie conformal algebroid (LCAd) \cite{DK09}, a conformal analogue of a Lie algebroid (LAd).
LCA's and LCAd's correspond to Lie$^*$ algebras and Lie$^*$ algebroids, respectively, in the terminology of \cite{BD},
where such structures are introduced within the framework of pseudotensor categories (and their generalization, compound tensor categories).
By definition, a \emph{Lie conformal algebroid} over a differential algebra $A$ is a left $A[\partial]$-module $E$, endowed with an LCA $\lambda$-bracket
$[\cdot\,_\lambda\,\cdot]:\,E\otimes E\to E[\lambda]$
and with a left $A[\partial]$-module homomorphism $\theta:E\to\CDer(A)$, called the anchor map, such that
(cf. Definition \ref{def:LCAd}):
\begin{enumerate}[(i)]
\item
$[u_\lambda av]=\theta(u)_\lambda(a) v+ a[u_\lambda v]$,
\item
$\theta([u_\lambda v])_\mu(a)
= \theta(u)_\lambda(\theta(v)_{\mu-\lambda}(a))-\theta(v)_{\mu-\lambda}(\theta(u)_\lambda(a))$,
\end{enumerate}
for $u,v\in E$ and $a\in A$.
Here $\CDer(A)$ is the space of conformal derivations of $A$, i.e. maps $\phi_\lambda:\,A\to A[\lambda]$
such that $\phi_\lambda(\partial a)=(\partial+\lambda)\phi_\lambda(a)$ and $\phi_\lambda(ab)=\phi_\lambda(a)b+a\phi_\lambda(b)$.
It has a left $A[\partial]$-module structure given by $(a\phi)_\lambda=(|_{x=\partial}a)\phi_{\lambda+x}$ and $(\partial\phi)_\lambda=-\lambda\phi_\lambda$.
It also has a (not necessarily polynomial valued) $\lambda$-bracket
given by $[\phi_\lambda\psi]_\mu(a)=\phi_\lambda(\psi_{\mu-\lambda}(a))-\psi_{\mu-\lambda}(\phi_\lambda(a))$,
satisfying the LCA axioms.
Hence, condition (ii) above is saying that the anchor map $\theta$ maps the $\lambda$-bracket of $E$
to the $\lambda$-bracket of $\CDer(A)$.
Note that an LCAd over the base field $\mb F$ is the same as an LCA.
Recall also that a \emph{module} over a LCAd $E$ (over the differential algebra $A$)
is a left $A[\partial]$-module endowed with a LCA $\lambda$-action of $E$ on $M$,
denoted $u\in E, m\in M\mapsto u_\lambda M$, satisfying the following conditions:
(i) $u_\lambda(a(\partial)m)=a(\lambda+x)\big(\big|_{x=\partial}u_\lambda m\big)
+\theta(u)_\lambda(a(x))\big(\big|_{x=\partial}m\big)$;
(ii) $(a(\partial)u)_\lambda m = \big(\big|_{x=\partial}a^*(\lambda)\big)(u_{\lambda+x}m)$.

While, as said above, the cohomology theories of LCA and PVA have been extensively investigated, a systematic study of LCAd cohomology was still missing.
The main purpose of the present paper is to fill this gap.

The definition of an LCAd is reviewed in Section \ref{sec:2}, together with some basic constructions and several illustrative examples:
we provide $\CDer(A)$ with a structure of an LCAd (Example \ref{ex:CDer-LCAd}), we present the \emph{gauge} LCAd $\mc G(A,M)$
associated to an $A[\partial]$-module $M$ (Section \ref{sec:2.15}), and the \emph{current} LCAd $\widehat{F}$ associated to a LAd $F$ 
(Section \ref{sec:2.2b}).
Another example, the LCAd of \emph{K\"ahler differentials} $\Omega(\mc V)$, associated to a PVA $\mc V$,
is actually presented in Section \ref{sec:5}, as it provides a bridge between PVA cohomology and LCAd cohomology.

Recall that \cite{H77,J69} the space $\Omega(A)$ of K\"ahler differentials of a commutative associative algebra $A$ is the unique
$A$-module endowed with a derivation $d:\,A\to M$ satisfying the following universal property:
for every $A$-module $M$ and every derivation $\delta:\,A\to M$, there is a unique $A$-module homomorphism
$\widetilde{\delta}:\,\Omega(A)\to M$ such that $\widetilde\delta\circ d=\delta$ (cf. the diagram \eqref{eq:diagram-delta}).
Explicitly, $\Omega(A)$ is spanned by elements of the form $ad(b)$, for $a,b\in A$.
In the special case when $A$ is a differential algebra, the space $\Omega(A)$ become a left $A[\partial]$-module,
with the action of $\partial$ given by $\partial(ad(b))=(\partial a)d(b)+ad(\partial b)$.
In this case, if $M$ is a left $A[\partial]$-module
and $\delta:\,A\to M$ is a derivation commuting with $\partial$,
then the map $\widetilde\delta:\,\Omega(A)\to M$ defined by the universal property
is an $A[\partial]$-module homomorphism, see Proposition \ref{prop:univ-prop1}.
Furthermore, if $\mc V$ is a PVA, the left $\mc V[\partial]$-module $\Omega(\mc V)$
acquires a natural structure of a LCAd,
with the $\lambda$-bracket defined on generators by
$[da_\lambda db]=d(\{a_\lambda b\})$ (for $a,b\in\mc V$),
and with the anchor map defined on generators by $\theta(da)_\lambda(b)=\{a_\lambda b\}$ (for $a,b\in\mc V$),
see Theorem \ref{thm:kahlerLCAd}.
Moreover, a module $M$ over the PVA $\mc V$
is naturally a module over the LCAd $\Omega(\mc V)$, and viceversa, see Proposition \ref{prop:PVA-LCAd-modules}.
The main result of the paper, Theorem \ref{thm:main},
provides an explicit isomorphism between the various cohomology complexes associated to the PVA $\mc V$
and its module $M$,
and the corresponding cohomology complexes associated to the LCAd $\Omega(\mc V)$
and the same module $M$.

In fact, there are three different variations of PVA cohomology theories, see e.g. \cite{DK11a},
based on the basic, the reduced and the variational complexes,
which are reviewed in Sections \ref{sec:6.1}-\ref{sec:6.2}.
Analogously, we introduce in Section \ref{sec:3} three different variations of LCAd cohomology theories,
called the basic, the reduced and the LCAd complexes.
These complexes naturally extend the corresponding LCA complexes of \cite{BKV}-\cite{DK09} 
by taking into account the presence of the anchor map.
Their construction is also motivated by analogy with the Chevalley-Eilenberg cohomology for LAd, see e.g. \cite{Mac}.
The resulting cohomology theory provides a natural setting in which to study derivations, abelian extensions, and formal deformations of LCAd.

Explicitly, the \emph{LCAd cohomology complex} of a LCAd $E$ with coefficients in a module $M$ is defined as follows.
The space $C^k(E,M)$ of $k$-cochains consists of maps 
\begin{equation}\label{eq:k-cochains-intro}
\varphi_{\lambda_1,\dots,\lambda_k}:\,E^{\otimes k}\to M[\lambda_1,\dots,\lambda_k]/\langle\partial+\lambda_1+\dots+\lambda_k\rangle
\,,
\end{equation}
mapping $u_1\otimes\dots\otimes u_k\mapsto\varphi_{\lambda_1,\dots,\lambda_k}(u_1,\dots,u_k)$,
satisfying conditions \eqref{eq:poly-lambda} and \eqref{eq:poly-skew}.
For example, $C^0(E,M)\simeq M/\partial M$ and $C^1(E,M)\simeq \Hom_{A[\partial]}(E,M)$.
The differential $d:\,C^k(E,M)\to C^{k+1}(E,M),\,k\geq0$ is defined by
\begin{equation}\label{eq:differential-intro}
\begin{split}
& 
(d\varphi)_{\lambda_1,\dots,\lambda_{k+1}}(u_1,\dots,u_{k+1})
=
\sum_{i=1}^{k+1}(-1)^{i+1}
{u_i}_{\lambda_i}\Big(
\varphi_{\lambda_1,\stackrel{i}{\check\dots},\lambda_{k+1}}(u_1,\stackrel{i}{\check\dots},u_{k+1})
\Big) \\
&\qquad +
\sum_{\substack{i,j=1 \\ i<j}}^{k+1}(-1)^{i+j}
\varphi_{\lambda_i+\lambda_j,\lambda_1,\stackrel{i}{\check\dots}\stackrel{j}{\check\dots},\lambda_{k+1}}
\big([{u_i}_{\lambda_i}{u_j}],u_1,\stackrel{i}{\check\dots}\stackrel{j}{\check\dots},u_{k+1}\big)
\,.
\end{split}
\end{equation}
We show in Proposition \ref{prop:LCAd-cohomology} that, indeed, $d$ is well defined and $d\circ d=0$,
making $(C^\bullet(E,M),d)$ a cohomology complex.
Similarly, the \emph{basic LCAd cohomology} complex $\widetilde C^\bullet(E,M)$ of $E$ with coefficients in $M$
is defined by taking maps as in \eqref{eq:k-cochains-intro} but with values
in $M[\lambda_1,\dots,\lambda_k]$ instead of the quotient by $\langle\partial+\lambda_1+\dots+\lambda_k\rangle$,
and the formula for the differential $\widetilde{d}$ is again \eqref{eq:differential-intro}.
This complex has a natural $\mb F[\partial]$-module structure, with the action of $\partial$ given by
\begin{equation}\label{eq:partial-basic-intro}
(\partial\widetilde\varphi)_{\lambda_1,\dots,\lambda_k}(u_1,\dots,u_k)
=
(\lambda_1+\dots+\lambda_k+\partial)
\widetilde\varphi_{\lambda_1,\dots,\lambda_k}(u_1,\dots,u_k)
\,,
\end{equation}
and the differential $\widetilde{d}$ commutes with this action of $\partial$, see Proposition \ref{prop:reduced-LCAd-cohomology}.
Hence, we can quotient by the action of $\partial$, to get the \emph{reduced LCAd cohomology complex}
$(\overline C^\bullet(E,M),\overline{d})$.

As mentioned above, Theorem \ref{thm:main} states that, given a PVA $\mc V$ and a PVA module $M$, there is a natural isomorphism
of complexes, hence of cohomologies,
$$
H_{\textrm{var}}(\mc V,M)\simeq H_{\textrm{LCAd}}(\Omega(\mc V),M)
\,,
$$
between the variational PVA cohomology of $\mc V$ with coefficients in $M$
and the LCAd cohomology of $\Omega(\mc V)$, again with coefficients in $M$.
The same result holds for the basic and the reduced cohomology theories, see Remark \ref{rem:5.3}.

This paper lays the ground for several directions of study, such as applications to deformation theory and quantization. In particular, there should be an analogue correspondence between the cohomology theories of a vertex algebra, which is the quantization of a PVA, and the corresponding cohomology theories of a quantization of a LCAd. This should be related to vertex (or chiral) algebroids \cite{BD,GMS}.

Throughout the paper all vector fields, linear maps, tensor products, etc.
are considered over a field $\mb F$ of characteristics $0$.

\section*{Acknowledgments}
The present research was carried out during the visit of J. Liu at the University of Rome “La Sapienza”, and he gratefully acknowledges the hospitality of the Mathematics Department. 
J. Liu is supported by the National Key Research and Development Program of China (2022T150109), the China Scholarship Council (202406620125) and NSFC (12371029, W2412041).
A. De Sole and D.Valeri are members of the GNSAGA INdAM and of the project MMNLP (Mathematical Methods in Non Linear Physics) of INFN.
A.~De Sole has been supported by the national PRIN grant 2022S8SSW2 of MUR.
D. Valeri  has been supported by the national PRIN grant 2022HMBTTL of MUR.

\section{Lie conformal algebroids}\label{sec:2}

\subsection{Conformal endomorphisms, conformal derivations, and their LCA structures}\label{sec:2.10}

Let $A$ be a \emph{differential algebra}. 
By this we mean a commutative, associative, unital algebra
with a derivation $\partial:\,A\to A$.
Associated to it we have the associative algebra $A[\partial]$ of (polynomial) differential operators,
with product given by composition.

The following notation will be used throughout the paper:
given $a(\partial)=\sum_ia_i\partial^i\in A[\partial]$
and $b,c\in A$, we let:
\begin{equation}\label{eq:notation}
a(\lambda+x)\big(\big|_{x=\partial}b)c
=c\sum_{i=0}^Na_i(\lambda+\partial)^i(b)\in A[\lambda]
\,.
\end{equation}
For example, the symbol of the adjoint operator $a^*(\partial)=\sum_i(-\partial)^i\circ a_i\in A[\partial]$
is $a^*(\lambda)=\big(\big|_{x=\partial}a(-\lambda-x)\big)$.

Recall that a \emph{(left) conformal homomorphism} of $\mb F[\partial]$-modules from $M$ to $N$
is a linear map $\phi_\lambda:\,M\to N[\lambda]$ 
satisfying $\phi_\lambda(\partial u)=(\partial+\lambda)\phi_\lambda(u)$.
We denote by $\CHom(M,N)$ the space of conformal homomorphisms from $M$ to $N$
and by $\CEnd(M)=\CHom(M,M)$.
It has a structure of an $\mb F[\partial]$-module, with the action of $\partial$ given by 
$(\partial\phi)_\lambda=-\lambda\phi_\lambda$.
More in general, if $M$ and $N$ are $A[\partial]$-modules,
a (left) conformal homomorphism of $A[\partial]$-modules from $M$ to $N$
is conformal homomorphism of $\mb F[\partial]$-modules commuting with the action of $A$,
i.e. $\phi_\lambda(a(\partial) u)=a(\partial+\lambda)\phi_\lambda(u)$ for $a(\partial)\in A[\partial]$.
We denote by $\CHom_A(M,N)$ the space of (left) conformal homomorphisms 
of $A[\partial]$-modules from $M$ to $N$.
Both $\CHom(M,N)$ and its subspace $\CHom_A(M,N)$ are left $A[\partial]$-modules, with the action of $a(\partial)\in A[\partial]$ given by
\begin{equation}\label{eq:lcder-mod}
(a(\partial)\phi)_\lambda
=
\big(\big|_{x=\partial}
a^*(\lambda)
\big)
\phi_{\lambda+x}
\,.
\end{equation}
Recall also that a \emph{right conformal homomorphism} of $A[\partial]$-modules from $M$ to $N$
is a linear map $\psi_\lambda:\,M\to N[\lambda]$ 
satisfying $\psi_\lambda(a(\partial) u)=\big(\big|_{x=\partial}a^*(\lambda)\big)\psi_{\lambda+x}(u)$.
We denote by $\RCHom_A(M,N)$ the space of right conformal homomorphisms 
of $A[\partial]$-modules from $M$ to $N$, and by $\RCEnd_A(M)=\RCHom_A(M,M)$.
It is a left $A[\partial]$-module, with the action of $a(\partial)\in A[\partial]$ given by 
\begin{equation}\label{eq:rcder-mod}
(a(\partial)\psi)_\lambda
=
a(\lambda+\partial)\circ\psi_\lambda
\,.
\end{equation}
We also let $\RCHom(M,N)=\RCHom_{\mb F}(M,N)$.
We have an isomorphism of $A[\partial]$-modules $\CHom_A(M,N)\to\RCHom_A(M,N)$,
 $\phi\mapsto\phi^*$, defined by
\begin{equation}\label{eq:phi-dual}
\phi^*_\lambda(u)=\big(\big|_{x=\partial}\phi_{-\lambda-x}(u)\big)
\,.
\end{equation}

A \emph{(left) conformal derivation} from the differential algebra $A$ 
to a left $A[\partial]$-module $M$ is a linear map 
$\phi_\lambda:\,A\to M[\lambda]$ satisfying
($a(\partial)\in A[\partial],\,b\in A$)
\begin{equation}\label{eq:lcder}
\phi_\lambda(a(\partial)b)
=
\big(\big|_{x=\partial}b\big)\phi_\lambda(a(x)) 
+
a(\lambda+x)\big(\big|_{x=\partial}\phi_\lambda(b)\big)
\,.
\end{equation}
We denote by $\CDer(A,M)$ the space of conformal derivations from $A$ to $M$,
which has the structure of a left $A[\partial]$-module given by the same equation as \eqref{eq:lcder-mod}.
Clearly, $\CDer(A,M)$ is an $A[\partial]$-submodule of $\CHom_A(A,M)$.
We also denote $\CDer(A)=\CDer(A,A)$.
Moreover, a \emph{right conformal derivation} 
from the differential algebra $A$ to the left $A[\partial]$-module $M$
is a linear map 
$\psi_\lambda:\,A\to M[\lambda]$ satisfying
($a(\partial)\in A[\partial],\,b\in A$)
\begin{equation}\label{eq:rcder}
\psi_\lambda(a(\partial)b)
=
\big(\big|_{x=\partial}
a^*(\lambda)
\big)
\psi_{\lambda+x}(b)
+
\big(\big|_{x=\partial}b\big) \psi_{\lambda+x}(a(x))
\,.
\end{equation}
We denote by $\RCDer(A,M)$ the space of right conformal derivations from $A$ to $M$,
which has the structure of a left $A[\partial]$-module given by 
the same equation as \eqref{eq:rcder-mod}.
Clearly, $\RCDer(A,M)$ is an $A[\partial]$-submodule of $\RCHom_A(A,M)$.
We also denote $\RCDer(A)=\RCDer(A,A)$.
It is easy to see that if $\phi\in\CDer(A,M)$, then $\phi^*\in\RCDer(A,M)$,
and conversely if $\psi\in\RCDer(A,M)$, then $\psi^*\in\CDer(A,M)$.
Moreover, the $A[\partial]$-module structures \eqref{eq:lcder-mod} and \eqref{eq:rcder-mod}
of $\CDer(A,M)$ and $\RCDer(A,M)$ are compatible with respect to this bijection.
Hence, the map $\phi\mapsto\phi^*$ restricts to an isomorphism
of left $A[\partial]$-modules $\CHom_A(A,M)\to\RCHom_A(A,M)$.

Recall \cite{Kac} that a \emph{Lie conformal algebra} (LCA) $E$ is an $\mb F[\partial]$-module
endowed with a $\lambda$-bracket $[\cdot\,_\lambda\,\cdot]:\,E\otimes E\to E[\lambda]$
satisfying:
\begin{itemize}
\item[($LCA$-i)] sesquilinearity:
$[\partial u_\lambda v]=-\lambda[u_\lambda v]$,
$[u_\lambda \partial v]=(\partial+\lambda)[u_\lambda v]$;
\item[($LCA$-ii)] skewsymmetry:
$[v_\lambda u]=-\big(\big|_{x=\partial}[u_{-\lambda-x} v]\big)$;
\item[($LCA$-iii)] Jacobi identity:
$[u_\lambda[v_\mu w]]-[v_\mu[u_\lambda w]]=[[u_\lambda v]_{\lambda+\mu}w]$.
\end{itemize}

To construct an example, note that 
we have a (not necessarily polynomial valued)
$\lambda$-bracket on $\CEnd(M)$ defined by the formula:
\begin{equation}\label{eq:lcder-lambda}
[\phi_\lambda\psi]_\mu(v)
=
\phi_\lambda(\psi_{\mu-\lambda}(v))-\psi_{\mu-\lambda}(\phi_\lambda(v))
\,,
\end{equation}
satisfying all conditions (i)--(iii) of an LCA.
Moreover, if $\phi,\psi\in\CDer(A)$, 
then $[\phi_\lambda\psi]$ has all its coefficients in $\CDer(A)$.
Furthermore, if $M$ is a finitely generated $\mb F[\partial]$-module,
then $[\phi_\lambda\psi]$ is polynomial in $\lambda$,
thus making $\CEnd(M)$ an LCA.
Likewise, if $A$ is finitely generated as a differential algebra and $\phi,\psi\in\CDer(A)$,
then $[\phi_\lambda\psi]$ is polynomial in $\lambda$,
thus making $\CDer(A)$ an LCA.
In the following sections, with an abuse of terminology, by an LCA homomorphism $R\to\CDer(A)$ 
we mean an $\mb F[\partial]$-module homomorphism from an LCA $R$ to $\CDer(A)$ 
mapping the LCA $\lambda$-bracket of $R$
to the $\lambda$-bracket \eqref{eq:lcder-lambda}.
This is of course an LCA homomorphism under the assumption that $A$ is finitely generated.
Of course, we have a corresponding LCA $\lambda$-bracket on the space $\RCDer(A)$,
induced by the bijection \eqref{eq:phi-dual}:
\begin{equation}\label{eq:rcder-lambda}
[\phi_\lambda\psi]^r_\mu(v)
=
\phi^*_\lambda(\psi_{\mu}(v))-\psi_{\lambda+\mu}(\phi_\mu(v))
\,,
\end{equation}
so the map $\CDer(A)\to\RCDer(A),\,\phi\mapsto\phi^*$,
becomes an LCA isomorphism.

\subsection{Definition and examples of LCAd}\label{sec:2.1}

\begin{definition}\label{def:LCAd}
A \emph{Lie conformal algebroid} (LCAd) $E$ over a differential algebra $A$ 
is a left $A[\partial]$-module endowed with a left $A[\partial]$-module homomorphism
$\theta:\, E\to \CDer(A)$, called the \emph{anchor map},
and an LCA $\lambda$-bracket $[\cdot\,_\lambda\,\cdot]:\,E\otimes E\to E[\lambda]$,
satisfying
\begin{enumerate}[(i)]
\item 
$[u_\lambda a(\partial)v]
=
a(\lambda+x)\big(\big|_{x=\partial}[u_\lambda v]\big)
+
\theta(u)_\lambda(a(x))\big(\big|_{x=\partial}v\big)$;
\item
$\theta([u_\lambda v])_\mu(a)
=
\theta(u)_\lambda(\theta(v)_{\mu-\lambda}(a))
-\theta(v)_{\mu-\lambda}(\theta(u)_{\lambda}(a))$.
\end{enumerate}
\end{definition}
\noindent
Note that condition (i) can be equivalently written, by skewsymmetry, as
\begin{enumerate}[(i')]
\item 
$[a(\partial)u_\lambda v]
=
\big(\big|_{x=\partial}a^*(\lambda)\big)
[u_{\lambda+x} v]
-
\theta(v)^*_{\lambda+x}(a(x)) \big(\big|_{x=\partial}u\big)$.
\end{enumerate}
In fact, we can put together conditions (i) and (i') to get the following \emph{total formula}:
\begin{equation}\label{eq:tot-form}
\begin{split}
& [a(\partial)u_\lambda b(\partial)v]
=
\big(\big|_{x=\partial}a^*(\lambda)\big)
b(\lambda+x+y)
\big(\big|_{y=\partial}[u_{\lambda+x}v]\big)
 \\
& +
\big(\big|_{x=\partial}a^*(\lambda)\big)
\theta(u)_{\lambda+x}(b(y)) \big(\big|_{y=\partial}v\big) 
- 
b(\lambda+x+y) 
\big(\big|_{y=\partial}\theta(v)^*_{\lambda+x}(a(x))\big)
\big(\big|_{x=\partial}u\big)
\,.
\end{split}
\end{equation}
Note, moreover, that if $A$ is finitely generated as a differential algebra, 
then $\CDer(A)$ is an LCA and condition (ii) is just saying that 
$\theta$ is an LCA homomorphism.

By definition, a \emph{homomorphism} $\varphi:\,E\to F$ of LCAd over $A$ is a linear map which is both
a homomorphism of left $A[\partial]$-modules and a homomorphism of LCA,
and such that $\theta_F(\varphi(u))=\theta_E(u)\in\CDer(A)$ for every $u\in E$.

By definition, an \emph{ideal} of an LCAd $E$ over $A$ is a left $A[\partial]$-submodule $J\subset E$ which is
an ideal with respect to
the LCA $\lambda$-bracket of $E$,
and such that $\theta|_{J}=0$. In this case, obviously, the quotient space $E/J$ has an induced structure of LCAd and the canonical quotient map
$\pi:E\to E/J$ is a homomorphism of LCAd.

The following Lemma will be useful to construct examples of LCAd.
\begin{lemma}\label{lem:gener}
Let $A$ be a differential algebra and let $E$ be a left $A[\partial]$-module
endowed with a left $A[\partial]$-module homomorphism
$\theta:\, E\to \CDer(A)$
and a $\lambda$-bracket $[\cdot\,_\lambda\,\cdot]:\,E\otimes E\to E[\lambda]$,
satisfying conditions (i) and (i') of Definition \ref{def:LCAd}.
\begin{enumerate}[(a)]
\item
Sesquilinearity (LCA-i) holds for every $u,v\in E$.
\item
If skewsymmetry (LCA-ii) holds for given elements $u,v\in E$,
then it holds for $a(\partial)u,b(\partial)v$, for all $a(\partial),b(\partial)\in A[\partial]$.
\item
If condition (ii) of Definition \ref{def:LCAd} holds for given $u,v\in E$ and $c\in A$,
then it holds for $a(\partial)u,b(\partial)v$ and $c\in A$,
for all $a(\partial),b(\partial)\in A[\partial]$.
\item
Assume moreover that skewsymmetry (LCA-ii) and condition (ii) of Definition \ref{def:LCAd} hold.
If the Jacobi identity (LCA-iii) holds for given $u,v,w\in E$,
then it holds for $a(\partial)u,b(\partial)v,c(\partial)w$
for all $a(\partial),b(\partial),c(\partial)\in A[\partial]$.
\end{enumerate}
In particular, $E$ is an LCAd provided that conditions (i) and (i') of Definition \ref{def:LCAd} hold,
and all the axioms of LCAd hold for a set of generators of $E$ as a left $A[\partial]$-module.
\end{lemma}
\begin{proof}
Claim (a) is obvious, since sesquilinearity (LCA-i) is a special case of conditions (i) and (i'),
when $a(\partial)=\partial$.
We already observed that conditions (i) and (i'), together with the assumption that the map $\theta$
commutes with the action of $A[\partial]$, are equivalent to the total formula \eqref{eq:tot-form}.
We then use that to prove claim (b). We have
\begin{align*}
& \big(\big|_{z=\partial}[b(\partial)v_{-\lambda-z}a(\partial)u]\big) \\
& =
\Big(\big|_{z=\partial}
\big(\big|_{x=\partial}b^*({-\lambda-z})\big)
a(-\lambda-z+x+y)
\big(\big|_{y=\partial}[v_{-\lambda-z+x}u]\big) \\
&\quad +
\big(\big|_{x=\partial}b^*(-\lambda-z)\big)
\theta(v)_{-\lambda-z+x}(a(y)) \big(\big|_{y=\partial}u\big) \\
&\quad -
a(-\lambda-z+x+y) 
\big(\big|_{y=\partial}\theta(u)^*_{-\lambda-z+x}(b(x))\big)
\big(\big|_{x=\partial}v\big)
\Big) \\
& =
-b(\lambda+y+z)
\big(\big|_{z=\partial}a^*(\lambda)\big)
\big(\big|_{y=\partial}[u_{\lambda+z}v]\big) \\
&\quad + b(\lambda+y+z)
\big(\big|_{z=\partial} \theta(v)^*_{\lambda+y}(a(y)) \big)
\big(\big|_{y=\partial}u\big) \\
&\quad -
\big(\big|_{z=\partial} a^*(\lambda) \big)
\theta(u)_{\lambda+z}(b(x))
\big(\big|_{x=\partial}v\big)
\Big)
=
-[a(\partial)u_\lambda b(\partial)v]
\,.
\end{align*}

For part (c), we have
\begin{align*}
& \theta\big([a(\partial)u_\lambda b(\partial)v]\big)_\mu(c) \\
& =
\theta\Big(
\big(\big|_{x=\partial}a^*(\lambda)\big)
b(\lambda+x+y)
\big(\big|_{y=\partial}[u_{\lambda+x}v]\big)
 \\
&\quad +
\big(\big|_{x=\partial}a^*(\lambda)\big)
\theta(u)_{\lambda+x}(b(y)) \big(\big|_{y=\partial}v\big) \\
&\quad - 
b(\lambda+x+y) 
\big(\big|_{y=\partial}\theta(v)^*_{\lambda+x}(a(x))\big)
\big(\big|_{x=\partial}u\big)
\Big)_\mu(c) \\
& =
\big(\big|_{x=\partial}a^*(\lambda)\big)
\big(\big|_{y=\partial}b^*(\mu-\lambda)\big)
\theta\big([u_{\lambda+x}v]\big)_{\mu+x+y}(c) \\
&\quad +
\big(\big|_{x=\partial}a^*(\lambda)\big)
\big(\big|_{y=\partial}\theta(u)_{\lambda+x}(b^*(\mu-\lambda))\big)
\theta(v)_{\mu+x+y}(c)
 \\
&\quad - 
\big(\big|_{x=\partial}b^*(\mu-\lambda)\big)
\big(\big|_{y=\partial}\theta(v)_{\mu+x-\lambda}(a^*(\lambda))\big)
\theta(u)_{\mu+x+y}(c)\,,
\end{align*}
where, for the first equality we used equation \eqref{eq:tot-form},
for the second equality we used the assumption that $\theta$ is a morphism of left $A[\partial]$-modules
and the definition \eqref{eq:lcder-mod} of the action of $A[\partial]$ on the space of conformal derivations.
On the other hand, we also have
\begin{align*}
& \big([\theta(a(\partial)u) _\lambda \theta(b(\partial)v)]\big)_\mu(c) \\
& =
\theta(a(\partial)u)_\lambda\big( \theta(b(\partial)v)_{\mu-\lambda}(c)\big)
-
\theta(b(\partial)v)_{\mu-\lambda}\big(\theta(a(\partial)u)_\lambda(c)\big) \\
& =
\big(\big|_{x=\partial}a^*(\lambda)\big)
\theta(u)_{\lambda+x}
\big( 
\big(\big|_{y=\partial}b^*(\mu-\lambda)\big)
\theta(v)_{\mu-\lambda+y}(c)
\big) \\
&\quad -
\big(\big|_{y=\partial}b^*(\mu-\lambda)\big)
\theta(v)_{\mu-\lambda+y}
\big(
\big(\big|_{x=\partial}a^*(\lambda)\big)
\theta(u)_{\lambda+x}(c)
\big) \\
& =
\big(\big|_{x=\partial}a^*(\lambda)\big)
\big(\big|_{y=\partial}b^*(\mu-\lambda)\big)
\theta(u)_{\lambda+x}
\big( 
\theta(v)_{\mu-\lambda+y}(c)
\big) \\
&\quad +
\big(\big|_{x=\partial}a^*(\lambda)\big)
\theta(u)_{\lambda+x} 
\big(\big|_{y=\partial}b^*(\mu-\lambda)\big)
\theta(v)_{\mu-\lambda+y}(c)
\\
&\quad -
\big(\big|_{y=\partial}b^*(\mu-\lambda)\big)
\big(\big|_{x=\partial}a^*(\lambda)\big)
\theta(v)_{\mu-\lambda+y}
\big(
\theta(u)_{\lambda+x}(c)
\big) \\
&\quad -
\big(\big|_{y=\partial}b^*(\mu-\lambda)\big)
\theta(v)_{\mu-\lambda+y}
\big(
\big(\big|_{x=\partial}a^*(\lambda)\big)
\big)
\theta(u)_{\lambda+x}(c)
\,,
\end{align*}
where for the first equality we used \eqref{eq:lcder-lambda},
for the second equality we used the assumption that $\theta$ is a morphism of left $A[\partial]$-modules,
and for the last equality we used the definition of left conformal derivation.
Comparing the right-hand sides of the above two equations,
and using \eqref{eq:lcder-lambda}, we get that 
$\theta[a(\partial)u_\lambda b(\partial)v]
=[\theta(a(\partial)u) _\lambda \theta(b(\partial)v)]$,
as claimed in part (c).

Finally, we prove claim (d).
We introduce the notation
\begin{equation}\label{eq:Jac}
J_{\lambda,\mu}(u,v,w)
=
[u_\lambda[v_\mu w]]
-[v_\mu[u_\lambda w]]
-[[u_\lambda v]_{\lambda+\mu}w]
\,.
\end{equation}
If condition (i) of Definition \ref{def:LCAd} holds, then a direct computation leads to
\begin{equation}\label{eq:J1}
\begin{split}
& J_{\lambda,\mu}(u,v,c(\partial)w)
=
c(\lambda+\mu+x)\big(\big|_{x=\partial}J_{\lambda,\mu}(u,v,w)\big) \\
& +\Big(
\theta(u)_\lambda\theta(v)_\mu(c(x))
-\theta(v)_\mu\theta(u)_\lambda(c(x))
-\theta([u_\lambda v])_{\lambda+\mu}(c(x))
\Big)\big(\big|_{x=\partial}w\big)
\,,
\end{split}
\end{equation}
for every $c(\partial)\in A[\partial]$.
In particular, if condition (ii) of Definition \ref{def:LCAd} holds, 
then
\begin{equation}\label{eq:J2}
J_{\lambda,\mu}(u,v,c(\partial)w)
=
c(\lambda+\mu+x)\big(\big|_{x=\partial}J_{\lambda,\mu}(u,v,w)\big)
\,.
\end{equation}
Moreover, it is easy to check that, if the skewsymmetry condition (LCA-ii) holds, 
\begin{equation}\label{eq:J3}
J_{\lambda,\mu}(u,v,w)
=
\big(\big|_{x=\partial}
J_{\mu,-\lambda-\mu-x}(v,w,u)
\big)
\,.
\end{equation}
Combining equations \eqref{eq:J2} and \eqref{eq:J3}
we immediately get that, if condition (ii) of Definition \ref{def:LCAd} holds, then
\begin{align*}
& J_{\lambda,\mu}(a(\partial)u,b(\partial)v,c(\partial)w) \\
& =
\big(\big|_{x=\partial}a^*(\lambda)\big)
\big(\big|_{y=\partial}b^*(\mu)\big)
c(\lambda+\mu+x+y+z)
\big(\big|_{z=\partial}J_{\lambda+x,\mu+y}(u,v,w)\big)
\,.
\end{align*}
Claim (d) follows from the above identity.
\end{proof}

\begin{example}\label{ex:trivial}
If $A=\mb F$ with $\partial=0$,
then $\CDer(A)=0$, so the anchor map must be trivial.
Hence, an LCAd over $\mb F$ is the same as an LCA.
\end{example}

\begin{example}\label{ex:CDer-LCAd}
Let $A$ be a finitely generated differential algebra.
Then the space $\CDer(A)$ of conformal derivations over $A$,
together with the left $A[\partial]$-module structure \eqref{eq:lcder-mod},
the LCA $\lambda$-bracket \eqref{eq:lcder-lambda}
and the anchor map given by the identity map on $\CDer(A)$ is an LCAd.
In \cite{BD} this is called the \emph{tangent} Lie$^*$ algebroid over $A$.
To check the LCAd axioms, condition (i) holds by a direct computation, while condition (ii) coincides with \eqref{eq:lcder-lambda}.
Likewise, the space $\RCDer(A)$ of right conformal derivations over $A$,
together with the left $A[\partial]$-module structure \eqref{eq:rcder-mod},
the LCA $\lambda$-bracket \eqref{eq:rcder-lambda}
and the anchor map $\RCDer(A)\to\CDer(A)$ mapping $\psi\mapsto\psi^*$, is an LCAd.
Moreover, the map $\CDer(A)\to\RCDer(A)$ mapping $\phi\mapsto\phi^*$, 
is an LCAd isomorphism.
\end{example}

\begin{example}\label{exa:Btilde}
Let $\mc V$ be a Poisson vertex algebra (see the definition in Section \ref{sec:5}).
In particular, $\mc V$ is a differential algebra and we may consider the LCAd $\CDer(\mc V)$ from Example \ref{ex:CDer-LCAd}.
Consider the subspace
$$
B(\mc V)=\Big\{\gamma\in \CDer(\mc V)\,\Big|\,
\gamma_{\lambda}(\{a_\mu b\})
=\{\gamma_\lambda(a)_{\lambda+\mu} b\}
+\{a_{\mu}\gamma_{\lambda}(b)\}\,,
\text{ for all }a,b\in\mc V\Big\}
\,.
$$
By a direct computation one can check that $B(\mc V)$ is an LCA subalgebra with respect to the $\lambda$-bracket \eqref{eq:lcder-lambda} of $\CDer(\mc V)$.
However, it is not a left $\mc V[\partial]$-submodule. Indeed, using  \eqref{eq:lcder-mod} and the axioms of a PVA we get
($f(\partial)\in\mc V[\partial]$, $\gamma\in B(\mc V)$, $a,b\in\mc V$)
\begin{equation}\label{20251013:eq1}
\begin{split}
&(f(\partial)\gamma)_{\lambda}(\{a_\mu b\})
-\big\{(f(\partial)\gamma)_\lambda(a)_{\lambda+\mu} b\big\}
-\big\{a_{\mu}(f(\partial)\gamma)_{\lambda}(b)\big\}
\\
&=\{f^*(\lambda)_{\lambda+\mu+y}b\}\big(\big|_{y=\partial}\gamma^*_\mu(a)\big)
+\big(\big|_{y=\partial}\{a_\mu f^*(\lambda)\}\big)\gamma_{\lambda+\mu+y}(b)
\,.
\end{split}
\end{equation}
On the other hand, consider
$$
\Cas(\mc V)=\{a\in\mc V\mid \{a_\lambda\, \cdot\}=0\}
\,,
$$
which is obviously a differential subalgebra of $\mc V$.
Equation \eqref{20251013:eq1} shows that $B(\mc V)$ is a left $\Cas(\mc V)[\partial]$-submodule of $\CDer(\mc V)$.\
Moreover, for every $\gamma\in B(\mc V)$, $a\in\Cas(\mc V)$ and $b\in\mc V$,
we have $\{\gamma_{\lambda}(a)_\mu b\}=0$, which means that
$\gamma_\lambda(a)\in\Cas(\mc V)[\lambda]$.
Hence, the restriction to $\Cas(\mc V)$ defines a map $\theta:B(\mc V)\to\CDer(\Cas(\mc V))$. It is then
immediate to check that this map satisfies the axioms (i) and (ii) of Definition \ref{def:LCAd}, hence
$B(\mc V)$ is an LCAd over $\Cas(\mc V)$.

For $f\in \mc V$, denote $X_f=\{f_{\lambda}\,\cdot\}$. By the Leibniz rule of the $\lambda$-bracket (see equation \eqref{eq:leib} in Section \ref{sec:5}) 
we have $X_f\in\CDer(\mc V)$, while by the Jacobi identity (LCA-iii) of the $\lambda$-bracket,
we actually have $X_f\in B(\mc V)$.
Consider then the subspace
$$
Z(\mc V)=\{X_f\mid f\in\mc V\}\subset B(\mc V)\,.
$$
By the right Leibniz rule and sesquilinearity of a PVA (see equation \eqref{eq:leibr} in Section \ref{sec:5}),
we have $a(\partial)X_f=X_{a(\partial)f}$, for $a(\partial)\in \Cas(\mc V)[\partial]$, $f\in\mc V$, so that
$Z(\mc V)$ is a left $\Cas(\mc V)[\partial]$-submodule of $B(\mc V)$.
Moreover,
by \eqref{eq:lcder-lambda}, we have $[{X_f}_{\lambda}X_g]=X_{\{f_\lambda g\}}\in Z(\mc V)[\lambda]$, and
obviously $\theta(X_f)=\{f_\lambda\,\cdot\}|_{\Cas(\mc V)}=0$.
Hence, $Z(\mc V)\subset B(\mc V)$ is an LCAd ideal. As a consequence,
$H(\mc V)=B(\mc V)/Z(\mc V)$
is an LCAd over $\Cas(\mc V)$.

In fact, this LCAd coincides, after applying the isomorphism $*:\CDer(\mc V)\to\RCDer(\mc V)$
given by \eqref{eq:phi-dual}, with the first cohomology space $\widetilde{H}^1(\mc V,\mc V)$ of the basic PVA
cohomology complex of $\mc V$ with coefficients in the adjoint module $\mc V$, defined in Section \ref{sec:6.2}. 
\end{example}

\begin{remark}\label{rem:exaB}
There is an LAd analogue of Example \ref{exa:Btilde}, by taking the first cohomology space $H^1(\mc V,\mc V)$
of the variational PVA cohomology (as opposed to the basic PVA cohomology, considered in Example \ref{exa:Btilde}).
It is obtained as follows.
We consider the LAd $\Der(\mc V)$ of all derivations of (the product of) $\mc V$, over the algebra $\mc V$,
and the Lie subalgebra $\Der^\partial(\mc V)$ of all derivations commuting with $\partial$.
We then take the subspace
$$
\overline{B}(\mc V)
=
\big\{D\in\Der^\partial(\mc V)\,\big|\,
D(\{a_\lambda b\})=\{D(a)_\lambda b\}+\{a_\lambda D(b)\}\,,\,\text{ for all } a,b\in\mc V\big\}
\,.
$$
This is a Lie subalgebra of $\Der^\partial(\mc V)$,
it is an $R$-submodule with respect to the multiplication by elements of the algebra $R=\{f\in\mc V\,|\,\partial f=0\}$,
and $[D_1,aD_2]=a[D_1,D_2]+D_1(a)D_2$, for $D_1,D_2\in\overline{B}(\mc V)$ and $a\in R$.
Hence, $\overline{B}(\mc V)$ is a LAd over $R$ with anchor map $\theta:\,\overline{B}(\mc V)\to\Der(R)$
given by $\theta(D)=D|_R$ (recall the definition of a LAd at the beginning of Section \ref{sec:2.2a}).

For $f\in\mc V$, denote $\overline{X}_f=\{f_\lambda\,\cdot\}\big|_{\lambda=0}$, which, by the PVA axioms,
lies in $\overline{B}(\mc V)$.
Consider then the subspace
$$
\overline{Z}(\mc V)
=
\big\{\overline{X}_f\,\big|\,f\in\mc V\big\}\subset\overline{B}(\mc V)
\,,
$$
which is a LAd ideal of $\overline{B}(\mc V)$.
As a consequence, $\overline{H}(\mc V)=\overline{B}(\mc V)/\overline{Z}(\mc V)$ is a LAd over $R$.
\end{remark}

\begin{example}\label{ex:transf-LCAd}
Let $A$ be a differential algebra,
let $(L,\partial,[\cdot\,_\lambda\,\cdot])$ be an LCA,
and let $\phi:\,L\to\CDer(A)$ be an LCA homomorphism.
The associated \emph{transformation LCAd} is defined as the space $E=A\otimes L$,
with the left $A[\partial]$-module structure given by 
$$
a(\partial)(b\otimes u)=\big(a(x+y)\big(\big|_{x=\partial}b\big)\big)\otimes \big(\big|_{y=\partial}u\big)
\,,
$$
the LCA $\lambda$-bracket given by
$$
[a\otimes u_\lambda b\otimes v]_E
=
\big(\big|_{x=\partial}a\big)b\otimes[u_{\lambda+x} v]
+
\big(\big|_{x=\partial}a\big)\phi(u)_{\lambda+x}(b)\otimes v
-
b\phi(v)^*_{\lambda+x}(a)\otimes \big(\big|_{x=\partial}u\big)
\,,
$$
i.e. it is defined as $[\cdot\,_\lambda\,\cdot]$ on $1\otimes L\simeq L$
and extended to $A\otimes L$ via the total formula \eqref{eq:tot-form},
and the anchor map $\theta:\,A\otimes L\to\CDer(A)$ is given by
$$
\theta(a\otimes u)_\lambda(b)
=
\big(\big|_{x=\partial}a\big)\phi(u)_{\lambda+x}(b)
\,.
$$
It is straightforward to check that, indeed, these maps are well defined and the LCAd axioms hold \cite{DK09}.
\end{example}

\subsection{Example: gauge LCAd}\label{sec:2.15}

Let $A$ be a finitely generated differential algebra and let $M$ be a finitely generated left $A[\partial]$-module.
The associated \emph{gauge LCAd} $\mc G(A,M)$ is defined as follows.
As a vector space $\mc G(A,M)\subset\CEnd(M)\oplus\CDer(A)$ consists of all pairs $(\phi,\sigma)$
satisfying (cf. \eqref{eq:lcder})
\begin{equation}\label{eq:derM}
\phi_\lambda(a(\partial)u)
=
\sigma_\lambda(a(x)) \big(\big|_{x=\partial}u\big)
+
a(\lambda+x)\big(\big|_{x=\partial}\phi_\lambda(u)\big)
\,.
\end{equation}
The left $A[\partial]$-module structure of $\mc G(A,M)$ is defined as
$a(\partial)(\phi,\sigma)=(a(\partial)\phi,a(\partial)\sigma)$, where (cf. \eqref{eq:lcder-mod})
\begin{equation}\label{eq:gauge-Ad}
(a(\partial)\phi)_\lambda(u)
=
\big(\big|_{x=\partial}
a^*(\lambda)
\big)
\phi_{\lambda+x}(u)
\,\,,\quad
(a(\partial)\sigma)_\lambda(b)
=
\big(\big|_{x=\partial}
a^*(\lambda)
\big)
\sigma_{\lambda+x}(b)
\,.
\end{equation}
The LCA $\lambda$-bracket on  $\mc G(A,M)$ is defined as
\begin{equation}\label{eq:gauge-lambda}
[(\phi,\sigma)_\lambda(\psi,\tau)]
=
\sum_{i=0}^\infty \big(\phi_{(i)}\psi,\sigma_{(i)}\tau\big)
\frac{\lambda^i}{i!}\,,
\end{equation}
where (cf. \eqref{eq:lcder-lambda})
\begin{align}
& \sum_{i=0}^\infty (\phi_{(i)}\psi)_\mu(v)\frac{\lambda^i}{i!}
=
[\phi_\lambda\psi]_\mu(v)
=
\phi_\lambda(\psi_{\mu-\lambda}(v))-\psi_{\mu-\lambda}(\phi_\lambda(v))
\,,\label{eq:gauge-lambda1}\\
& \sum_{i=0}^\infty (\sigma_{(i)}\tau)_\mu(a)\frac{\lambda^i}{i!}
=
[\sigma_\lambda\tau]_\mu(a)
=
\sigma_\lambda(\tau_{\mu-\lambda}(a))-\tau_{\mu-\lambda}(\sigma_\lambda(a))
\,.\label{eq:gauge-lambda2}
\end{align}
Finally, the anchor map $\theta:\,\mc G(A,M)\to\CDer(A)$ is the projection on the second factor:
\begin{equation}\label{eq:gauge-anchor}
\theta(\phi,\sigma)=\sigma
\,.
\end{equation}

\begin{proposition}\label{prop:gauge}
\begin{enumerate}[(a)]
\item
The space $\mc G(A,M)\subset\CEnd(M)\oplus\CDer(A)$ defined by \eqref{eq:derM}
with the $A[\partial]$-action \eqref{eq:gauge-Ad}, the $\lambda$-bracket \eqref{eq:gauge-lambda}
and the anchor map \eqref{eq:gauge-anchor} is a Lie conformal algebroid.
\item
The anchor map $\theta:\,\mc G(A,M)\to\CDer(A)$ is a homomorphism of LCAd.
\item
The diagonal map $\Delta:\,\CDer(A)\to\mc G(A,A)$, mapping $\sigma\mapsto\Delta(\sigma)=(\sigma,\sigma)$,
is a homomorphism of LCAd.
Moreover, for $M=A$, we have $\theta\circ\Delta=\id_{\CDer(A)}$.
\end{enumerate}
\end{proposition}

Before proving that indeed $\mc G(A,M)$ is a well defined LCAd,
we describe it explicitly in a special case.
If $A=\mb F[u_i^{(n)}\,;i=1,\dots,r,\,n\geq0]$
is the algebra of differential polynomials in $r$ variables,
then we have a bijection $\CDer(A)\simeq A[\lambda]^{\oplus r}$
mapping 
$\sigma\in\CDer(A)$ to $(\sigma_\lambda(u_i))_{i=1}^r\in A[\lambda]^{\oplus r}$,
and conversely mapping
$F=(f_i(\lambda))_{i=1}^r\in A[\lambda]^{\oplus r}$
to $\sigma_F\in\CDer(A)$ given by
$$
\sigma_F(p)
=
\sum_{i=1}^r\sum_{n=0}^\infty
\frac{\partial p}{\partial u_i^{(n)}} (\lambda+\partial)^nf_i(\lambda)
\,.
$$
On the other hand, if $M=\bigoplus_{\ell=1}^kA[\partial]w_\ell$ is a free $A[\partial]$-module,
then, for fixed $\sigma\in\CDer(A)$, the space $\mc G_\sigma(M)$
of maps $\phi:\,M\to M[\lambda]$ satisfying \eqref{eq:derM} is in bijection with $M[\lambda]^{\oplus k}$,
the bijection mapping $\phi\in\mc G_\sigma(M)$
to $\big(\phi_\lambda(w_\ell)\big)_{\ell=1}^k\in M[\lambda]^{\oplus k}$,
and conversely mapping
$P=(p_\ell(\lambda))_{\ell=1}^k\in M[\lambda]^{\oplus k}$
to $\phi_P\in\mc G_\sigma(M)$ given by
$$
\phi_P(\sum_{\ell=1}^ka_\ell(\partial)w_\ell)
=
\sum_{\ell=1}^k
\sigma_\lambda(a_\ell(x)) \big(\big|_{x=\partial}w_\ell\big)
+
\sum_{\ell=1}^k
a_\ell(\lambda+x)\big(\big|_{x=\partial}p_\ell(\lambda)\big)
\,.
$$
Hence, if $A$ is an algebra of differential polynomials in $r$ variables
and $M$ is a free left $A[\partial]$-module of rank $k$,
then we have a bijection 
$$
\mc G(A,M)\simeq \Mat_{k\times k}(A[\partial,\lambda])\oplus A[\lambda]^{\oplus r}
\,.
$$

\begin{proof}[{Proof of Prop.\ref{prop:gauge}}]
First, to check that \eqref{eq:gauge-Ad} defines a left action of the associative algebra $A[\partial]$
is the same, straightforward, computation needed to show that \eqref{eq:lcder-mod} 
defines a left action of $A[\partial]$ on $\CDer(A)$.
Moreover, to get a left $A[\partial]$-module structure on $\mc G(A,M)$, 
we also need to check that the pair $(a(\partial)\phi,a(\partial)\sigma)$ 
defined in \eqref{eq:gauge-Ad} lies in $\mc G(A,M)$ whenever $(\phi,\sigma)$ does,
i.e. $a(\partial)\sigma$ is a conformal derivation,
and $a(\partial)\phi$ satisfies condition \eqref{eq:derM}.
This is also a straightforward computation which we leave to the reader.

Next, we need to check that formula \eqref{eq:gauge-lambda}
defines an LCA $\lambda$-bracket on $\mc G(A,M)$.
Before checking the LCA axioms, we need to show that the coefficient of each power of $\lambda$
in \eqref{eq:gauge-lambda} indeed lies in $\mc G(A,M)$.
In other words, 
$\sigma_{(i)}\tau$ should be in $\CDer(A)$, which we know since \eqref{eq:lcder-lambda}
is a $\lambda$-bracket on $\CDer(A)$,
and the pair $(\phi_{(i)}\psi,\sigma_{(i)}\tau)$ should satisfy \eqref{eq:derM}, meaning that
$$
(\phi_{(i)}\psi)_\mu(a(\partial)u)
=
(\sigma_{(i)}\tau)_\mu(a(x)) \big(\big|_{x=\partial}u\big)
+
a(\mu+x)\big(\big|_{x=\partial}(\phi_{(i)}\psi)_\mu(u)\big)
\,.
$$
Multiplying by $\lambda^i/i!$ and summing over $i$,
the above equation becomes
\begin{equation}\label{eq:gauge-pr1}
[\phi_\lambda\psi]_\mu(a(\partial)u)
=
[\sigma_\lambda\tau]_\mu(a(x)) \big(\big|_{x=\partial}u\big)
+
a(\mu+x)\big(\big|_{x=\partial}[\phi_\lambda\psi]_\mu(u)\big)
\,.
\end{equation}
Indeed, the left-hand side of \eqref{eq:gauge-pr1} is, by \eqref{eq:gauge-lambda1},
\begin{align*}
& [\phi_\lambda\psi]_\mu(a(\partial)u)
=
\phi_\lambda\big(\psi_{\mu-\lambda}(a(\partial)u)\big)
-
\psi_{\mu-\lambda}\big(\phi_\lambda(a(\partial)u)\big) \\
& =
\phi_\lambda\Big(
\tau_{\mu-\lambda}(a(x))\big(\big|_{x=\partial}u\big)
+
a(\mu-\lambda+x)\big(\big|_{x=\partial}\psi_{\mu-\lambda}(u)\big)
\Big) \\
& -
\psi_{\mu-\lambda}\Big(
\sigma_\lambda(a(x)) \big(\big|_{x=\partial}u\big)
+
a(\lambda+x)\big(\big|_{x=\partial}\phi_\lambda(u)\big)
\Big) \\
& =
\sigma_\lambda\big(\tau_{\mu-\lambda}(a(x))\big)\big(\big|_{x=\partial}u\big)
+
\tau_{\mu-\lambda}(a(x))
\phi_\lambda\big(\big|_{x=\partial}u\big) \\
& +
\sigma_\lambda\big(a(\mu-\lambda+x)\big)
\big(\big|_{x=\partial}\psi_{\mu-\lambda}(u)\big)
+
a(\mu-\lambda+x)
\phi_\lambda\big(\big|_{x=\partial}\psi_{\mu-\lambda}(u)\big) \\
& -
\tau_{\mu-\lambda}\big(\sigma_\lambda(a(x))\big)
\big(\big|_{x=\partial}u\big)
-
\sigma_\lambda(a(x))
\psi_{\mu-\lambda}\big(\big|_{x=\partial}u\big) \\
& -
\tau_{\mu-\lambda}\big(a(\lambda+x)\big)
\big(\big|_{x=\partial}\phi_\lambda(u)\big)
-
a(\lambda+x)
\psi_{\mu-\lambda}\big(\big|_{x=\partial}\phi_\lambda(u)\big)
\,.
\end{align*}
The second and seventh terms of the RHS above cancel out since $\phi$ is a conformal endomorphism,
as do the third and sixth terms since $\phi_{\mu-\lambda}$ is a conformal endomorphism.
We then combine the remaining four terms using \eqref{eq:gauge-lambda1} and \eqref{eq:gauge-lambda2}
to get the right-hand side of \eqref{eq:gauge-pr1}.
Moreover, the proof that the $\lambda$-bracket \eqref{eq:gauge-lambda} satisfies the LCA axioms
is the same, straightforward, computation needed to prove that the $\lambda$-bracket \eqref{eq:lcder-lambda}
defines a structure of an LCA on $\CDer(A)$.

Furthermore, comparing the $A[\partial]$-action \eqref{eq:gauge-Ad} with \eqref{eq:lcder}
and the $\lambda$-bracket \eqref{eq:gauge-lambda} with \eqref{eq:lcder-lambda},
we have that the anchor map \eqref{eq:gauge-anchor} commutes with the left action of $A[\partial]$
and with taking $\lambda$-brackets.
Hence, we only need to check that condition (i) of Definition \ref{def:LCAd} holds.
This reads
$$
[(\phi,\sigma)_\lambda a(\partial)(\psi,\tau)]
=
a(\lambda+x)\big(\big|_{x=\partial}[(\phi,\sigma)_\lambda (\psi,\tau)]\big)
+
\sigma_{\lambda}(a(x))\big(\big|_{x=\partial}(\psi,\tau)\big)
\,.
$$
Since both the $A[\partial]$-action and the $\lambda$-bracket on $\mc G(A,M)$
are defined component-wise, the above equation reduces to the following two identities
\begin{align}
& [\phi_\lambda a(\partial)\psi]_\mu(u)
=
\Big(a(\lambda+x)\big(\big|_{x=\partial}[\phi_\lambda \psi]\big)\Big)_\mu(u)
+
\Big(\sigma_{\lambda}(a(x))\big(\big|_{x=\partial}\psi\big)\Big)_\mu(u)
\,,\label{eq:gauge-pr2}\\
& [\sigma_\lambda a(\partial)\tau]_\mu(b)
=
\Big(a(\lambda+x)\big(\big|_{x=\partial}[\sigma_\lambda \tau]\big)\Big)_\mu(b)
+
\Big(\sigma_{\lambda}(a(x))\big(\big|_{x=\partial}\tau\big)\Big)_\mu(b)
\,.\label{eq:gauge-pr3}
\end{align}
The proof of equation \eqref{eq:gauge-pr3} is the same, straightforward, 
computation needed to prove that $\CDer(A)$ is an LCAd, see Example \ref{ex:CDer-LCAd},
so we omit it. 
We prove instead equation \eqref{eq:gauge-pr2}, which anyway is similar.
The left-hand side of \eqref{eq:gauge-pr2} is
\begin{equation}\label{eq:gauge-pr4}
\begin{split}
& [\phi_\lambda a(\partial)\psi]_\mu(u)
=
\phi_\lambda\big(\big(a(\partial)\psi\big)_{\mu-\lambda}(u)\big)
-
\big(a(\partial)\psi\big)_{\mu-\lambda}\big(\phi_\lambda(u)\big) \\
& =
\phi_\lambda\Big(
\big(\big|_{x=\partial}a^*(\mu-\lambda)\big)
\psi_{\mu-\lambda+x}(u)
\Big)
-
\big(\big|_{x=\partial}a^*(\mu-\lambda)\big)
\psi_{\mu-\lambda+x}\big(\phi_\lambda(u)\big) \\
& =
\sigma_\lambda\big(\big|_{x=\partial}a^*(\mu-\lambda)\big)
\psi_{\mu-\lambda+x}(u)
+
\big(\big|_{x=\partial}a^*(\mu-\lambda)\big)
\phi_\lambda\Big(
\psi_{\mu-\lambda+x}(u)
\Big) \\
& -
\big(\big|_{x=\partial}a^*(\mu-\lambda)\big)
\psi_{\mu-\lambda+x}\big(\phi_\lambda(u)\big) 
\,.
\end{split}
\end{equation}
The first term in the right-hand side of \eqref{eq:gauge-pr2} is
\begin{align*}
& \Big(a(\lambda+x)\big(\big|_{x=\partial}[\phi_\lambda \psi]\big)\Big)_\mu(u)
=
\big(\big|_{y=\partial}a^*(\mu-\lambda)\big)[\phi_\lambda \psi]_{\mu+y}(u) \\
& =
\big(\big|_{y=\partial}a^*(\mu-\lambda)\big)
\phi_\lambda \psi_{\mu+y-\lambda}(u)
-
\big(\big|_{y=\partial}a^*(\mu-\lambda)\big)
\psi_{\mu+y-\lambda} \phi_\lambda (u)
\,,
\end{align*}
which coincides with the second and third terms in the right-hand side of \eqref{eq:gauge-pr4}.
Moreover, the second term in the right-hand side of \eqref{eq:gauge-pr2} is
\begin{align*}
& \Big(\sigma_{\lambda}(a(x))\big(\big|_{x=\partial}\psi\big)\Big)_\mu(u)
=
\big(\big|_{y=\partial}\sigma_{\lambda}(a(x))\big)
(\big(\big|_{x=\partial}\psi\big)_{\mu+y}(u) \\
& =
\big(\big|_{y=\partial}\sigma_{\lambda}(a(-\mu-y))\big)
\psi_{\mu+y}(u)
=
\big(\sigma_{\lambda}(\big|_{z=\partial}a(\lambda-\mu-z))\big)
\psi_{\mu+z-\lambda}(u) \\
& =
\big(\sigma_{\lambda}(\big|_{z=\partial}a^*(\mu-\lambda))\big)
\psi_{\mu+z-\lambda}(u)
\,,
\end{align*}
which coincides with the first term in the right-hand side of \eqref{eq:gauge-pr4}.
This proves \eqref{eq:gauge-pr2} and completes the proof claim (a).
Claims (b) and (c) follow immediately by the definitions.
\end{proof}

In a ``specular'' fashion, the \emph{right gauge LCAd} is the space 
$\mc{RG}(A,M)\subset\RCEnd(M)\oplus\RCDer(A)$ consisting of all pairs $(\psi,\tau)$
satisfying (cf. \eqref{eq:rcder})
\begin{equation}\label{eq:rderM}
\psi_\lambda(a(\partial)u)
=
\tau_{\lambda+x}(a(x))\big(\big|_{x=\partial}u\big)
+
\big(\big|_{x=\partial}
a^*(\lambda)
\big)
\psi_{\lambda+x}(u)
\,.
\end{equation}
The left $A[\partial]$-module structure of $\mc{RG}(A,M)$ is defined as
$a(\partial)(\psi,\tau)=(a(\partial)\psi,a(\partial)\tau)$, where (cf. \eqref{eq:rcder-mod})
\begin{equation}\label{eq:rgauge-Ad}
(a(\partial)\psi)_\lambda(u)
=
a(\lambda+x)\big(\big|_{x=\partial}\psi_\lambda(u)\big)
\,\,,\quad
(a(\partial)\tau)_\lambda(b)
=
a(\lambda+x)\big(\big|_{x=\partial}\tau_\lambda(b)\big)
\,.
\end{equation}
The LCA $\lambda$-bracket on  $\mc{RG}(A,M)$ is defined as
$[(\psi,\tau)_\lambda(\rho,\xi)]^r
=
\big([\psi_\lambda\rho]^r,[\tau_\lambda\xi]^r\big)$
where (cf. \eqref{eq:rcder-lambda})
\begin{equation}\label{eq:rgauge-lambda}
\begin{split}
& [\psi_\lambda\rho]^r_\mu(v)
=
\psi^*_\lambda(\rho_{\mu}(v))-\rho_{\lambda+\mu}(\psi_\mu(v))
\,,\\
& [\tau_\lambda\xi]^r_\mu(a)
=
\tau^*_\lambda(\xi_{\mu}(a))-\xi_{\lambda+\mu}(\tau_\mu(a))
\,.
\end{split}
\end{equation}
Finally, the anchor map $\theta^r:\,\mc{RG}(A,M)\to\CDer(A)$ is
\begin{equation}\label{eq:rgauge-anchor}
\theta^r(\psi,\tau)=\tau^*
\,.
\end{equation}

\begin{proposition}\label{prop:rgauge}
\begin{enumerate}[(a)]
\item
The space $\mc{RG}(A,M)\subset\RCEnd(M)\oplus\RCDer(A)$ defined by \eqref{eq:rderM}
with the $A[\partial]$-action \eqref{eq:rgauge-Ad}, the $\lambda$-bracket \eqref{eq:rgauge-lambda}
and the anchor map \eqref{eq:rgauge-anchor} is a Lie conformal algebroid.
\item
There is an isomorphism of LCAd 
$\mc G(A,M)\stackrel{\sim}{\to}\mc{RG}(A,M)$ mapping $(\phi,\sigma)\mapsto(\phi^*,\sigma^*)$.
\end{enumerate}
\end{proposition}
\begin{proof}
The proof of claim (a) is the same as Proposition \ref{prop:gauge}(a).
Claim (b) is straightforward and is left as an exercise.
\end{proof}

\subsection{The quotient LAd of an LCAd}\label{sec:2.2a}

Let $R$ be a commutative associative unital algebra 
and let $\Der(R)$ be the Lie algebra of derivations of $R$.
Recall that a \emph{Lie algebroid} (LAd) over a  $R$
is an $R$-module $F$ endowed with 
a Lie algebra bracket $[\cdot\,,\,\cdot]:\,F\otimes F\to F$
and a Lie algebra homomorphism $\theta:\,F\to\Der(R)$ commuting with the action of $R$, 
called the anchor map, such that
$[u,av]=a[u,v]+\theta(u)(a)v$ for all $a\in R,\,u,v\in F$.

\begin{proposition}\label{prop:quot-LA}
Let $E$ be an LCAd over the differential algebra $A$.
Consider the commutative associative unital algebra $\bar A=A/\langle\partial A\rangle_{A}$,
where $\langle\partial A\rangle_{A}=A\partial A$ is the algebra ideal generated by $\partial A$.
Then, we have a LAd $(\bar E,[\cdot\,,\,\cdot]^{\bar{}},\bar\theta)$ over the algebra $\bar A$, 
called the \emph{quotient LAd} of $E$, defined as follows.
As an $\bar A$-module,
$$
\bar E=E/\langle\partial E\rangle_{A}\,,
\quad\text{where }
\langle\partial E\rangle_{A}=A\partial E
$$
with the $\bar A$-action induced by the $A[\partial]$-action on $E$.
The Lie bracket on $\bar E$ is induced by the $\lambda$-bracket on $E$ at $\lambda=0$:
$$
[\bar u,\bar v]=\overline{[u_\lambda v]}\big|_{\lambda=0}
\,.
$$
The anchor map on $\bar E$ is induced by the anchor map on $E$ at $\lambda=0$:
$$
\bar\theta(\bar u)(\bar a)=\overline{\theta(u)_{\lambda}(a)}\big|_{\lambda=0}
\,.
$$
\end{proposition}
\begin{proof}
The proof reduces to a sequence of straightforward verifications.
\end{proof}

\subsection{Jet algebra and the current LCAd of a LAd}\label{sec:2.2b}

Let $R$ be a commutative associative unital algebra. 
Recall that the \emph{jet algebra} $J_\infty R$ of $R$
is the unique (up to isomorphism) differential algebra
endowed with an (injective) algebra homomorphism $\iota:\,R\to J_\infty R$,
satisfying the following universal property:
for every algebra homomorphism $f:\,R\to A$ to a differential algebra $A$,
there exists a unique homomorphism of differential algebras $\tilde f:\,J_\infty R\to A$
making the following diagram commute:
$$
\xymatrix{
R \ar[r]^{f}\ar@{^(->}[d]_{\iota} & A \\
J_\infty R \ar@/^-1pc/[ur]_{\exists !\tilde f}
}
$$

For example, if we consider the algebra of polynomials in $k$ variables $R_k=\mb F[x_1,\dots,x_k]$,
its jet algebra is the algebra of differential polynomials in the same variables:
$$
J_\infty R_k=\mb F[x_j^{(n)};\,j=1,\dots,k,\,n\in\mb Z_+]
\,.
$$
It is a differential algebra with derivation $\partial:\,J_\infty R_k\to J_\infty R_k$
defined on generators as $\partial x_j^{(n)}=x_j^{(n+1)}$, $j=1,\dots,k,\,n\in\mb Z_+$,
and extended to $J_\infty R_k$ by the Leibniz rule.
The injective algebra homomorphism $\iota:\,R_k\hookrightarrow J_\infty R_k$ is defined
on generators by $\iota(x_j)=x_j^{(0)}$, $j=1,\dots,k$.
We also have a canonical surjective algebra homomorphism 
$\chi:\,J_\infty R_k\twoheadrightarrow R_k$
defined on generators by $\chi(x_j^{(n)})=\delta_{n,0}x_j$, $j=1,\dots,k,\,n\geq0$,
and clearly $\iota\circ\chi=\id_{R_k}$.

More generally, 
if the algebra $R$ is finitely generated, then it can be realised as the quotient
of an algebra of polynomials in finitely many variables, say $R=R_k/I$.
In this case the jet algebra can be constructed explicitly as the quotient
$$
J_\infty R=J_\infty R_k\big/\langle I\rangle_\partial
\,,
$$
where $\langle I\rangle_\partial$ is the differential ideal of $J_\infty R_k$ generated by $I$:
$$
\langle I\rangle_\partial
=
\sum_{n=0}^\infty (J_\infty R_k)\partial^nI
\,.
$$
The derivation $\partial:\, J_\infty R\to J_\infty R$, 
the injective algebra homomorphism $\iota:\,R\hookrightarrow J_\infty R$, 
and the surjective algebra homomorphism $\chi:\,J_\infty R\to R$,
are induced by the corresponding maps on the algebras of polynomials,
which factor through the quotients. In particular, $\iota\circ\chi=\id_R$.

\begin{lemma}\label{lem:jet-der}
Let $R$ be a finitely generated commutative associative unital algebra.
For every derivation $D\in\Der(R)$ there exists a unique conformal derivation 
$\widehat D\in\CDer(J_\infty R)$
making the following diagram commute:
\begin{equation}\label{eq:hatD}
\xymatrix{
R \ar[r]^{D}\ar@{^(->}[d]_{\iota} & R\ar@{^(->}[d]_{\iota} \\
J_\infty R \ar[r]_{\widehat D_\lambda} & J_\infty R[\lambda]
}
\end{equation}
\end{lemma}
\begin{proof}
The conformal derivation $\widehat{D}_\lambda$ is explicitly defined as follows.
Let $[x_j]_I\in R,\,j=1,\dots,k$, be the classes of $x_j\in R_k$ in the quotient algebra,
let $D([x_j]_I)=[f_j]_I$ be the corresponding images via the derivation $D\in\Der(R)$,
and fix representatives $f_j\in R_k$.
If $p\in J_\infty R_k$ and if $[p]_{\langle I\rangle_\partial}$ is its class in $J_\infty R$,
then
$$
\widehat{D}_\lambda([p]_{\langle I\rangle_\partial})
=
\Big[
\sum_{j=1}^k\sum_{n=0}^\infty\frac{\partial p}{\partial x_j^{(n)}}(\lambda+\partial)^n f_j
\Big]_{\langle I\rangle_\partial}
\,.
$$
It is easy to check that $\widehat{D}_\lambda$ is well defined by the above formula,
that it is indeed a conformal derivation of $J_\infty R$,
and that $\widehat{D}_\lambda\circ\iota=\iota\circ D$.
The uniqueness of such conformal derivation is obvious.
\end{proof}
\begin{proposition}\label{prop:cur}
Let $(F,[\cdot\,,\,\cdot],\theta)$ be a LAd over a finitely generated commutative associative unital algebra $R$.
We have an LCAd $(\widehat F,\partial,[\cdot\,_\lambda\,\cdot]^{\widehat{}},\widehat{\theta})$ 
over the differential algebra $\widehat R=J_\infty R$, called the \emph{current LCAd} of $F$, 
defined as follows.
As a $\widehat R[\partial]$-module, 
$$
\widehat F=\widehat R[\partial]\otimes_R F\,.
$$
The $\lambda$-bracket on $\widehat F$ is uniquely defined by letting
\begin{equation}\label{eq:cur-lambda}
[1\otimes_Ru \,_\lambda\,1\otimes_Rv]^{\widehat{}}
=
1\otimes_R[u,v]
\,,\,\,\text{ for } u,v\in F
\end{equation}
and extended to $\widehat F$ by the total formula \eqref{eq:tot-form}.
The anchor map on $\widehat F$ is uniquely defined, thanks to Lemma \ref{lem:jet-der}, by letting
\begin{equation}\label{eq:cur-anchor}
\widehat{\theta}(1\otimes_Ru)_\lambda
=
\widehat{\theta(u)}_{\lambda}
\,,\,\,\text{ for } u\in F
\,,
\end{equation}
and extended to $\widehat F$ to be a homomorphism of $\widehat A[\partial]$-modules,
using \eqref{eq:lcder-mod}.
Furthermore, the quotient LAd of the current LCAd of $F$ is canonically isomorphic to the LAd $F$.
\end{proposition}
\begin{proof}
By \eqref{eq:lcder-mod} the anchor map should be defined on the whole space $\widehat F$ 
by the following formula
($a(\partial)\in J_\infty R[\partial],\,u\in F,\,b\in J_\infty R$):
\begin{equation}\label{eq:cur-anchor-tot}
\widehat\theta(a(\partial)\otimes_R u)_\lambda(b)
=
\big(\big|_{x=\partial}a^*(\lambda)\big)\widehat{\theta(u)}_{\lambda+x}(b)
\,.
\end{equation}
We need to check that it is indeed well defined 
on the tensor product space $\widehat F=J_\infty R[\partial]\otimes_R F$,
that it defines a conformal derivation of $\widehat R=J_\infty R$,
and that the map $\widehat\theta$ so defined commutes with the left action of $J_\infty R[\partial]$.
To show that it is well defined, we need to show that the formulas that we get for
$\widehat\theta(a(\partial)\circ r\otimes_R u)_\lambda(b)$
and $\widehat\theta(a(\partial)\otimes_R ru)_\lambda(b)$ are the same for every $r\in R$.
Indeed, by \eqref{eq:cur-anchor-tot}, we have
\begin{equation}\label{eq:pr29.1}
\widehat\theta(a(\partial)\circ r\otimes_R u)_\lambda(b)
=
\big(\big|_{x=\partial}ra^*(\lambda)\big)\widehat{\theta(u)}_{\lambda+x}(b)
\,,
\end{equation}
and
\begin{equation}\label{eq:pr29.2}
\widehat\theta(a(\partial)\otimes_R ru)_\lambda(b)
=
\big(\big|_{x=\partial}a^*(\lambda)\big)\widehat{\theta(ru)}_{\lambda+x}(b)
\,.
\end{equation}
On the other hand, since $\theta$ commutes with the action of $R$,
and by the universal property \eqref{eq:hatD} of $\widehat{r\theta(u)}$,
we have that 
\begin{equation}\label{eq:pr29.5}
\widehat{\theta(ru)}_\lambda
=
\widehat{r\theta(u)}_\lambda
=
\big(\big|_{y=\partial}r\big)\widehat{\theta(u)}_{\lambda+y}
\,.
\end{equation}
(Indeed, the right-hand side above is a conformal derivation of $J_\infty R$ such that
$\big(\big|_{y=\partial}r\big)\widehat{\theta(u)}_{\lambda+y}\circ\iota=\iota\circ(r\theta(u))$,
hence it equals the left-hand side by uniqueness.)
Using \eqref{eq:pr29.5}, we immediately get that the right-hand sides 
of \eqref{eq:pr29.1} and \eqref{eq:pr29.2} coincide.
Since $\widehat{\theta(u)}_{\lambda}$ is a conformal derivation of $J_\infty R$,
so is the map $\widehat\theta(a(\partial)\circ r\otimes_R u)_\lambda$ defined by \eqref{eq:cur-anchor-tot}.
Furthermore, $\widehat\theta(a(\partial)\circ r\otimes_R u)$ commutes with the left action
of $J_\infty R[\partial]$ by construction.

Next, by \eqref{eq:tot-form} the $\lambda$-bracket should be defined on the whole space $\widehat F$ 
by the following formula
($a(\partial),b(\partial)\in J_\infty R[\partial],\,u,v\in F$):
\begin{equation}\label{eq:cur-lambda-tot}
\begin{split}
& [a(\partial)\otimes_Ru _\lambda b(\partial)\otimes_Rv]^{\widehat{}}
=
\big(\big|_{x=\partial}a^*(\lambda)\big)b(\lambda+x+y)\big|_{y=\partial}\otimes_R[u,v] \\
&\quad +
\big(\big|_{x=\partial}a^*(\lambda)\big)\widehat{\theta(u)}_{\lambda+x}(b(y))\big|_{y=\partial}\otimes_Rv \\
&\quad -
b(\lambda+x+y)
\big(\big|_{x=\partial}\widehat{\theta(v)}^*_{\lambda+y}(a(y))\big)\big|_{y=\partial}\otimes_Ru
\,.
\end{split}
\end{equation}
We need to check that it is well defined 
on the tensor product space $\widehat F=J_\infty R[\partial]\otimes_R F$.
In order to prove well definiteness,
we need to check that
\begin{equation}\label{eq:pr29.3}
[a(\partial)\otimes_Rru \,_\lambda\, b(\partial)\otimes_Rsv]^{\widehat{}}
=
[a(\partial)\circ r\otimes_Ru \,_\lambda\, b(\partial)\circ s\otimes_Rv]^{\widehat{}}
\,,
\end{equation}
for $a(\partial),b(\partial)\in J_\infty R[\partial]$, $u,v\in F$ and $r,s\in R$.
By \eqref{eq:cur-lambda-tot}, the right-hand side of \eqref{eq:pr29.3} is
\begin{equation}\label{eq:pr29.4}
\begin{split}
& \big(\big|_{x=\partial}ra^*(\lambda)\big)b(\lambda+x+y+z)
\big(\big|_{z=\partial}s\big)
\big|_{y=\partial}\otimes_R[u,v] \\
& +
\big(\big|_{x=\partial}ra^*(\lambda)\big)\widehat{\theta(u)}_{\lambda+x}
(b(y+z)\big(\big|_{z=\partial}s\big))\big|_{y=\partial}\otimes_Rv \\
& -
b(\lambda+x+y+w)
\big(\big|_{w=\partial}s\big)
\big(\big|_{x=\partial}\widehat{\theta(v)}^*_{\lambda+y}\big(a(y+z)\big(\big|_{z=\partial}r\big)\big)\big)
\big|_{y=\partial}\otimes_Ru
\,.
\end{split}
\end{equation}
By the definition of tensor product, the first term of \eqref{eq:pr29.4} is
\begin{equation}\label{eq:pr29.6}
\big(\big|_{x=\partial}a^*(\lambda)\big)b(\lambda+x+y+z)
\big|_{y=\partial}\otimes_Rrs[u,v]\,.
\end{equation}
By equation \eqref{eq:pr29.5} and the definition of conformal derivation,
the second term of \eqref{eq:pr29.4} is equal to
\begin{equation}\label{eq:pr29.7}
\begin{split}
& \big(\big|_{x=\partial}a^*(\lambda)\big)\widehat{\theta(ru)}_{\lambda+x}
(b(y))\big|_{y=\partial}\otimes_R sv \\
& + \big(\big|_{x=\partial}a^*(\lambda)\big)
b(\lambda+x+y)
\big|_{y=\partial}
\otimes_R
\theta(ru)(s) v
\,.
\end{split}
\end{equation}
Similarly, we can use the identity 
$s\widehat{\theta(v)}^*_{\lambda}=\widehat{\theta(sv)}^*_{\lambda}$,
which follows easily from \eqref{eq:pr29.5},
and the definition of right conformal derivations,
to rewrite the third term of \eqref{eq:pr29.4} as
\begin{equation}\label{eq:pr29.8}
\begin{split}
& -
b(\lambda+x+y)
\big(\big|_{x=\partial}
\widehat{\theta(sv)}^*_{\lambda+y}(a(y))
\big)
\big|_{y=\partial}\otimes_R ru \\
& -
b(\lambda+x+y)
\big(\big|_{x=\partial}
a^*(\lambda)
\big)
\big|_{y=\partial}
\otimes_R \theta(sv)(r) u
\,.
\end{split}
\end{equation}
Combining \eqref{eq:pr29.6}, \eqref{eq:pr29.7} and \eqref{eq:pr29.8},
we get the  left-hand side of \eqref{eq:pr29.3}, as claimed.

Next, we check that 
the anchor map $\widehat{\theta}$ defined by \eqref{eq:cur-anchor-tot}
and the $\lambda$-bracket defined by \eqref{eq:cur-lambda-tot}
satisfy condition (i) and (i') of Definition \ref{def:LCAd}.
For condition (i), we have
\begin{equation}\label{eq:pr30.1}
\begin{split}
& [a(\partial)\otimes_Ru \,_\lambda\, c(\partial)(b(\partial)\otimes_Rv)]^{\widehat{}}
=
c(\lambda+\partial)\Big(
\big(\big|_{x=\partial}a^*(\lambda)\big)b(\lambda+x+y)\big|_{y=\partial}\otimes_R[u,v]
\Big) \\
&\quad +
\big(\big|_{x=\partial}a^*(\lambda)\big)\widehat{\theta(u)}_{\lambda+x}\big(c(y+z)\big(\big|_{z=\partial}b(y)\big)\big)\big|_{y=\partial}\otimes_Rv \\
&\quad -
c(\lambda+\partial)\Big(
b(\lambda+x+y)
\big(\big|_{x=\partial}\widehat{\theta(v)}^*_{\lambda+y}(a(y))\big)\big|_{y=\partial}\otimes_Ru
\Big)
\,.
\end{split}
\end{equation}
By the definition of conformal derivations, we have
\begin{align*}
& \widehat{\theta(u)}_{\lambda+x}\big(c(y+z)\big(\big|_{z=\partial}b(y)\big)\big)
=
\widehat{\theta(u)}_{\lambda+x}\big(c(y+z)\big)
\big(\big|_{z=\partial}b(y)\big)\big) \\
&\qquad +
c(\lambda+x+y+z)
\big(\big|_{z=\partial}
\widehat{\theta(u)}_{\lambda+x}(b(y))
\big)
\,.
\end{align*}
Hence, the right-hand side of \eqref{eq:pr30.1} becomes
$$
c(\lambda+\partial)
[a(\partial)\otimes_Ru\,_\lambda\, b(\partial)\otimes_Rv]^{\widehat{}}
+
\widehat{\theta}(a(\partial)\otimes_R u)_{\lambda}(c(y))
\big(\big|_{y=\partial}
(b(\partial)\otimes_Rv)
\big)
\,,
$$
proving condition (i).
Condition (i') is proved in a similar way.

We next observe that, since $F$ is a LAd,  the LCA axioms and condition (ii) of Definition \ref{def:LCAd}
hold for $[\cdot\,_\lambda\,\cdot]^{\widehat{}}$ and $\widehat{\theta}$ on the set $\{1\otimes_Ru\}_{u\in F}$,
which generate $\widehat F$ as a left $\widehat{A}[\partial]$-module.
Hence, by Lemma \ref{lem:gener}, $(\widehat{F},\partial,[\cdot\,_\lambda\,\cdot]^{\widehat{}},\widehat{\theta})$ 
is indeed an LCAd.

We finally prove the last assertion of the proposition.
If $R=R_k/I$, then $\widehat{R}=J_\infty R\simeq J_\infty R_k/\langle I\rangle_\partial$, so that
\begin{align*}
& \overline{\widehat{R}}
=
\widehat{R}/(\widehat{R}\partial\widehat{R})
\simeq
J_\infty R_k
\big/
\Big(
\sum_{\ell=0}^\infty (J_\infty R_k)\partial^\ell I
+
(J_\infty R_k)\partial(J_\infty R_k)\Big) \\
& \simeq
\big(J_\infty R_k\big/\big((J_\infty R_k)\partial(J_\infty R_k)\big)\big)
\big/
\Big((J_\infty R_k) I\big/\big((J_\infty R_k)\partial(J_\infty R_k)\big)\Big)
\simeq
R_k/I=R\,.
\end{align*}
Similarly, for the quotient of $\widehat{F}$, we have
\begin{align*}
& \overline{\widehat{F}}
=
\widehat{F}/(\widehat{R}\partial\widehat{F})
\simeq
(J_\infty R[\partial]\otimes_RF)
\big/
\big((J_\infty R)\partial\big(J_\infty R[\partial]\otimes_RF\big)\big) \\
& \simeq
(J_\infty R\otimes_RF)
\big/
\big(((J_\infty R)\partial(J_\infty R)) \otimes_R F\big) 
\simeq
R\otimes_RF\simeq F
\,.
\end{align*}
The induced Lie bracket and anchor map on $\overline{\widehat{F}}$ obviously correspond,
under these isomorphisms, to the Lie bracket and anchor map of $F$.
\end{proof}

\subsection{Extended LCAd}\label{sec:2.3}

Recalling that an LCA is the same as an LCAd over the base field $\mb F$ with the zero anchor map $\theta=0$,
we generalize here the construction of Example \ref{ex:transf-LCAd}.

\begin{proposition}\label{prop:extension}
Let $E$ be an LCAd over the differential algebra $A$,
let $B\supset A$ be a differential algebra extension of $A$,
and let $\phi:\,E\to\CDer(B)$ be an LCA homomorphism extending the anchor map,
in the sense that $\phi(u)\big|_{A}=\theta(u)$ for every $u\in E$.
Then we have an LCAd, 
called the $B$-\emph{extended LCAd}, defined as the space $E_B=B\otimes_A E$,
with the left $B[\partial]$-module structure given by 
($a(\partial)\in B[\partial],\,b\in B,\,u\in E$):
\begin{equation}\label{eq:ext1}
a(\partial)(b\otimes_A u)=\big(a(x+y)\big(\big|_{x=\partial}b\big)\big)\otimes_A \big(\big|_{y=\partial}u\big)
\,,
\end{equation}
the LCA $\lambda$-bracket given by
($a,b\in B,\,u,v\in E$):
\begin{equation}\label{eq:ext2}
\begin{split}
& [a\otimes_A u_\lambda b\otimes_A v]_B
=
\big(\big|_{x=\partial}a\big)b\otimes_A[u_{\lambda+x} v] \\
& +
\big(\big|_{x=\partial}a\big)\phi(u)_{\lambda+x}(b)\otimes_A v
-
b\phi(v)^*_{\lambda+x}(a)\otimes_A \big(\big|_{x=\partial}u\big)
\,,
\end{split}
\end{equation}
and the anchor map $\theta_B:\,B\otimes_A E\to\CDer(B)$ is given by ($a,b\in B,\,u\in E$):
\begin{equation}\label{eq:ext3}
\theta_B(a\otimes_A u)_\lambda(b)
=
\big(\big|_{x=\partial}a\big)\phi(u)_{\lambda+x}(b)
\,.
\end{equation}
\end{proposition}
\begin{proof}
To prove that the $B[\partial]$-action \eqref{eq:ext1}, the $\lambda$-bracket \eqref{eq:ext2}
and the anchor map \eqref{eq:ext3} are well defined on the tensor product space $B\otimes_AE$
is a straightforward computation.
The LCAd axioms hold, by construction, on the space of elements $\{1\otimes_A u\}_{u\in E}$.
Hence, by Lemma \ref{lem:gener}, they hold on the whole $E_B$,
provided that conditions (i) and (i') hold, which can be checked directly.
\end{proof}

Let $E$ be an LCAd over $A$.
Consider the differential algebra extension $\widetilde{A}=A[t^{\pm1}]$
with derivation $\widetilde{\partial}$ given by 
$\widetilde{\partial}(at^n)=(\partial a)t^n+nat^{n-1}$,
and the LCA homomorphism $\phi:\,E\to\CDer(\widetilde A)$ given by 
$\phi(u)_\lambda(at^m)=\phi(u)_\lambda(a) t^m$.
The \emph{affinization LCAd} of $E$ is the corresponding $\widetilde{A}$-extended LCAd
$\widetilde{E}=\widetilde{A}\otimes_AE\simeq E[t^{\pm1}]$.

\subsection{Representations of LCAd}\label{sec:2.5}

\begin{definition}\label{def:module}
Let $E$ be an LCAd over the differential algebra $A$.
A \emph{module} over $E$ is a left $A[\partial]$-module $M$
endowed with a $\lambda$-\emph{action} of $E$ on $M$, i.e. a map
\begin{equation}\label{eq:action}
E\times M\to M[\lambda]
\,,\qquad (u,m)\mapsto u_\lambda m\,,
\end{equation}
satisfying the following conditions ($a(\partial)\in A[\partial],\,u,v\in E,\,m\in M$):
\begin{enumerate}[(i)]
\item
$u_\lambda(a(\partial)m) = a(\lambda+x) \big(\big|_{x=\partial}u_\lambda m\big)
+\theta(u)_\lambda(a(x))\big(\big|_{x=\partial}m\big)$;
\item
$(a(\partial)u)_\lambda m=\big(\big|_{x=\partial}a^*(\lambda)\big)(u_{\lambda+x}m)$;
\item
$[u_\lambda v]_\mu(m)=u_\lambda(v_{\mu-\lambda}m)-v_{\mu-\lambda}(u_\lambda m)$.
\end{enumerate}
\end{definition}
\begin{proposition}\label{prop:LCAd-mod}
Let $E$ be an LCAd over the differential algebra $A$,
and let $M$ be a left module over the associative algebra $A[\partial]$.
Assume that $A$ is finitely generated as a differential algebra,
and that $M$ is finitely generated as an $A[\partial]$-module.
Then, 
there is a bijective correspondence between:
\begin{enumerate}[(a)]
\item
the $\lambda$-actions of $E$ on $M$ making $M$ an LCAd-module over $E$;
\item
the LCAd homomorphisms $\rho:\,E\to\mc G(A,M)$, mapping $u\mapsto\rho(u)=(\phi(u),\theta(u))$,
with $\phi(u)\in\CEnd(M)$, where $\theta$ is the anchor map.
\end{enumerate}
This correspondence associates to an LCAd $\lambda$-action $u_\lambda m$ of $E$ on $M$,
the map $\rho=(\phi,\theta)$, where $\phi$ is given by (for $u\in E$ and $m\in M$):
\begin{equation}\label{eq:action2}
\phi(u)_\lambda(m)= u_\lambda m\,.
\end{equation}
\end{proposition}
\begin{proof}
In terms of the map $\phi$ in \eqref{eq:action2}, 
condition (i) of Definition \ref{def:module} with $a(\partial)=\partial$ guarantees that $\phi(u)$ lies in $\CEnd(M)$,
while the whole condition (i) can be expressed by saying that the pair $(\phi(u),\theta(u))$ lies in $\mc G(A,M)$
i.e. satisfies \eqref{eq:derM};
condition (ii) is equivalent to say that $\phi:\,E\to\CEnd(M)$, hence $\rho=(\phi,\theta):\,E\to\mc G(A,M)$, 
is a homomorphism of $A[\partial]$-modules,
where the $A[\partial]$-module structure of $\CEnd(M)$ is given by equation \eqref{eq:lcder-mod};
moreover, condition (iii) can be expressed by saying that $\phi:\,E\to\CEnd(M)$
maps the $\lambda$-bracket of $E$ to the $\lambda$-bracket \eqref{eq:lcder-lambda} of $\CEnd(M)$. 
Hence $\rho=(\phi,\theta):\,E\to\mc G(A,M)$ maps the $\lambda$-bracket of $E$ 
to the $\lambda$-bracket \eqref{eq:gauge-lambda} of $\mc G(A,M)$. 
In particular,
in the case when $A$ is a finitely generated differential algebra 
and $M$ is a finitely generated left $A[\partial]$-module,
$(\rho,\theta)$ is a homomorphism of LCAd $(\rho,\theta):\,E\to\mc G(A,M)$.
Conversely, given an LCAd homomorphism $\rho:\,E\to\mc G(A,M)$ of the form $\rho=(\phi,\theta)$,
then, by the same arguments as above, we get an LCAd $\lambda$-action of $E$ on $M$ given by formula \eqref{eq:action2}
(read from right to left).
\end{proof}
\begin{remark}\label{rem:LCAd-mod}
If we remove the finitely generatedness assumptions on $A$ and $M$, the same statement still holds,
except that $\mc G(A,M)$ is not, strictly speaking, an LCA as the $\lambda$-bracket may have values in power series.
Hence the map $\rho$, which maps the $\lambda$-bracket of $E$ to the $\lambda$-bracket of $\mc G(A,M)$,
will not be, strictly speaking, an LCA homomorphism.
\end{remark}

\begin{example}\label{ex:trivial-rep}
Given an LCAd $E$ over the differential algebra $A$,
then $A$ is automatically a module over $E$, called the \emph{trivial module},
with the $\lambda$-action
$u_\lambda a=\theta(u)_\lambda(a)$, for $a\in A,\,u\in E$.
\end{example}
\begin{remark}
The analogue of the adjoint representation is missing for LCAd:
the $\lambda$-bracket $[u_\lambda v]$, considered as a $\lambda$-action $E\times E\to E[\lambda]$,
does not satisfy condition (ii).
The obstruction to this is the extra term 
$-\theta(v)^*_{\lambda+x}(a(x))\big(\big|_{x=\partial}u\big)$
of condition (i') in Definition \ref{def:LCAd}.
In fact, $\ker(\theta)\subset E$ is a module over the LCAd $E$.
\end{remark}
\begin{example}\label{ex:dual-mod}
If $M$ is a left $A[\partial]$-module, the dual $A[\partial]$-module $M^{\dagger}$ is defined as the space
\begin{equation}\label{eq:Mstar}
M^\dagger=\Bigg\{\phi_\lambda:\,M\to A[\lambda]
\,\Bigg|\,
\begin{array}{r}
\vphantom{\Big(}
\phi_\lambda(a(\partial)m)=a(\lambda+x)\big(\big|_{x=\partial}\phi_{\lambda}(m)\big) \\
\vphantom{\Big(} 
\text{ for } a(\partial)\in A[\partial],\,m\in M
\end{array}
\Bigg\}
\,,
\end{equation}
with the left $A[\partial]$-action given by ($a(\partial)\in A[\partial],\,m\in M$):
\begin{equation}\label{eq:Mstar1}
(a(\partial)\phi)_\lambda(m)
=
\big(\big|_{x=\partial}a^*(\lambda)\big)\phi_{\lambda+x}(m)
\,.
\end{equation}
It is immediate to check that \eqref{eq:Mstar1} is a well defined left action of $A[\partial]$
on the space $M^\dagger$ in \eqref{eq:Mstar}.
Moreover, if $M$ is a module over the LCAd $E$,
then $M^\dagger$ has a natural structure of a module over $E$, given by
\begin{equation}\label{eq:Mstar2}
(\rho^\dagger(u)_\lambda(\phi))_\mu(m)
=
\theta(u)_\lambda(\phi_{\mu-\lambda}(m))-\phi_{\mu-\lambda}(\rho(u)_\lambda(m))
\,.
\end{equation}
Indeed, it is straightforward to check that this formula gives a well defined action of $E$ on $M^\dagger$.
\end{example}
\begin{proposition}\label{prop:semi}
Let $E$ be an LCAd 
over the differential algebra $A$ and let $M$ be a module over $E$.
Then, we have an LCAd structure on $E\oplus M=\{u+m\,|\,u\in E,\,m\in M\}$, 
called the \emph{semidirect product} of $E$ and $M$,
with the left $A[\partial]$-module structure 
\begin{equation}\label{eq:semi1}
a(\partial)(u+m)=a(\partial)u+a(\partial)m
\,,
\end{equation}
the anchor map given by 
\begin{equation}\label{eq:semi2}
\theta_{E\oplus M}(u+m)=\theta(u)
\,,
\end{equation}
and the $\lambda$-bracket
\begin{equation}\label{eq:semi3}
[u+m_\lambda v+n]_{E\oplus M}
=
[u_\lambda v]
+
\big(
u_\lambda n
-
\big|_{x=\partial}v_{-\lambda-x}m
\big)
\,.
\end{equation}
\end{proposition}
\begin{proof}
Equation \eqref{eq:semi1} obviously defines a left action of $A[\partial]$ on $E\oplus M$,
likewise equation \eqref{eq:semi2} obviously defines a homomorphism 
of left $A[\partial]$-modules $E\oplus M\to\CDer(A)$.
Moreover, the $\lambda$-bracket \eqref{eq:semi3} is obviously sesquilinear and skewsymmetric.
The Jacobi identity of the $\lambda$-bracket
follows by a straightforward computation, using the Jacobi identity of the $\lambda$-bracket of $E$
and axiom (iii) of the Definition \ref{def:module} of the $\lambda$-action of $E$ on $M$.
It is the same computation needed for the construction of the semidirect product of LCA,
so we leave it to the reader.
Next, we prove condition (i) of Definition \ref{def:LCAd}.
We have,
\begin{align*}
& [u+m_\lambda a(\partial)(v+n)]
=
[u_\lambda(a(\partial)v]
+
u_\lambda(a(\partial)n)
-
\big(\big|_{x=\partial}(a(\partial)v)_{-\lambda-x}m\big) \\
& =
a(\lambda+x)\big(\big|_{x=\partial}
[u_\lambda v]+u_\lambda n-v_{-\lambda-x}m
\big)
+
\theta(u)_\lambda(a(x))
\big(\big|_{x=\partial}v+n\big) \\
& =
a(\lambda+x)\big(\big|_{x=\partial}
[u+m_\lambda v_n]_{E\oplus M}
\big)
+
\theta_{E\oplus M}(u+m)_\lambda(a(x))\big(\big|_{x=\partial}v+n\big)
\,.
\end{align*}
Finally, condition (ii) of Definition \ref{def:LCAd} is obvious, since $\theta_{E\oplus M}$
acts trivially on $M$ and coincides with $\theta$ on $E$.
\end{proof}
\begin{example}
As a special case, we can consider the trivial module $A$ over $E$
and apply Proposition \ref{prop:semi} $n$-times to get, for every $n\geq0$,
an LCAd structure on $E\oplus A^{\oplus n}$ with the left action of $A[\partial]$ given component-wise,
the anchor map coinciding with the anchor map of $E$ and acting trivially on $A^{\oplus n}$,
and the $\lambda$-bracket
$$
\big[u+(a_i)_{i=1}^n \,_\lambda v+(b_j)_{j=1}^n\big]
=
[u_\lambda v]
+
\big(
u_\lambda b_j
\big)_{j=1}^n
-
\big(\big|_{x=\partial}v_{-\lambda-x}a_i
\big)_{i=1}^n
\,.
$$
\end{example}

\section{Lie conformal algebroid cohomology complexes}\label{sec:3}

\subsection{Cohomology complex of an LCAd}\label{sec:3.1}

Recall the definition of LCA cohomology complex in \cite[Sec.2]{DK09}.
Recalling that an LCA is the same as an LCAd over the trivial differential algebra $A=\mb F$ 
with zero anchor map, we generalize the construction in \cite{DK09} of the LCA cohomology complex
to define the LCAd cohomology complex.

First, let $E$ and $M$ be left $A[\partial]$-modules.
The space of $k$-\emph{cochains} $C^k(E,M)$, $k\geq0$,
is the space of maps 
\begin{equation}\label{eq:k-cochains}
\varphi_{\lambda_1,\dots,\lambda_k}:\,E^{\otimes k}\to M[\lambda_1,\dots,\lambda_k]/\langle\partial+\lambda_1+\dots+\lambda_k\rangle
\,,
\end{equation}
mapping $u_1\otimes\dots\otimes u_k\mapsto\varphi_{\lambda_1,\dots,\lambda_k}(u_1,\dots,u_k)$,
satisfying
\begin{equation}\label{eq:poly-lambda}
\begin{split}
& \varphi_{\lambda_1,\dots,\lambda_k}\big(a_1(\partial)u_1,\dots,a_k(\partial)u_k\big) \\
& =
\big(\big|_{x_1=\partial}a_1^*(\lambda_1)\big)
\dots
\big(\big|_{x_{k}=\partial}a_{k}^*(\lambda_{k})\big)
\varphi_{\lambda_1+x_1,\dots,\lambda_k+x_k}\big(u_1,\dots,u_k)
\,,
\end{split}
\end{equation}
and
\begin{equation}\label{eq:poly-skew}
\varphi_{\lambda_1,\dots,\lambda_k}\big(u_1,\dots,u_k)
=
\sign(\pi)
\varphi_{\lambda_{\pi(1)},\dots,\lambda_{\pi(k)}}(u_{\pi(1)},\dots,u_{\pi(k)})
\,,
\end{equation}
for every permutation $\pi\in S_k$,
where $\sign(\pi)$ is the sign of the permutation.
In particular, $C^0(E,M)\simeq M/\partial M$.
For $k\geq1$
we can identify $M[\lambda_1,\dots,\lambda_k]/\langle\partial+\lambda_1+\dots+\lambda_k\rangle$ 
with $M[\lambda_1,\dots,\lambda_{k-1}]$ 
by setting $\lambda_k=-\lambda_1-\dots-\lambda_{k-1}-\partial$,
with $\partial$ acting on coefficients.
Hence, $C^k(E,M)$ is identified with the space of $k$-$\lambda$-brackets
$E^{\otimes k}\to M[\lambda_1,\dots,\lambda_{k-1}]$,
satisfying the analogue of conditions \eqref{eq:poly-lambda} and \eqref{eq:poly-skew},
cf. \cite{DK09}.
In particular, $C^1(E,M)\simeq \Hom_{A[\partial]}(E,M)$.

Next, if $E$ is an LCAd over the differential algebra $A$
and $M$ is a module over $E$,
we make $C^\bullet(E,M)=\oplus_{k\geq0}C^k(E,M)$
into a cohomology complex, with the differential 
$d:\,C^k(E,M)\to C^{k+1}(E,M),\,k\geq0$, defined by
(cf. \cite[Eq.(15)]{DK09})
\begin{equation}\label{eq:differential}
\begin{split}
& 
(d\varphi)_{\lambda_1,\dots,\lambda_{k+1}}(u_1,\dots,u_{k+1})
=
\sum_{i=1}^{k+1}(-1)^{i+1}
{u_i}_{\lambda_i}\Big(
\varphi_{\lambda_1,\stackrel{i}{\check\dots},\lambda_{k+1}}(u_1,\stackrel{i}{\check\dots},u_{k+1})
\Big) \\
&\qquad +
\sum_{\substack{i,j=1 \\ i<j}}^{k+1}(-1)^{i+j}
\varphi_{\lambda_i+\lambda_j,\lambda_1,\stackrel{i}{\check\dots}\stackrel{j}{\check\dots},\lambda_{k+1}}
\big([{u_i}_{\lambda_i}{u_j}],u_1,\stackrel{i}{\check\dots}\stackrel{j}{\check\dots},u_{k+1}\big)
\,.
\end{split}
\end{equation}

\begin{proposition}\label{prop:LCAd-cohomology}
Formula \eqref{eq:differential} provides a well defined linear map $C^k(E,M)\to C^{k+1}(E,M)$, $k\geq0$,
and $d\circ d=0$.
The corresponding cohomology complex $(C^\bullet(E,M),d)$ is, by definition, the \emph{LCAd cohomology complex}
of $E$ with coefficients in $M$.
\end{proposition}
\begin{proof}
First, note that if 
$\varphi_{\lambda_1,\dots,\lambda_k}\big(u_1,\dots,u_k)
=(\partial+\lambda_1+\dots+\lambda_k)m_{\lambda_1,\dots,\lambda_k}\big(u_1,\dots,u_k)
\equiv 0$ mod $\langle\partial+\lambda_1+\dots+\lambda_k\rangle$,
then the right-hand side of \eqref{eq:differential} is equal to
\begin{align*}
& (\partial+\lambda_1+\dots+\lambda_{k+1})
\Big(
\sum_{i=1}^{k+1}(-1)^{i+1}
{u_i}_{\lambda_i}\Big(
m_{\lambda_1,\stackrel{i}{\check\dots},\lambda_{k+1}}(u_1,\stackrel{i}{\check\dots},u_{k+1})
\Big) \\
&\quad +
\sum_{\substack{i,j=1 \\ i<j}}^{k+1}(-1)^{i+j}
m_{\lambda_i+\lambda_j,\lambda_1,\stackrel{i}{\check\dots}\stackrel{j}{\check\dots},\lambda_{k+1}}
\big([{u_i}_{\lambda_i}{u_j}],u_1,\stackrel{i}{\check\dots}\stackrel{j}{\check\dots},u_{k+1}\big)
\Big)
\,,
\end{align*}
which vanishes in the quotient 
$M[\lambda_1,\dots,\lambda_{k+1}]/\langle\partial+\lambda_1+\dots+\lambda_{k+1}\rangle$.
Hence, the map $d$ factors through the quotient and it is well defined on elements of $C^k(E,M)$.

Next, we show that $d\varphi$, defined by \eqref{eq:differential}, is an element of $C^k(E,M)$,
i.e. it satisfies conditions \eqref{eq:poly-lambda} and \eqref{eq:poly-skew}.
For condition \eqref{eq:poly-lambda} we have
\begin{equation}\label{eq:pr31.1}
\begin{split}
&  (d\varphi)_{\lambda_1,\dots,\lambda_{k+1}}(a_1(\partial)u_1,\dots,a_{k+1}(\partial)u_{k+1}) \\
& =
\sum_{i=1}^{k+1}(-1)^{i+1}
{(a_i(\partial)u_i)}_{\lambda_i}\Big(
\varphi_{\lambda_1,\stackrel{i}{\check\dots},\lambda_{k+1}}
(a_1(\partial)u_1,\stackrel{i}{\check\dots},a_{k+1}(\partial)u_{k+1})
\Big) \\
&+
\sum_{\substack{i,j=1 \\ i<j}}^{k+1}(-1)^{i+j}
\varphi_{\lambda_i+\lambda_j,\lambda_1,\stackrel{i}{\check\dots}\stackrel{j}{\check\dots},\lambda_{k+1}}
\!\!\!\!\!\!\!\!\!\!
\big([{(a_i(\partial)u_i)}_{\lambda_i}{(a_j(\partial)u_j)}],
a_1(\partial)u_1,\stackrel{i}{\check\dots}\stackrel{j}{\check\dots},a_{k+1}(\partial)u_{k+1}\big)
\,.
\end{split}
\end{equation}
By conditions (i) and (ii) of Definition \ref{def:module},
the first term in the right-hand side of \eqref{eq:pr31.1} is
\begin{equation}\label{eq:pr31.2}
\begin{split}
& \sum_{i=1}^{k+1}(-1)^{i+1}
\big(\big|_{x_i=\partial}a_i^*(\lambda_i)\big)
{u_i}_{\lambda_i+x_i}\Big(
\big(\big|_{x_1=\partial}a_1^*(\lambda_1)\big)
\dots \\
&\qquad
\stackrel{i}{\check\dots}
\big(\big|_{x_{k+1}=\partial}a_{k+1}^*(\lambda_{k+1})\big)
\varphi_{\lambda_1+x_1,\stackrel{i}{\check\dots},\lambda_{k+1}+x_{k+1}}
(u_1,\stackrel{i}{\check\dots},u_{k+1})
\Big) \\
& =
\sum_{i=1}^{k+1}(-1)^{i+1}
\big(\big|_{x_1=\partial}a_1^*(\lambda_1)\big)
\dots
\big(\big|_{x_{k+1}=\partial}a_{k+1}^*(\lambda_{k+1})\big) \\
&\qquad\cdot
{u_i}_{\lambda_i+x_i}\Big(
\varphi_{\lambda_1+x_1,\stackrel{i}{\check\dots},\lambda_{k+1}+x_{k+1}}
(u_1,\stackrel{i}{\check\dots},u_{k+1})
\Big) \\
& +
\sum_{i=1}^{k+1}(-1)^{i+1}
\big(\big|_{x_i=\partial}a_i^*(\lambda_i)\big)
{\theta(u_i)}_{\lambda_i+x_i}\Big(
\big(\big|_{x_1=\partial}a_1^*(\lambda_1)\big) 
\dots \\
&\qquad
\stackrel{i}{\check\dots}
\big(\big|_{x_{k+1}=\partial}a_{k+1}^*(\lambda_{k+1})\big)
\Big)
\varphi_{\lambda_1+x_1,\stackrel{i}{\check\dots},\lambda_{k+1}+x_{k+1}}
(u_1,\stackrel{i}{\check\dots},u_{k+1})
\,.
\end{split}
\end{equation}
By the total formula \eqref{eq:tot-form} we have
\begin{align*}
& [{(a_i(\partial)u_i)}_{\lambda_i}{(a_j(\partial)u_j)}]
=
\big(\big|_{x_i=\partial}a_i^*(\lambda_i)\big)
a_j(\lambda_i+x_i+y)
\big(\big|_{y=\partial}[{u_i}_{\lambda_i+x_i}u_j]\big)
 \\
&\qquad +
\big(\big|_{x_i=\partial}a_i^*(\lambda_i)\big)
\theta(u_i)_{\lambda_i+x_i}(a_j(y)) \big(\big|_{y=\partial}u_j\big)  \\
&\qquad - 
a_j(\lambda_i+y+z) 
\big(\big|_{y=\partial}\theta(u_j)^*_{\lambda_i+z}(a_i(z))\big)
\big(\big|_{z=\partial}u_i\big)
\,.
\end{align*}
Hence, the second term in the right-hand side of \eqref{eq:pr31.1} becomes,
after some manipulations involving properties of the anchor map $\theta$ 
and the sesquilinearity property \eqref{eq:poly-lambda}
of $\varphi$,
\begin{equation}\label{eq:pr31.3}
\begin{split}
& \sum_{\substack{i,j=1 \\ i<j}}^{k+1}(-1)^{i+j}
\big(\big|_{x_1=\partial}a_1^*(\lambda_1)\big)
\dots
\big(\big|_{x_{k+1}=\partial}a_{k+1}^*(\lambda_{k+1})\big) \\
&\qquad\qquad
\varphi_{\lambda_i+\lambda_j+x_i+x_j,\lambda_1,\stackrel{i}{\check\dots}\stackrel{j}{\check\dots},\lambda_{k+1}}
\big(
[{u_i}_{\lambda_i+x_i}u_j],
u_1,\stackrel{i}{\check\dots}\stackrel{j}{\check\dots},u_{k+1}\big) \\
& +
\sum_{\substack{i,j=1 \\ i<j}}^{k+1}(-1)^{i+j}
\big(\big|_{x_1=\partial}a_1^*(\lambda_1)\big)
\stackrel{j}{\check\dots}
\big(\big|_{x_{k+1}=\partial}a_{k+1}^*(\lambda_{k+1})\big)
\theta(u_i)_{\lambda_i+x_i}\big(\big|_{x_j=\partial}a_j^*(\lambda_j)\big) \\
&\qquad\qquad
\varphi_{\lambda_j+x_j,\lambda_1+x_1,\stackrel{i}{\check\dots}\stackrel{j}{\check\dots},\lambda_{k+1}+x_{k+1}}
\big(u_j,u_1,\stackrel{i}{\check\dots}\stackrel{j}{\check\dots},u_{k+1}\big) \\
& - 
\sum_{\substack{i,j=1 \\ i<j}}^{k+1}(-1)^{i+j}
\big(\big|_{x_1=\partial}a_1^*(\lambda_1)\big)
\stackrel{i}{\check\dots}
\big(\big|_{x_{k+1}=\partial}a_{k+1}^*(\lambda_{k+1})\big)
\theta(u_j)_{\lambda_j+x_j}\big(\big|_{x_i=\partial}a_i^*(\lambda_i)\big) \\
&\qquad\qquad
\varphi_{\lambda_i+x_i,\lambda_1+x_1,\stackrel{i}{\check\dots}\stackrel{j}{\check\dots},\lambda_{k+1}+x_{k+1}}
\big(u_i,u_1,\stackrel{i}{\check\dots}\stackrel{j}{\check\dots},u_{k+1}\big)
\,.
\end{split}
\end{equation}
Combining the first term in the right-hand side of \eqref{eq:pr31.2} with the first term in \eqref{eq:pr31.3}
we get the desired expression
$$
\big(\big|_{x_i=\partial}a_i^*(\lambda_i)\big)
a_j(\lambda_i+x_i+y)
\big(\big|_{y=\partial}[{u_i}_{\lambda_i+x_i}u_j]\big)
(d\varphi)_{\lambda_1,\dots,\lambda_{k+1}}(u_1,\dots,u_{k+1})
\,.
$$
On the other hand, by the skewsymmetry condition \eqref{eq:poly-skew},
the second and third terms of \eqref{eq:pr31.3} combine to give
\begin{align*}
& \sum_{\substack{i,j=1 \\ i\neq j}}^{k+1}(-1)^{i}
\big(\big|_{x_1=\partial}a_1^*(\lambda_1)\big)
\stackrel{j}{\check\dots}
\big(\big|_{x_{k+1}=\partial}a_{k+1}^*(\lambda_{k+1})\big)
\theta(u_i)_{\lambda_i+x_i}\big(\big|_{x_j=\partial}a_j^*(\lambda_j)\big) \\
&\qquad\qquad
\varphi_{\lambda_1+x_1,\stackrel{i}{\check\dots},\lambda_{k+1}+x_{k+1}}
\big(u_1,\stackrel{i}{\check\dots},u_{k+1}\big)
\,,
\end{align*}
which is opposite to the second term in the right-hand side of \eqref{eq:pr31.2}
since $\theta(u_i)_{\lambda_i+x_i}$ is a conformal derivation of $A$.

As for the skewsymmetry condition \eqref{eq:poly-skew}, we have
\begin{equation}\label{eq:pr31.4}
\begin{split}
& 
(d\varphi)_{\lambda_{\pi(1)},\dots,\lambda_{\pi(k+1)}}(u_{\pi(1)},\dots,u_{\pi(k+1)}) \\
& =
\sum_{i=1}^{k+1}(-1)^{i+1}
{u_{\pi(i)}}_{\lambda_{\pi(i)}}\Big(
\varphi_{\lambda_{\pi(1)},\stackrel{i}{\check\dots},\lambda_{\pi(k+1)}}
(u_{\pi(1)},\stackrel{i}{\check\dots},u_{\pi(k+1)})
\Big) \\
& +
\sum_{\substack{i,j=1 \\ i<j}}^{k+1}(-1)^{i+j}
\varphi_{\lambda_{\pi(i)}+\lambda_{\pi(j)},
\lambda_{\pi(1)},\stackrel{i}{\check\dots}\stackrel{j}{\check\dots},\lambda_{\pi(k+1)}}
\big([{u_{\pi(i)}}_{\lambda_{\pi(i)}}{u_{\pi(j)}}],
u_{\pi(1)},\stackrel{i}{\check\dots}\stackrel{j}{\check\dots},u_{\pi(k+1)}\big)
\,.
\end{split}
\end{equation}
We then observe that, for $1\leq i\leq k+1$, $\{\pi(1),\stackrel{i}{\check\dots},\pi(k+1)\}$
is a permutation of $\{1,\stackrel{\pi(i)}{\check\dots},k+1\}$ of sign $(-1)^{i+\pi(i)}\sign(\pi)$.
Hence, by the skewsymmetry assumption on $\varphi$, the first term of \eqref{eq:pr31.4} is
(setting $\alpha=\pi(i)$)
\begin{equation}\label{eq:pr31.5}
\sign(\pi)\sum_{\alpha=1}^{k+1}(-1)^{\alpha+1}
{u_{\alpha}}_{\lambda_{\alpha}}\Big(
\varphi_{\lambda_{1},\stackrel{\alpha}{\check\dots},\lambda_{k+1}}
(u_{1},\stackrel{\alpha}{\check\dots},u_{k+1})
\Big) 
\,.
\end{equation}
Likewise, $\{\pi(1),\stackrel{i}{\check\dots}\stackrel{j}{\check\dots},\pi(k+1)\}$
is a permutation of $\{1,\stackrel{\pi(i)}{\check\dots}\stackrel{\pi(j)}{\check\dots},k+1\}$
and its sign is $(-1)^{i+j+\pi(i)+\pi(j)}\sign(\pi)$ if $\pi(i)<\pi(j)$,
and it is $(-1)^{i+j+\pi(i)+\pi(j)+1}\sign(\pi)$ if $\pi(i)>\pi(j)$.
Hence, the second term of \eqref{eq:pr31.4} is
\begin{align*}
&\sign(\pi) 
\sum_{\substack{\alpha,\beta=1 \\ \alpha<\beta,\, \pi^{-1}(\alpha)<\pi^{-1}(\beta)}}^{k+1}(-1)^{\alpha+\beta}
\varphi_{\lambda_{\alpha}+\lambda_{\beta},
\lambda_{1},\stackrel{\alpha}{\check\dots}\stackrel{\beta}{\check\dots},\lambda_{k+1}}
\big([{u_{\alpha}}_{\lambda_{\alpha}}{u_{\beta}}],
u_{1},\stackrel{\alpha}{\check\dots}\stackrel{\beta}{\check\dots},u_{k+1}\big) \\
& +\sign(\pi)
\sum_{\substack{\alpha,\beta=1 \\ \alpha<\beta,\pi^{-1}(\beta)<\pi^{-1}(\alpha)}}^{k+1}(-1)^{\alpha+\beta+1}
\varphi_{\lambda_{\alpha}+\lambda_{\beta},
\lambda_{1},\stackrel{\alpha}{\check\dots}\stackrel{\beta}{\check\dots},\lambda_{k+1}}
\big([{u_{\beta}}_{\lambda_{\beta}}{u_{\alpha}}],
u_{1},\stackrel{\alpha}{\check\dots}\stackrel{\beta}{\check\dots},u_{k+1}\big)
\,.
\end{align*}
(We called $\alpha=\pi(i),\,\beta=\pi(j)$ in the first sum, and $\alpha=\pi(j),\,\beta=\pi(i)$ in the second sum.)
We then use the skewsymmetry of the $\lambda$-bracket
and the sesquilinearity condition \eqref{eq:poly-lambda} of $\varphi$, to get
\begin{equation}\label{eq:pr31.6}
\sign(\pi)\sum_{\substack{\alpha,\beta=1 \\ \alpha<\beta}}^{k+1}(-1)^{\alpha+\beta}
\varphi_{\lambda_{\alpha}+\lambda_{\beta},
\lambda_{1},\stackrel{\alpha}{\check\dots}\stackrel{\beta}{\check\dots},\lambda_{k+1}}
\big([{u_{\alpha}}_{\lambda_{\alpha}}{u_{\beta}}],
u_{1},\stackrel{\alpha}{\check\dots}\stackrel{\beta}{\check\dots},u_{k+1}\big)
\,.
\end{equation}
Combining \eqref{eq:pr31.5} and \eqref{eq:pr31.6}, we get
$$
\sign(\pi)(d\varphi)_{\lambda_1,\dots,k+1}(u_1,\dots,u_{k+1})
\,,
$$
as desired.

Hence, $d$ is a well defined map $C^k(E,M)\to C^{k+1}(E,M)$,
and we are left to prove that $d\circ d=0$.
In fact, the formula for the differential is the same as for LCA cohomology,
and the proof that $d\circ d=0$ is the same, see e.g. \cite[Lem.2.1]{BKV} or \cite[Prop.1]{DK09}.
\end{proof}

\begin{definition}\label{def:LCAd-cohomology}
The \emph{LCAd cohomology} of the LCAd $E$ with coefficients in the $E$-module $M$
is the cohomology of the complex $(C^\bullet(E,M),d)$, i.e.
$$
H^k(E,M)
=
\ker\big(d\big|_{C^k(E,M)}\big)\big/ \im\big(d\big|_{C^{k-1}(E,M)}\big)
\,,\quad k\geq0
\,.
$$
\end{definition}

\subsection{Basic and reduced complexes}\label{sec:3.2}

If we do not quotient $M[\lambda_1,\dots,\lambda_k]$ by $\langle\partial+\lambda_1+\dots+\lambda_k\rangle$
in the definition \eqref{eq:k-cochains} of $k$-cochains,
we obtain the so-called \emph{basic} cohomology complex associated to an LCAd $E$
and its module $M$.

First, if $E$ and $M$ are left $A[\partial]$-modules,
the space of \emph{basic} $k$-\emph{cochains} $\widetilde{C}^k(E,M)$, $k\geq0$,
is the space of maps 
$\widetilde{\varphi}_{\lambda_1,\dots,\lambda_k}:\,E^{\otimes k}\to M[\lambda_1,\dots,\lambda_k]$
satisfying the same equations \eqref{eq:poly-lambda} and \eqref{eq:poly-skew}
with $\widetilde\varphi$ in place of $\varphi$,
except that now all these identities are understood in $M[\lambda_1,\dots,\lambda_k]$
(and not in its quotient by $\langle\partial+\lambda_1+\dots+\lambda_k\rangle$).
In particular, $\widetilde{C}^0(E,M)\simeq M$
and $\widetilde{C}^1(E,M)=\RCHom_A(E,M)$.

Next, if $E$ is an LCAd over the differential algebra $A$
and $M$ is a module over $E$,
we make $\widetilde{C}^\bullet(E,M)=\oplus_{k\geq0}\widetilde{C}^k(E,M)$
into a cohomology complex, with the differential 
$\widetilde{d}:\,\widetilde{C}^k(E,M)\to \widetilde{C}^{k+1}(E,M),\,k\geq0$, defined by
the same formula as in \eqref{eq:differential} 
with $\widetilde\varphi$ in place of $\varphi$ and $\widetilde d$ in place of $d$
(where now both sides are understood in $M[\lambda_1,\dots,\lambda_{k+1}]$).

\begin{proposition}\label{prop:basic-LCAd-cohomology}
Formula \eqref{eq:differential} 
with $\widetilde\varphi$ in place of $\varphi$ and $\widetilde d$ in place of $d$
provides a well defined linear map 
$\widetilde{C}^k(E,M)\to\widetilde C^{k+1}(E,M)$, $k\geq0$,
and $\widetilde d\circ\widetilde d=0$.
The corresponding cohomology complex $(\widetilde C^\bullet(E,M),\widetilde d)$ is, 
by definition, the \emph{basic LCAd cohomology complex}
of $E$ with coefficients in $M$.
\end{proposition}
\begin{proof}
The proof is the same as for the proof of Proposition \ref{prop:LCAd-cohomology}.
\end{proof}

\begin{definition}\label{def:LCAd-cohomology-basic}
The \emph{basic LCAd cohomology} of the LCAd $E$ with coefficients in the $E$-module $M$
is the cohomology of the complex $(\widetilde{C}^\bullet(E,M),\widetilde d)$, i.e.
$$
\widetilde{H}^k(E,M)
=
\ker\big(\widetilde{d}\big|_{\widetilde{C}^k(E,M)}\big)\big/ \im\big(\widetilde{d}\big|_{\widetilde{C}^{k-1}(E,M)}\big)
\,,\quad k\geq0
\,.
$$
\end{definition}

\begin{proposition}\label{prop:reduced-LCAd-cohomology}
We have an action of $\mb F[\partial]$ on the spaces $\widetilde{C}^k(E,M)$, $k\geq0$,
given by
\begin{equation}\label{eq:partial-basic}
(\partial\widetilde\varphi)_{\lambda_1,\dots,\lambda_k}(u_1,\dots,u_k)
=
(\lambda_1+\dots+\lambda_k+\partial)
\widetilde\varphi_{\lambda_1,\dots,\lambda_k}(u_1,\dots,u_k)
\,,
\end{equation}
which commutes with the action of the differential:
$\widetilde d\circ\partial=\partial\circ\widetilde d$.
Hence, quotienting by the action of $\partial$ we get the induced \emph{reduced LCAd cohomology complex},
$\big(\overline{C}^\bullet(E,M), \overline{d}\big)$, where
\begin{equation}\label{eq:reduced-complex}
\overline{C}^k(E,M):=\widetilde{C}^k(E,M)/\partial\widetilde{C}^k(E,M),\,k\geq0
\,,
\end{equation}
and $\overline d$ is induced by $d$.

We have a natural morphism of complexes 
$p_*:\,(\widetilde{C}^\bullet(E,M),\widetilde{d})\to(C^\bullet(E,M),d)$,
defined by 
$p_*(\widetilde\varphi)=p\circ\widetilde{\varphi}$, where 
$p:\,M[\lambda_1,\dots,\lambda_k]\twoheadrightarrow 
M[\lambda_1,\dots,\lambda_k]/\langle\partial+\lambda_1+\dots+\lambda_k\rangle$
is the canonical quotient map,
and $\partial C^\bullet(E,M)$ is the kernel of this morphism.
Hence, we get an induced injective morphism of complexes
(still denoted by $p_*$ with an abuse of notation)
$$
p_*:\,(\overline{C}^\bullet(E,M),\overline{d})\to(C^\bullet(E,M),d)\,,
$$
which is surjective provided that $E$ is free as a left $A[\partial]$-module.
\end{proposition}
\begin{proof}
To show that $\partial\widetilde\varphi$, defined by \eqref{eq:partial-basic},
satisfies conditions \eqref{eq:poly-lambda} and \eqref{eq:poly-skew}
is a straightforward computation. Hence, \eqref{eq:partial-basic} gives a well defined 
endomorphism of $\widetilde{C}^k(E,M)$, $k\geq0$.
Also, the commutation relation $\widetilde d\circ\partial=\partial\circ\widetilde d$
is straightforward (and it is the same as for the basic LCA cohomology complex, \cite{BKV}).

Next, consider the pushforward map $p_*:\,\widetilde{C}^k(E,M)\to C^k(E,M)$
given by $p_*(\widetilde\varphi)= p\circ\widetilde\varphi$.
It clearly commutes with the corresponding differentials $\widetilde d$ and $d$,
as the formulas defining them are the same, \eqref{eq:differential}. 
Namely, $p_*\circ\widetilde d=d\circ p_*$.
Hence $p_*$ is a morphism of complexes.
We need to show that its kernel coincides with $\partial\widetilde{C}^k(E,M)$,
and it is surjective provided that $E$ is a free left $A[\partial]$-module.
Obviously, by \eqref{eq:partial-basic} 
$(\partial\widetilde\varphi)_{\lambda_1,\dots,\lambda_k}(u_1,\dots,u_k)$
lies in $\langle\partial+\lambda_1+\dots+\lambda_k\rangle$,
so that $p\circ\partial\widetilde\varphi=0$,
i.e. $\partial\widetilde\varphi$ is in the kernel of $p_*$.
Conversely, if $\widetilde\varphi$ is in the kernel of $p_*$,
we have, for every $u_1,\dots,u_k\in E$,
$$
\widetilde\varphi_{\lambda_1,\dots,\lambda_k}(u_1,\dots,u_k)
=(\partial+\lambda_1+\dots+\lambda_k)F_{\lambda_1,\dots,\lambda_k}(u_1,\dots,u_k)
\,.
$$
Notice that $\partial+\lambda_1+\dots+\lambda_k$
defines an injective endomorphism of $M[\lambda_1,\dots,\lambda_k]$.
Using this fact, we obtain that $F_{\lambda_1,\dots,\lambda_k}(u_1,\dots,u_k)$
is linear in each $u_i$, satisfies \eqref{eq:poly-lambda} and \eqref{eq:poly-skew},
since $\widetilde\varphi_{\lambda_1,\dots,\lambda_k}(u_1,\dots,u_k)$ satisfies the same properties.
Linearity in each $u_i$ and the skewsymmetry condition \eqref{eq:poly-skew}
are the same as in the case of LCA cohomology.
As for condition \eqref{eq:poly-lambda}, we have, 
from condition \eqref{eq:poly-lambda} on $\widetilde\varphi$,
\begin{align*}
& (\partial+\lambda_1+\dots+\lambda_k)F_{\lambda_1,\dots,\lambda_k}(a_1(\partial)u_1,\dots,a_k(\partial)u_k) \\
& =
\big(\big|_{x_1=\partial}a_1^*(\lambda_1)\big)
\dots
\big(\big|_{x_{k}=\partial}a_{k}^*(\lambda_{k})\big)
(\partial+\lambda_1+x_1+\dots+\lambda_k+x_k)
F_{\lambda_1+x_1,\dots,\lambda_k+x_k}\big(u_1,\dots,u_k) \\
& =
(\partial+\lambda_1+\dots+\lambda_k)
\big(\big|_{x_1=\partial}a_1^*(\lambda_1)\big)
\dots
\big(\big|_{x_{k}=\partial}a_{k}^*(\lambda_{k})\big)
F_{\lambda_1+x_1,\dots,\lambda_k+x_k}\big(u_1,\dots,u_k)
\,,
\end{align*}
from which we get that $F$ satisfies \eqref{eq:poly-lambda}.
As a consequence, $F\in\widetilde{C}^k(E,M)$, and $\widetilde\varphi=\partial F\in\partial\widetilde{C}^k(E,M)$,
as desired.

Finally, we turn to surjectivity of $p_*$.
For $k=0$ it is obvious as $\widetilde{C}^0(E,M)=M$ while $C^0(E,M)=M/\partial M$,
and $p_*$ coincides with the quotient map.
For $k\geq1$, 
consider the endomorphism $\widetilde q$ of $M[\lambda_1,\dots,\lambda_k]$ given by
$$
\widetilde q(m(\lambda_1,\dots,\lambda_k))
=
\Big(\Big|_{x=\partial}
m\big(\lambda_1-k^{-1}(\Lambda+x),\dots,\lambda_k-k^{-1}(\Lambda+x)\big)
\Big)
\,,
$$
where $\Lambda=\lambda_1+\dots+\lambda_k$.
Note that
$\widetilde q$ vanishes on $\langle\partial+\lambda_1+\dots+\lambda_k\rangle$,
hence it factors through a map
$$
q:\,M[\lambda_1,\dots,\lambda_k]/\langle\partial+\lambda_1+\dots+\lambda_k\rangle
\to M[\lambda_1,\dots,\lambda_k]
\,.
$$
It follows by construction that $p\circ q=\id$.
In particular $q$ is injective.
Assume now that $E=A[\partial]\otimes_\mb FU$ is free as an $A[\partial]$-module.
Given $\varphi\in C^k(E,M)$, we want to construct a preimage $\widetilde\varphi\in\widetilde{C}^k(E,M)$
such that $p_*(\widetilde\varphi)=\varphi$.
Such preimage is defined by
\begin{align*}
& \widetilde{\varphi}_{\lambda_1,\dots,\lambda_k}(a_1(\partial)u_1,\dots,a_k(\partial)u_k) \\
& =
\big(\big|_{x_1=\partial}a_1^*(\lambda_1)\big)
\dots
\big(\big|_{x_k=\partial}a_k^*(\lambda_k)\big)
q\big(
\varphi_{\lambda_1+x_1,\dots,\lambda_k+x_k}(u_1,\dots,u_k)
\big)
\,,
\end{align*}
for $a_i(\partial)\in A[\partial],\,u_i\in U,\,i=1,\dots,k$.
Note that such $\widetilde{\varphi}$ satisfies condition \eqref{eq:poly-lambda} by construction,
and condition \eqref{eq:poly-skew} since $\varphi$ does.
Hence, $\widetilde{\varphi}$ lies in $\widetilde{C}^k(E,M)$.
On the other hand, for $u_1,\dots,u_k\in U$, we have
$$
(p_*\widetilde{\varphi})_{\lambda_1,\dots,\lambda_k}(u_1,\dots,u_k)
=
p\circ q\big(\varphi_{\lambda_1,\dots,\lambda_k}(u_1,\dots,u_k)\big)
=
\varphi_{\lambda_1,\dots,\lambda_k}(u_1,\dots,u_k)
\,.
$$
Hence, $p_*\widetilde{\varphi}$ coincides with $\varphi$ on $U^{\otimes k}$,
and then they coincide on the whole $E^{\otimes k}$
since they both satisfy \eqref{eq:poly-lambda}.
\end{proof}
\begin{remark}\label{rem:not-surj}
If $E$ is not free as an $A[\partial]$-module, the map $p_*:\,\widetilde{C}^k(E,M)\to C^k(E,M)$
needs not be surjective.
As an example, consider the case when $E=A=\mb F$, with trivial action of $\partial$,
and $M$ is an $\mb F[\partial]$-module with non-zero subspace of invariants 
$M^\partial=\{u\in M\,|\,\partial(u)=0\}\neq0$.
In this case, $p_*$ cannot be surjective for $k=1$ as $\widetilde{C}^1(\mb F,M)=0$
while $C^1(\mb F,M)=M^\partial\neq0$.
Indeed, 
an element $\widetilde{\varphi}\in\widetilde{C}^1(E,M)$ is uniquely determined by 
$\widetilde{\varphi}_\lambda(1)\in M[\lambda]$ satisfying
$0=\widetilde{\varphi}_\lambda(\partial 1)=-\lambda\widetilde{\varphi}_\lambda(1)$,
which implies $\widetilde{\varphi}_\lambda(1)=0$.
On the other hand, 
an element $\varphi$ in $C^1(\mb F,M)$ is a map
$\varphi:\,\mb F\to M[\lambda]/\langle\partial+\lambda\rangle\simeq M$
commuting with $\partial$,
hence it is uniquely determined by $\varphi(1)\in M^\partial$.
\end{remark}
\begin{remark}\label{rem:Ad-basic}
In fact, we have a left action of $A[\partial]$ on the spaces $\widetilde{C}^k(E,M)$, $k\geq0$,
given by
\begin{equation}\label{eq:apartial-basic}
(a(\partial)\widetilde\varphi)_{\lambda_1,\dots,\lambda_k}(u_1,\dots,u_k)
=
a(\lambda_1+\dots+\lambda_k+x)
\big(\big|_{x=\partial}\widetilde\varphi_{\lambda_1,\dots,\lambda_k}(u_1,\dots,u_k)\big)
\,.
\end{equation}
However, this action does not commute with the action of the differential $\widetilde d$,
as we have, for $a\in A$,
\begin{align*}
& ((\widetilde d\circ a-a\circ\widetilde d)
\widetilde\varphi)_{\lambda_1,\dots,\lambda_{k+1}}(u_1,\dots,u_{k+1}) \\
& =
\sum_{i=1}^{k+1}(-1)^{i+1}
\theta(u_i)_{\lambda_i}(a)
\widetilde\varphi_{\lambda_1,\stackrel{i}{\check\dots},\lambda_{k+1}}(u_1,\stackrel{i}{\check\dots},u_{k+1})
\,,
\end{align*}
hence, we do not have an induced action of $A[\partial]$, or of $A$, on the reduced complex
$(\overline{C}^\bullet(E,M),\overline{d})$.
\end{remark}

\begin{definition}\label{def:LCAd-cohomology-red}
The \emph{reduced LCAd cohomology} of the LCAd $E$ with coefficients in the $E$-module $M$
is the cohomology of the complex $(\overline{C}^\bullet(E,M),\overline{d})$, i.e.
$$
\overline{H}^k(E,M)
=
\ker\big(\overline{d}\big|_{\overline{C}^k(E,M)}\big)\big/ \im\big(\overline{d}\big|_{\overline{C}^{k-1}(E,M)}\big)
\,,\quad k\geq0
\,.
$$
\end{definition}

\begin{lemma}\label{lem:les-cohom}
For an LCAd $E$ and an $E$-module $M$,
we have 
\begin{equation}\label{eq:les-cohom2}
H^k(\partial\widetilde{C}^\bullet(E,M),\widetilde{d})
\simeq
\widetilde{H}^k(E,M)
\,,\,\, k\geq1\,.
\end{equation}
\end{lemma}
\begin{proof}
The map $\partial:\,\widetilde{C}^k(E,M)\to\partial\widetilde{C}^k(E,M)$ is obviously surjective.
Recalling \eqref{eq:partial-basic}, we also have that 
$\partial:\,\widetilde{C}^k(E,M)\to\partial\widetilde{C}^k(E,M)$
is injective for every $k\geq1$,
since $\partial+\lambda_1+\dots+\lambda_k$ is injective on $M[\lambda_1,\dots,\lambda_k]$.
While, for $k=0$, this is not necessarily the case, as
$$
\ker\big(\partial\big|_{\widetilde{C}^0(E,M)}\big)
=
\ker(\partial|_M)
=:M^\partial
\,.
$$
Hence, for $k\geq2$, we then have the commutative diagram
$$
\begin{tikzcd}
\widetilde{C}^{k-1}(E,M) \arrow[r, "\sim"{below}, "\partial"{above}] \arrow[d, "\widetilde{d}"]
& \partial\widetilde{C}^{k-1}(E,M) \arrow[d, "\widetilde{d}"] \\
\widetilde{C}^{k}(E,M) \arrow[r, "\sim"{below}, "\partial"{above}] \arrow[d, "\widetilde{d}"]
& \partial\widetilde{C}^{k}(E,M) \arrow[d, "\widetilde{d}"] \\
\widetilde{C}^{k+1}(E,M) \arrow[r, "\sim"{below}, "\partial"{above}]
& \partial\widetilde{C}^{k+1}(E,M) 
\end{tikzcd}
$$
from which it follows immediately that
\begin{align*}
& \widetilde{H}^k(E,M)
=
\ker\big(\widetilde{d}\big|_{\widetilde{C}^{k}(E,M)}\big)
\big/
\im\big(\widetilde{d}\big|_{\widetilde{C}^{k-1}(E,M)}\big) \\
& \simeq
\ker\big(\widetilde{d}\big|_{\partial\widetilde{C}^{k}(E,M)}\big)
\big/
\im\big(\widetilde{d}\big|_{\partial\widetilde{C}^{k-1}(E,M)}\big)
=
H^k(\partial\widetilde{C}^\bullet(E,M),\widetilde{d})
\,.
\end{align*}
A similar argument works for $k=1$. We have the commutative diagram
$$
\begin{tikzcd}
\widetilde{C}^{0}(E,M)=M \arrow[r,  twoheadrightarrow, "\partial"{above}] \arrow[d, "\widetilde{d}"]
& \partial\widetilde{C}^{0}(E,M)=\partial M \arrow[d, "\widetilde{d}"] \\
\widetilde{C}^{1}(E,M) \arrow[r, "\sim"{below}, "\partial"{above}] \arrow[d, "\widetilde{d}"]
& \partial\widetilde{C}^{1}(E,M) \arrow[d, "\widetilde{d}"] \\
\widetilde{C}^{2}(E,M) \arrow[r, "\sim"{below}, "\partial"{above}]
& \partial\widetilde{C}^{2}(E,M) 
\end{tikzcd}
$$
The first line of maps factors through
$$
M\twoheadrightarrow M/M^\partial \stackrel{\sim}{\longrightarrow} \partial M
\,.
$$
On the other hand, $M^\partial$ lies in the kernel of $\widetilde{d}:\, M\to\widetilde{C}^{1}(E,M)$.
Indeed, if $m\in M^\partial$, i.e. $\partial m=0$, then
$$
(\partial+\lambda)(\widetilde{d}m)_\lambda(u)
=
(\partial+\lambda)(u_\lambda m)
=
u_\lambda(\partial m)
=0
\,,
$$
which implies $\widetilde{d}m=0$ as $\partial+\lambda$ is injective on $M[\lambda]$.
Hence, the map $\widetilde{d}:\, M\to\widetilde{C}^{1}(E,M)$ factors through
$$
M\twoheadrightarrow M/M^\partial
\stackrel{\hat d}{\longrightarrow}
\widetilde{C}^{1}(E,M)
\,.
$$
And, by construction, $\widetilde{d}(M)=\hat{d}(M/M^\partial)$.
Then, the above diagram implies
$$
\begin{tikzcd}
M/M^\partial \arrow[r,  "\sim"{below}] \arrow[d, "\hat{d}"]
& \partial M \arrow[d, "\widetilde{d}"] \\
\widetilde{C}^{1}(E,M) \arrow[r, "\sim"{below}, "\partial"{above}] \arrow[d, "\widetilde{d}"]
& \partial\widetilde{C}^{1}(E,M) \arrow[d, "\widetilde{d}"] \\
\widetilde{C}^{2}(E,M) \arrow[r, "\sim"{below}, "\partial"{above}]
& \partial\widetilde{C}^{2}(E,M) 
\end{tikzcd}
$$
From this it follows immediately that
\begin{align*}
& \widetilde{H}^1(E,M)
=
\ker\big(\widetilde{d}\big|_{\widetilde{C}^{1}(E,M)}\big)
\big/
\hat{d}(M/M^\partial) \\
& \simeq
\ker\big(\widetilde{d}\big|_{\partial\widetilde{C}^{1}(E,M)}\big)
\big/
\widetilde{d}(\partial M)
=
H^1(\partial\widetilde{C}^\bullet(E,M),\widetilde{d})
\,.
\end{align*}
\end{proof}

\begin{proposition}\label{prop:les-cohom}
Let $E$ be an LCAd which is free as a left $A[\partial]$-module
and let $M$ be a module over the LCAd $E$.
We have a long exact sequence of the form
\begin{equation}\label{eq:les-cohom}
\begin{split}
& \widetilde{H}^0(E,M)
\to
H^0(E,M)
\to
\widetilde{H}^1(E,M)
\to
\widetilde{H}^1(E,M)
\to
H^1(E,M)
\to
\dots \\
& \dots 
\to
\widetilde{H}^k(E,M)
\to
\widetilde{H}^k(E,M)
\to
H^k(E,M)
\to
\widetilde{H}^{k+1}(E,M)
\to
\dots
\end{split}
\end{equation}
\end{proposition}
\begin{proof}
According to Proposition \ref{prop:reduced-LCAd-cohomology},
for an LCAd $E$ and an $E$-module $M$
we have an exact sequence of complexes
$$
0\to (\partial\widetilde{C}^\bullet(E,M),\widetilde{d})
\stackrel{\iota}{\longrightarrow}
(\widetilde{C}^\bullet(E,M),\widetilde{d})
\stackrel{p_*}{\longrightarrow}
(C^\bullet(E,M),d)
\to 0
\,,
$$
where $\iota$ is the canonical inclusion map,
and the map $p_*$ is obtained composing with the quotient map $p$.
In fact, $\iota$ is always injective, while $p_*$ is surjective
under the assumption that $E$ is free as left module over $A[\partial]$.
Hence, we get the corresponding long exact sequence in cohomology:
\begin{align*}
& 0\to H^0(\partial\widetilde{C}^\bullet(E,M),\widetilde{d})
\to
\widetilde{H}^0(E,M)
\to
H^0(E,M)
\stackrel{\delta}{\longrightarrow} \\
& \stackrel{\delta}{\longrightarrow}
H^1(\partial\widetilde{C}^\bullet(E,M),\widetilde{d})
\to
\widetilde{H}^1(E,M)
\to
H^1(E,M)
\stackrel{\delta}{\longrightarrow} \dots \\
& \dots \stackrel{\delta}{\longrightarrow}
H^k(\partial\widetilde{C}^\bullet(E,M),\widetilde{d})
\to
\widetilde{H}^k(E,M)
\to
H^k(E,M)
\stackrel{\delta}{\longrightarrow} \dots
\end{align*}
where $\delta$ is the coboundary map.
Combining this long exact sequence with the isomorphisms \eqref{eq:les-cohom2},
and ignoring the first two maps of the sequence,
we get the exact sequence \eqref{eq:les-cohom}.
\end{proof}

\begin{corollary}\label{cor:les-cohom}
Let $E$ be an LCAd which is free as a left $A[\partial]$-module
and let $M$ be a module over the LCAd $E$.
If the basic LCAd cohomology vanishes in degree $\geq d$, then so does the LCAd cohomology:
$$
\widetilde{H}^k(E,M)=0 \text{ for } k\geq d
\,\,\Rightarrow\,\,
H^k(E,M)=0 \text{ for } k\geq d
\,.
$$
\end{corollary}
\begin{proof}
It is an immediate consequence of the long exact sequence \eqref{eq:les-cohom}.
\end{proof}

\subsection{Abelian extensions of an LCAd}\label{sec:3.3}

\begin{definition}\label{def:abext}
Let $E$ be an LCAd over the differential algebra $A$,
and let $M$ be a left $A[\partial]$-module, which we view as an \emph{abelian} LCAd,
with $[\cdot\,_\lambda\,\cdot]_M=0$ and $\theta_M=0$.
An \emph{abelian extension} of $E$ by $M$ is an LCAd $\widetilde{E}$
endowed with a short exact sequence of homomorphisms of LCAd
\begin{equation}\label{eq:ses-LCAd}
0\to M\stackrel{\iota}{\hookrightarrow} \widetilde E\stackrel{p}{\twoheadrightarrow} E\to 0
\,.
\end{equation}
\end{definition}
Such abelian extension is called $A[\partial]$-\emph{split}
if there is a morphism of left $A[\partial]$-modules $j:\,E\to\widetilde E$
such that $p\circ j=\id_E$.
Equivalently, we can identify $\widetilde E=E\oplus M$ as a left $A[\partial]$-module.
In this case, the anchor map $\theta^{\widetilde{}}$ of $\widetilde E$ acts trivially on $M$
and restricts to $\theta_E$ on $E$, i.e. ($v\in E,m\in M$):
\begin{equation}\label{eq:anchor-ext}
\theta^{\widetilde{}}(v+m)=\theta_E(v)\in\CDer(A)
\,,
\end{equation}
while the $\lambda$-bracket $[\cdot\,_\lambda\,\cdot]^{\widetilde{}}$ on $\widetilde E$
has the form ($u,v\in E,\,m,n\in M$)
\begin{equation}\label{eq:lambda-ext}
[u+m_\lambda v+n]^{\widetilde{}}
=
[u_\lambda v]_E
+
\omega_\lambda(u,v)
+
\rho(u)_\lambda(n)
-
\rho(v)_\lambda^*(m)
\,,
\end{equation}
where $\omega_{\lambda}(u,v)\in M[\lambda]$ and $\rho(u)\in\CEnd(M)$.
The \emph{trivial} abelian extension corresponds to the choice $\omega=0$
and $\rho=0$.

Two abelian extensions $E_1$ and $E_2$ are \emph{equivalent}
if there is an isomorphism of LCAd $\chi:\,E_1\to E_2$
making the following diagram commute:
\begin{equation}\label{diagram}
\begin{tikzcd}
0 \arrow[r] & 
M \arrow[r, hook,  "\iota_1"{above}] \arrow[d, equal] & 
E_1 \arrow[r,  two heads, "p_1"{above}] \arrow[d, "\chi"] & 
E \arrow[r] \arrow[d, equal] & 
0 \\
0 \arrow[r] &
M \arrow[r,  hook, "\iota_2"{above}] & 
E_2 \arrow[r,  two heads, "p_2"{above}] & 
E \arrow[r] & 
0
\end{tikzcd}
\end{equation}

\begin{lemma}\label{lem:ext}
\begin{enumerate}[(a)]
\item
Equations \eqref{eq:anchor-ext} and \eqref{eq:lambda-ext}
define an $A[\partial]$-split abelian extension of $E$ by $M$
if and only if $\rho$ defines an $E$-module structure on $M$ via \eqref{eq:action2}
and $\omega_{\lambda}$ satisfies the following identities:
\begin{align}
& \omega_\lambda(a(\partial)u,b(\partial)v)
=
\big(\big|_{x=\partial}a^*(\lambda)\big)b(\lambda+x+y)\big(\big|_{y=\partial}\omega_{\lambda+x}(u,v)\big)
\label{eq:2cocycle1}
\,, \\
& \omega_{\lambda}(u,v)
=
-\big(\big|_{x=\partial}\omega_{-\lambda-x}(v,u)\big)
\label{eq:2cocycle2}
\,, \\
& \rho(u)_\lambda(\omega_\mu(v,w))
-\rho(v)_\mu(\omega_\lambda(u,w))
+\rho(w)_{\lambda+\mu}^*(\omega_\lambda(u,v)) \nonumber\\
& \quad
+\omega_\lambda(u,[v_\mu w])-\omega_\mu(v,[u_\lambda w])
-\omega_{\lambda+\mu}([u_\lambda v],w)
=0 \label{eq:2cocycle3}
\,.
\end{align}
\item
Let $E_1=E_2=E\oplus M$ be $A[\partial]$-split extensions of $E$ by $M$,
associated to the maps $\omega_i:\, E^{\otimes2}\to M[\lambda]$ and $\rho_i:\,E\to\CEnd(M)$, $i=1,2$,
via \eqref{eq:anchor-ext}-\eqref{eq:lambda-ext}.
An equivalence $\chi:\,E_1\to E_2$ of abelian extensions 
has the form $\chi(v+m)=v+m+\psi(v)$ for some homomorphism of left $A[\partial]$-modules $\psi:\,E\to M$,
and the commutativity of the diagram \eqref{diagram} is equivalent to requiring that
$\rho_1=\rho_2=:\rho$ and
\begin{equation}\label{equiv-ext}
{\omega_2}_\lambda(u,v)-{\omega_1}_\lambda(u,v)
=
\psi([u_\lambda v]_E)-\rho(u)_\lambda(\psi(v))+\rho(v)_\lambda^*(\psi(u))
\,.
\end{equation}
\end{enumerate}
\end{lemma}
\begin{proof}
For part (a), 
we need to show that
the $\lambda$-bracket $[\cdot\,_\lambda\,\cdot]^{\widetilde{}}$ on $E\oplus M$ 
and the anchor map $\theta^{\widetilde{}}:\,E\oplus M\to\CDer(A)$
satisfy the LCAd axioms of Definition \ref{def:LCAd}
if and only if the map $\rho:\,E\to\CEnd(M)$ satisfies the $E$-module axioms of Definition \ref{def:module}
and the map $\omega_\lambda:\,E^{\otimes2}\to M[\lambda]$
satisfies conditions \eqref{eq:2cocycle1}, \eqref{eq:2cocycle2} and \eqref{eq:2cocycle3}.
Indeed, 
axiom (i) of Definition \ref{def:LCAd} for the $\lambda$-bracket $[u_\lambda a(\partial)m]^{\widetilde{}}$,
for $u\in E$ and $m\in M$,
translates to axiom (i) of Definition \ref{def:module} (or equivalently to the condition that $(\rho(u),\theta(u))$
lies in $\mc G(A,M)$);
axiom (i') of Definition \ref{def:LCAd} for the $\lambda$-bracket $[a(\partial)u_\lambda m]^{\widetilde{}}$,
for $u\in E$ and $m\in M$,
translates to axiom (ii) of Definition \ref{def:module} (or equivalently to the fact that 
$\rho:\,E\to\CEnd(M)$ is a homomorphism of left $A[\partial]$-modules);
the Jacobi identity of the $\lambda$-bracket $[\cdot\,_\lambda\,\cdot]^{\widetilde{}}$
for the triple of elements $(u,v,m)$, with $u,v\in E$ and $m\in M$,
translates to axiom (iii) of Definition \ref{def:module} (or equivalently to the condition that 
$\rho:\,E\to\CEnd(M)$ is a homomorphism of LCA);
furthermore, the total formula \eqref{eq:tot-form} of Definition \ref{def:LCAd} 
for the $\lambda$-bracket $[a(\partial)u_\lambda b(\partial)v]^{\widetilde{}}$,
for $u,v\in E$ and $a(\partial),b(\partial)\in A[\partial]$,
translates to equation \eqref{eq:2cocycle1};
the skewsymmetry of the $\lambda$-bracket $[\cdot\,_\lambda\,\cdot]^{\widetilde{}}$
for the pair of elements $u,v\in E$
translates to equation \eqref{eq:2cocycle2};
and the Jacobi identity of the $\lambda$-bracket $[\cdot\,_\lambda\,\cdot]^{\widetilde{}}$
for the triple of elements $u,v,w\in E$
translates to equation \eqref{eq:2cocycle3}.
This shows one direction of claim (a).
The converse direction is also straightforward.

For part (b), the map $\chi:\,E\oplus M\to E\oplus M$, by the commutativity of the digram \eqref{diagram}
must restrict to the identity on $M$ and, modulo $M$, must factor to the identity of $E$;
namely, it must have the required form $\chi(v+m)=v+m+\psi(v)$ for some $\psi:\,E\to M$.
Since $\chi$ is a homomorphism of left $A[\partial]$-modules, so must be $\psi$.
Furthermore, $\chi$ trivially preserves the anchor map,
while it intertwines the $\lambda$-bracket $[\cdot\,_\lambda\,\cdot]_1$ and $[\cdot\,_\lambda\,\cdot]_2$
if and only if ($u,v\in E$, $m,n\in M$):
$$
[\chi(u+m)_\lambda\chi(v+n)]_2
=
\chi\big([u+m_\lambda v+n]_1\big)
\,.
$$
The above equations with $v=m=0$ (or $u=n=0$) translates  to the identity $\rho_1=\rho_2$,
while the above equation for $m=n=0$ translates to condition \eqref{equiv-ext}.
\end{proof}
According to Lemma \ref{lem:ext}, an $A[\partial]$-split abelian extension of an LCAd $E$
by an $E$-module $M$, with a given $E$-module structure $\rho$,
is the same as a map $\omega_{\lambda}:\,E^{\otimes2}\to M[\lambda]$
satisfying \eqref{eq:2cocycle1}-\eqref{eq:2cocycle3}.
\begin{proposition}\label{prop:ext}
Let $E$ be an LCAd and let $M$ be an $E$-module with action associated to 
the map $\rho:\,E\to\CEnd(M)$ via \eqref{eq:action2}.
\begin{enumerate}[(a)]
\item
The space of maps $\omega_{\lambda}:\,E^{\otimes2}\to M[\lambda]$,
defining an $A[\partial]$-split abelian extension of $E$ by the $E$-module $M$, 
is canonically isomorphic to $\ker\big(d\big|_{C^2(E,M)}\big)$.
\item
Under the above isomorphism, the subspace of maps $\omega_{\lambda}$
corresponding to trivial abelian extensions is identified with $\im\big(d\big|_{C^1(E,M)}\big)$.
\end{enumerate}
\end{proposition}
\begin{proof}
Note that we can identify 
$$
M[\lambda_1,\lambda_2]/\langle\partial+\lambda_1+\lambda_2\rangle\simeq M[\lambda]
\,
$$
by setting $\lambda_1=\lambda$ and $\lambda_2=-\lambda-\partial$ (with $\partial$ acting on coefficients).
Under this identification, equations \eqref{eq:poly-lambda} and \eqref{eq:poly-skew}
defining a $2$-cochain become \eqref{eq:2cocycle1} and \eqref{eq:2cocycle2} respectively,
while the condition $d\varphi=0$ coincides with equation in \eqref{eq:2cocycle3}.
We only check the last assertion.
Indeed, by the definition \eqref{eq:differential} of the differential $d$,
we have
\begin{align*}
& (d\varphi)_{\lambda_1,\lambda_2,\lambda_3}(u,v,w)
=
u_{\lambda_1}(\varphi_{\lambda_2,\lambda_3}(v,w))
- v_{\lambda_2}(\varphi_{\lambda_1,\lambda_3}(u,w))
+ w_{\lambda_3}(\varphi_{\lambda_1,\lambda_2}(u,v)) \\
&\qquad
- \varphi_{\lambda_1+\lambda_2,\lambda_3}\big([u_{\lambda_1}v],w\big)
+ \varphi_{\lambda_1+\lambda_3,\lambda_2}\big([u_{\lambda_1}w],v\big)
- \varphi_{\lambda_2+\lambda_3,\lambda_1}\big([v_{\lambda_2}w],u\big) \\
&\quad
=
\rho(u)_{\lambda}(\omega_{\mu}(v,w))
- \rho(v)_{\mu}(\omega_{\lambda}(u,w))
+ \big(\big|_{x=\partial} \rho(w)_{-\lambda-\mu-x}(\omega_{\lambda}(u,v)) \big) \\
&\qquad 
- \omega_{\lambda+\mu}\big([u_{\lambda}v],w\big)
+ \big(\big|_{x=\partial} \omega_{-\mu-x}\big([u_{\lambda}w],v\big) \big)
- \big(\big|_{x=\partial} \omega_{-\lambda-x}\big([v_{\mu}w],u\big) \big)
\,.
\end{align*}
For the last equality we used the identification
$M[\lambda_1,\lambda_2,\lambda_3]/\langle\partial+\lambda_1+\lambda_2+\lambda_3\rangle
\simeq M[\lambda,\mu]$
by setting $\lambda_1=\lambda$, $\lambda_2=\mu$ and $\lambda_3=-\lambda-\mu-\partial$,
and we set $\varphi_{\lambda,-\lambda-\partial}=\omega_\lambda$.
We then observe that the right-hand side of the above equation
coincides with the left-hand side of \eqref{eq:2cocycle3} 
by the skewsymmetry \eqref{eq:2cocycle2} of $\omega_\lambda$.
This proves part (a).

For part (b), we can canonically identify 
$M[\lambda_1]/\langle\partial+\lambda_1\rangle\simeq M$
and, under this identification, a 1-cochain $\psi\in C^1(E,M)$ becomes a map
$\psi:\,E\to M$ and $(d\psi)_\lambda(u,v)$.
Moreover, by the definition \eqref{eq:differential} of the differential $d$,
\begin{align*}
& (d\varphi)_{\lambda_1,\lambda_2}(u,v)
=
u_{\lambda_1}(\varphi_{\lambda_2}(v))
- v_{\lambda_2}(\varphi_{\lambda_1}(u)) 
- \varphi_{\lambda_1+\lambda_2}([u_{\lambda_1}v]) \\
&\qquad
=
\rho(u)_{\lambda}(\psi(v))
- \rho(v)^*_{\lambda}(\psi(u)) 
- \psi([u_{\lambda}v])
\,.
\end{align*}
Again, for the last equality we identified
$M[\lambda_1,\lambda_2]/\langle\partial+\lambda_1+\lambda_2\rangle\simeq M[\lambda]$
by setting $\lambda_1=\lambda$, $\lambda_2=-\lambda-\partial$,
and we set $\varphi_{-\partial}=\psi$.
The right-hand side above coincides, up to a sign,
with the right-hand side of \eqref{equiv-ext}.
\end{proof}
\begin{corollary}\label{cor:ext}
The second cohomology $H^2(E,M)$ classifies $A[\partial]$-split abelian extensions
of the LCAd $E$ by the $E$-module $M$, up to trivial extensions.
\end{corollary}
\begin{proof}
Obvious.
\end{proof}

\section{Lie conformal algebroids and Poisson vertex algebras}\label{sec:5}

\subsection{Definition of PVA, PVA module, and the PVA $S_A(E)$ of an LCAd}\label{sec:5.1}

Recall that a \emph{Poisson vertex algebra} (PVA) is a differential algebra $\mc V$
endowed with an LCA $\lambda$-bracket $\{\cdot\,_\lambda\,\cdot\}:\,\mc V\times\mc V\to\mc V[\lambda]$
satisfying the Leibniz rule ($f,g,h\in\mc V$):
\begin{equation}\label{eq:leib}
\{f_\lambda gh\}=g\{f_\lambda h\}+h\{f_\lambda g\}\,,
\end{equation}
or, equivalently, the \emph{right} Leibniz rule:
\begin{equation}\label{eq:leibr}
\{fg_\lambda h\}
=
\{f_{\lambda+x} h\}\big(\big|_{x=\partial}g\big)
+
\{g_{\lambda+x}h\}\big(\big|_{x=\partial}f\big)\,.
\end{equation}

For example, if $R$ is an LCA, then the symmetric algebra $S(R)$
acquires automatically the structure of a PVA,
where $\partial$ is obtained by extending, uniquely, the endomorphism $\partial$ of the LCA $R$
to a derivation of the commutative associative product of $S(R)$,
and the $\lambda$-bracket is obtained by extending, uniquely, 
the $\lambda$-bracket $[\cdot\,_\lambda\,\cdot]$ of the LCA $R$ to $S(R)$
via the Leibniz rules \eqref{eq:leib} and \eqref{eq:leibr}.

This construction can be generalized to an arbitrary LCAd $E$, over a differential algebra $A$.
In general, given a differential algebra $A$ and a left $A[\partial]$-module $E$,
we can construct the symmetric algebra of $E$ over $A$,
$$
S_A(E)
=
A\oplus E\oplus S^2_A(E)\oplus\dots\,,
$$
quotient of the tensor algebra $\mc T_A(E)=\oplus_{n\geq0}(E\otimes_A\dots\otimes_A E)$
by the two-sided ideal generated by the relations
$u\otimes_A v-v\otimes_Au$, for $u,v\in E$.
It is a commutative associative algebra over $A$,
where the commutative associative product is induced by the tensor product (over $A$)
of the tensor algebra $\mc T_A(E)$:
$$
u_1\dots u_n
=
[u_1\otimes_A\dots\otimes_A u_n]\in S_A(E)
\,,
$$
for $u_1,\dots,u_n\in E$.
The action of $\partial$ on $E$ uniquely extends to a derivation of $S_A(E)$:
$$
\partial(a u_1\dots u_n)
=
(\partial a) u_1\dots u_n 
+
\sum_{i=1}^n
a\, u_1\dots(\partial u_i)\dots u_n
\,.
$$
In particular, $S_A(E)$ acquires the structure of a left $A[\partial]$-module,
where $A$ acts by multiplication.
\begin{theorem}\label{thm:SE}
Let $E$ be an LCAd over the differential algebra $A$.
Then, the symmetric algebra $S_A(E)$ of $E$ over $A$,
has a natural structure of a PVA.
The $\lambda$-bracket of $S_A(E)$
is obtained by letting ($u,v\in E,\,a,b\in A$):
\begin{equation}\label{eq:lambda-vvv}
\{a_\lambda b\}=0
\,,\quad
\{u_\lambda a\} = \theta(u)_\lambda(a)
\,,\quad
\{a_\lambda u\} = -\theta(u)^*_\lambda(a)
\,,\quad
\{u_\lambda v\} = [u_\lambda v]
\,,
\end{equation}
and extending it (uniquely) to a $\lambda$-bracket of $S_A(E)$
via the Leibniz rules \eqref{eq:leib} and \eqref{eq:leibr}.
\end{theorem}
\begin{proof}
The Lebniz rule for $\{u_\lambda ab\}$, as well as the right Leibniz rule for $\{ab_\lambda u\}$,
with $a,b\in A$ and $u\in E$,
agrees with the assumptions that $\theta(u)_\lambda$ is a left conformal derivation of $A$
and $\theta(u)^*_\lambda$ is a right conformal derivation of $A$,
as assumed in the Definition \ref{def:LCAd} of an LCAd.
Moreover, 
the Lebniz rule $\{u_\lambda av\}$ and the right Leibniz rule for $\{au_\lambda v\}$,
with $a\in A$ and $u,v\in E$,
agree with axioms (i) and (i') of the Definition \ref{def:LCAd} of an LCAd.
It is then immediate to check that the above formulas define a well-defined 
skewsymmetric $\lambda$-bracket on $S_A(E)$,
satisfying, by construction, the left and right Leibniz rules.
In order to prove the Jacobi identity (LCA-iii) in Section \ref{sec:2.10}, 
it then sufficies to prove that it holds on a triple of elements from $A$ or $E$.
For a triple $a,b,c\in A$, it is trivial as all three terms of the Jacobi identity vanish;
for a triple $u,v,w\in E$, it holds since $E$ is an LCA by assumption;
for a triple $a,b,u$, with $a,b\in A$ and $u\in E$, again all three terms vanish,
as $\{a_\lambda b\}=0$ by construction,
while $\{a_\lambda\{b_\mu u\}\}=-\{a_\lambda(\theta(u)^*_\mu(b))\}=0$ for the same reason,
as $\theta(u)_\lambda(b)$ has coefficients in $A$;
for a triple $a,u,v$, with $a\in A$ and $u,v\in E$, we have
\begin{align*}
& \{u_\lambda\{v_\mu a\}\}-\{v_\mu\{u_\lambda a\}\}-\{\{u_\lambda v\}_{\lambda+\mu}a\} \\
& =
\theta(u)_\lambda(\theta(v)_\mu(a))-\theta(v)_\mu(\theta(u)_\lambda(a))
-\theta([u_\lambda v])_{\lambda+\mu}(a)
=0
\,,
\end{align*}
by axiom (ii) of the Definition \ref{def:LCAd} of an LCAd.
\end{proof}

Recall also, see e.g. \cite{BDK}, that a \emph{module} $M$ over a PVA $\mc V$
is a left $\mb F[\partial]$-module endowed with
an action of the differential algebra $\mc V$, denoted by $f\cdot m$ for $f\in\mc V$ and $m\in M$,
(by this we mean that $\partial(f\cdot m)=(\partial f)\cdot m+f\cdot(\partial m)$),
and a $\lambda$-action of the LCA $\mc V$, denoted by $f_\lambda m\in M[\lambda]$,
such that the following Leibniz rules hold:
\begin{align}
& f_\lambda(g\cdot m)=\{f_\lambda g\}\cdot m+g\cdot(f_\lambda m)
\,, \label{eq:Leibniz1}\\
& 
(fg)_\lambda m
=\big(\big|_{x=\partial}f\big)\cdot(g_{\lambda+x}m)
+\big(\big|_{x=\partial}g\big)\cdot(f_{\lambda+x}m)
\,.\label{eq:Leibniz2}
\end{align}
For example, the \emph{adjoint module} is $M=\mc V$,
with the action $\cdot$ given by the commutative associative product of $\mc V$,
and the $\lambda$-action given by the $\lambda$-bracket of $\mc V$.

Recall that a \emph{derivation} from a commutative associative algebra $\mc V$ to a $\mc V$-module $M$
is a linear map $\delta:\,\mc V\to M$ satisfying ($f,g\in\mc V$):
$$
\delta(fg)=f\cdot\delta(g)+g\cdot\delta(f)
\,.
$$
Recall also from Section \ref{sec:2.15} that if $\mc V$ is a finitely generated differential algebra 
and $M$ is a finitely generated left module over $\mc V[\partial]$,
then we have the associated Gauge LCAd $\mc G(\mc V,M)$ with the LCA $\lambda$-bracket \eqref{eq:gauge-lambda}
and the anchor map \eqref{eq:gauge-anchor}.

The following statement is the analogue of Proposition \ref{prop:LCAd-mod} for PVA-modules.
\begin{proposition}\label{prop:PVAmod}
Let $\mc V$ be a PVA
and let $M$ be a left module over the associative algebra $\mc V[\partial]$.
Assume that $\mc V$ is finitely generated as a differential algebra,
and that $M$ is finitely generated as a $\mc V[\partial]$-module.
Then, 
there is a bijective correspondence between:
\begin{enumerate}[(a)]
\item
the $\lambda$-actions of $\mc V$ on $M$ making $M$ a PVA module over $\mc V$;
\item
the maps $\rho:\,\mc V\to\mc G(\mc V,M)$, mapping $f\mapsto\rho(f)=(\phi(f),\sigma(f))$
with $\phi(f)\in\CEnd(M)$ and $\sigma(f):=\{f_\lambda\,\cdot\}\in\CDer(\mc V)$,
such that $\rho$ is both a derivation from $\mc V$ to $\mc G(\mc V,M)$ commuting with $\partial$, and an LCA homomorphism.
\end{enumerate}
This correspondence associates to a PVA $\lambda$-action $f_\lambda m$ of $\mc V$ on $M$
the map $\rho=(\phi,\sigma)$ given by ($f\in\mc V$, $m\in M$)
\begin{equation}\label{eq:action3}
\phi(f)_\lambda(m)=f_\lambda m
\,.
\end{equation}
\end{proposition}
\begin{proof}
In terms of the map $\phi$ in \eqref{eq:action3}, the sesquilinearity condition $f_\lambda(\partial m)=(\lambda+\partial)(f_\lambda m)$, $f\in\mc V,m\in M$, 
of the $\lambda$-action  of $\mc V$ on $M$, means that $\phi(f)\in\CEnd(M)$. 
Furthermore, using the Leibniz rule \eqref{eq:Leibniz1} and PVA axioms, one shows that the pair $(\phi(f),\sigma(f))\in \mc G(\mc V,M)$, cf. \eqref{eq:derM}.
Recalling the left $\mc V[\partial]$-module structure of $\CEnd(M)$ given by \eqref{eq:lcder-mod}, in terms of the map $\phi$ in \eqref{eq:action3}, 
the Leibniz rule \eqref{eq:Leibniz2} means that $\phi(fg)=f\phi(g)+g\phi(f)$, for every $f,g\in\mc V$. Similarly, using the PVA axioms 
we have that $\sigma(fg)=f\sigma(g)+g\sigma(f)$. 
Hence, the map $\rho=(\phi,\sigma)$ is a derivation from $\mc V$ to $\mc G(\mc V,M)$. 
Moreover, the sesquilinearity condition $(\partial f)_\lambda m=-\lambda f_{\lambda}m$ and equation \eqref{eq:lcder-mod} 
imply that $\phi$ commutes with $\partial$. Similarly, the PVA sesquilinearity implies that also $\sigma$ commutes with $\partial$. 
Thus the map $\rho=(\phi,\sigma)$ commutes with $\partial$.
Finally, the fact that $\rho$ is an LCA homomorphism follows by the Jacobi identity of the $\lambda$-action of $\mc V$ on $M$ 
and by the Jacobi identity of the PVA $\mc V$.
Conversely, given a map $\rho$ as in part (b), the same arguments as above show that \eqref{eq:action3},
read from right to left, defines a $\lambda$-action of $\mc V$ on $M$.
\end{proof}
\begin{remark}\label{20251103:rem1}
If we remove the finitely generatedness assumptions the same statement still holds,
except that $\mc G(\mc V,M)$ is not, strictly speaking, an LCA as the $\lambda$-bracket may have values in power series,
hence the map $\rho$, which maps the $\lambda$-bracket of $\mc V$ to the $\lambda$-bracket of $\mc G(\mc V,M)$,
will not be, strictly speaking, an LCA homomorphism,
cf. Remark \ref{rem:LCAd-mod}.
\end{remark}

\subsection{Space of K\"{a}hler differentials of a commutative associative algebra 
and of a differential algebra}\label{sec:5.2b}

Let $A$ be a commutative associative algebra.
Recall that \cite{J69,H77}
the \emph{space of K\"{a}hler differentals} $\Omega(A)$ of $A$
is the unique $A$-module endowed with a derivation $d:\,A\to\Omega(A)$
satisfying the following universal property:
for every $A$-module $M$ and every derivation $\delta:\,A\to M$,
there exists a unique $A$-module homomorphism 
$\widetilde{\delta}:\,\Omega(A)\to M$ such that the following diagram is commutative:
\begin{equation}\label{eq:diagram-delta}
\xymatrix{
\Omega(A) \ar@{.>}@/^1pc/[dr]^{\exists!\,\,\widetilde{\delta}} \\
A \ar[r]^{\forall\,\,\delta} \ar[u]^{d} & M 
}
\end{equation}
This universal property can be expressed by the isomorphism of vector spaces,
associated to an $A$-module $M$,
\begin{equation}\label{eq:univ-Kahlera}
\Der(A,M)
\simeq
\Hom_A(\Omega(A),M)
\,,
\end{equation}
where $\Der(A,M)$ denotes the space of all derivations $\delta:\, A\to M$,
while $\Hom_A(\Omega(A),M)$ denotes the space of all $A$-module homomorphisms
$\widetilde{\delta}:\,\Omega(A)\to M$.

The space $\Omega(A)$ of K\"{a}hler differentials can be constructed explicitly 
as the free $A$-module generated by the elements $da,\,a\in A$,
quotient by the $A$-submodule generated by the relations
($k\in\mb F,\,a,b\in A$):
\begin{equation}\label{eq:defOmega}
d(ka)-kd(a)
\,,\quad
d(a+b)-d(a)-d(b)
\,,\quad
d(ab)-ad(b)-bd(a)
\,.
\end{equation}

There is yet a third construction of the space of K\"{a}hler differentials:
one can prove that there is an isomorphism of $A$-modules
$\Omega(A)\simeq I/I^2$, where $I$ is the kernel of the multiplication map $A\otimes A\to A$,
and $I^2$ is the submodule generated by products $ij$, for $i,j\in I$.
Under this isomorphism, the differential map $d:\,A\to\Omega(A)$ maps
$a\mapsto [a\otimes1-1\otimes a]_{I^2}$.

For example, if $A=\mb F[x_1,\dots,x_n]$ is the algebra of polynomials in $n$ variables,
then $\Omega(A)=\bigoplus_{i=1}^n\mb F[x_1,\dots,x_n]dx_i$,
with the differential map $d$ 
given by 
$$
d(p)=\sum_{i=1}^n\frac{\partial p}{\partial x_i}dx_i
\,.
$$

If $A$ is a differential algebra with derivation $\partial$,
then the space of K\"{a}hler differentials $\Omega(A)$ is endowed with the structure 
of a left $A[\partial]$-module, with the action of $\partial$ on $\Omega(A)$
uniquely determined by the commutativity of the diagram
$$
\xymatrix{
\Omega(A) \ar[r]^{\partial} & \Omega(A) \\
A \ar[r]^{\partial} \ar[u]^{d} & A \ar[u]^{d}
}
$$
In particular, we have, for $a,b\in A$,
\begin{equation}\label{eq:kahler-Ad}
\partial(ad(b))=(\partial a)d(b)+ad(\partial b)
\,.
\end{equation}

For example, consider the algebra of differential polynomials in $m$ variables
$A=\mb F\big[x_i^{(n)},\,i=1,\dots,m,\,n\in\mb Z_+\big]$.
Then $\Omega(A)=\bigoplus_{1\leq j\leq m ,\, \ell\in\mb Z_+}
\mb F\big[x_i^{(n)},\,i=1,\dots,m,\,n\in\mb Z_+\big]dx_j^{(\ell)}$,
with the differential map $d$ 
given by 
$$
d(p)=\sum_{i=1}^m\sum_{n=0}^\infty\frac{\partial p}{\partial x_i^{(n)}}dx_i^{(n)}
\,.
$$
In this case the derivation $\partial$ acts on the space $\Omega(A)$ by
$$
\partial\Big(\sum_{i=1}^m\sum_{n=0}^\infty f_{i,n} dx_i^{(n)}\Big)
=
\sum_{i=1}^m\sum_{n=0}^\infty (\partial f_{i,n}) dx_i^{(n)}
+
\sum_{i=1}^m\sum_{n=0}^\infty f_{i,n} dx_i^{(n+1)}
\,.
$$

In the case when $A$ is a differential algebra,
there are two different versions of the universal property \eqref{eq:univ-Kahlera},
one in terms of derivations commuting with $\partial$, stated in Proposition \ref{prop:univ-prop1}, 
another in terms of conformal derivations, stated in Proposition \ref{prop:univ-prop2}.
\begin{proposition}\label{prop:univ-prop1}
For every left $A[\partial]$-module $M$ 
and every derivation $\delta:\,A\to M$ commuting with $\partial$,
there exists a unique homomorphism of left $A[\partial]$-modules
$\widetilde{\delta}:\,\Omega(A)\to M$ making the diagram \eqref{eq:diagram-delta} commutative.
Equivalently, there is a canonical vector space isomorphism,
associated to a left $A[\partial]$-module $M$,
\begin{equation}\label{eq:univ-Kahler-d}
\Der^\partial(A,M)
\simeq
\Hom_{A[\partial]}(\Omega(A),M)
\,,
\end{equation}
where $\Der^\partial(A,M)\subset\Der(A,M)$ is the subspace of derivations from $A$ to $M$ commuting with $\partial$, 
while $\Hom_{A[\partial]}(\Omega(A),M)\subset\Hom_A(\Omega(A),M)$
is the subspace of left $A[\partial]$-module homomorphisms from $\Omega(A)$ to $M$.
\end{proposition}
\begin{proof}
When $A$ is a differential algebra,
both $\Der(A,M)$ and $\Hom_A(\Omega(A),M)$ are $\mb F[\partial]$-modules,
with the adjoint action of $\partial$:
$$
(\partial\delta)(a)=\partial(\delta(a))-\delta(\partial(a))
\,,\qquad
(\partial\phi)(u)=\partial(\phi(u))-\phi(\partial(u))
\,,
$$
for $\delta\in\Der(A,M)$, $\phi\in\Hom_A(\Omega(A),M)$, $a\in A$ and $u\in\Omega(A)$.
Then, it is immediate to check that the isomorphism \eqref{eq:univ-Kahlera}, mapping $\delta\mapsto\widetilde{\delta}$,
is an $\mb F[\partial]$-module isomorphism.
As a corollary, it restricts to the isomorphism \eqref{eq:univ-Kahler-d} of the kernels of the action of $\partial$
on both spaces.
\end{proof}

Recall that a conformal derivation from $A$ to a left $A[\partial]$-module $M$ is a linear map
$\delta_\lambda:\,A\to M[\lambda]$ such that \eqref{eq:lcder} holds.
Recall also that a right conformal derivation from $A$ to a left $A[\partial]$-module $M$ is a linear map
$\eta_\lambda:\,A\to M[\lambda]$ satisfying \eqref{eq:rcder}.
Then, the space of K\"{a}hler differentials $\Omega(A)$
satisfies the following universal properties.
\begin{proposition}\label{prop:univ-prop2}
\begin{enumerate}[(a)]
\item
For every left $A[\partial]$-module $M$ 
and every left conformal derivation $\delta_\lambda:\,A\to M[\lambda]$,
there exists a unique conformal homomorphism of left $A[\partial]$-modules
$\widetilde{\delta}_\lambda:\,\Omega(A)\to M[\lambda]$ such that the following diagram is commutative:
$$
\xymatrix{
\Omega(A) \ar@{.>}@/^1pc/[dr]^{\exists!\,\,\widetilde{\delta}_\lambda} \\
A \ar[r]^{\forall\,\delta_\lambda} \ar[u]^{d} & M[\lambda] 
}
$$
Equivalently, we have a left $A[\partial]$-module isomorphism,
associated to a left $A[\partial]$-module $M$,
\begin{equation}\label{eq:univ-Kahlerb}
\CDer(A,M)
\simeq
\CHom_A(\Omega(A),M)
\,,
\end{equation}
where $\CDer(A,M)$ and $\CHom_A(\Omega(A),M)$
were defined in Section \ref{sec:2.10}.
\item
Similarly, for every left $A[\partial]$-module $M$ 
and every right conformal derivation $\eta_\lambda:\,A\to M[\lambda]$,
there exists a unique right conformal homomorphism of left $A[\partial]$-modules
$\widetilde{\eta}_\lambda:\,\Omega(A)\to M[\lambda]$ such that the following diagram is commutative:
$$
\xymatrix{
\Omega(A) \ar@{.>}@/^1pc/[dr]^{\exists!\,\,\widetilde{\eta}_\lambda} \\
A \ar[r]^{\forall\,\eta_\lambda} \ar[u]^{d} & M[\lambda] 
}
$$
Equivalently, we have a left $A[\partial]$-module isomorphism,
associated to a left $A[\partial]$-module $M$,
\begin{equation}\label{eq:univ-Kahlerc}
\RCDer(A,M)
\simeq
\RCHom_A(\Omega(A),M)
\,,
\end{equation}
where $\RCDer(A,M)$ and $\RCHom_A(\Omega(A),M)$ were defined in Section \ref{sec:2.10}.
\end{enumerate}
\end{proposition}
\begin{proof}
If $\delta_\lambda:\,A\to M[\lambda]$ is a conformal derivation,
then in particular it is a derivation;
hence, by the universal property of $\Omega(A)$ there is a homomorphism
of $A$-modules $\widetilde{\delta}_\lambda:\,A\to M[\lambda]$
such that $\widetilde{\delta}_\lambda\circ d=\delta_\lambda$.
In order to prove part (a), we only need to show that,
if $\delta_\lambda$ is a conformal linear map
then $\widetilde{\delta}_\lambda\in\CHom_A(\Omega(A),M)$, cf. Section \ref{sec:2.10}.
Indeed, we have
\begin{align*}
& \widetilde{\delta}_\lambda(p(\partial)(ad(b)))
=
\widetilde{\delta}_\lambda\Big(p(x+y)\big(\big|_{x=\partial}a\big)d\big(\big|_{y=\partial}b\big)\Big)
=
p(x+y)\big(\big|_{x=\partial}a\big)
\delta_\lambda\big(\big|_{y=\partial}b\big) \\
& =
p(x+y)\big(\big|_{x=\partial}a\big)
\big(\big|_{y=\lambda+\partial}\delta_\lambda(b)\big)
=
p(\lambda+\partial)\big(a\delta_\lambda(b)\big)
=
p(\lambda+\partial)
\widetilde{\delta}_\lambda(ad(b))
\,.
\end{align*}
Claim (b) follows from part (a).
Indeed, if $\eta_\lambda:\,A\to M[\lambda]$ is a right conformal derivation,
then $\eta_\lambda^*:\,A\to M[\lambda]$
is a left conformal derivation, hence by part (a)
there exists a unique $\widetilde{\eta^*}_\lambda\in\CHom_A(\Omega(A),M)$
such that $\widetilde{\eta^*}_\lambda\circ d=\eta^*_\lambda$,
and then, applying the map \eqref{eq:phi-dual},
we obtain $\widetilde{\eta_\lambda}=(\widetilde{\eta^*})^*_\lambda$, 
which is a right conformal homomorphism of left $A[\partial]$-modules
from $\Omega(A)$ to $M$, as required.
\end{proof}

\subsection{LCAd structure on the space of K\"{a}hler differentials of a PVA}\label{sec:5.2a}

\begin{theorem}\label{thm:kahlerLCAd}
If $\mc V$ is a PVA, then the space of K\"{a}hler differentials $\Omega(\mc V)$
has a natural structure of an LCAd,
with the left $\mc V[\partial]$-module structure given by \eqref{eq:kahler-Ad},
the anchor map $\theta:\,\Omega(\mc V)\to\CDer(\mc V)$ defined on the generators via
\begin{equation}\label{eq:kahler-theta}
\theta(df)_\lambda(g)=\{f_\lambda g\}
\,,
\end{equation}
and uniquely extended as a left $\mc V[\partial]$-module homomorphism $\Omega(\mc V)\to\CDer(\mc V)$,
and the $\lambda$-bracket $[\cdot\,_\lambda\,\cdot]:\,\Omega(\mc V)^{\otimes2}\to\Omega(\mc V)[\lambda]$
defined on the generators via
\begin{equation}\label{eq:kahler-lambda}
[df_\lambda dg]=d\big(\{f_\lambda g\}\big)
\,,
\end{equation}
and uniquely extended to a $\lambda$-bracket of $\Omega(\mc V)$ satisfying
the axioms (i) and (i') of the Definition \ref{def:LCAd} of an LCAd.
\end{theorem}
In fact, we can write an explicit formula for the anchor map (cf. \eqref{eq:lcder-mod})
\begin{equation}\label{eq:kahler-anchor}
\theta(adf)_\lambda(g)
=
\big(\big|_{x=\partial}a\big)
\{f_{\lambda+x} g\}
\,,
\end{equation}
and for the $\lambda$-bracket (cf. \eqref{eq:tot-form}):
\begin{equation}\label{eq:kahler-tot-form}
[ad(f)_\lambda bd(g)]
=
\big(\big|_{x=\partial}a\big)
b d(\{f_{\lambda+x}g\})
+
\big(\big|_{x=\partial}a\big) 
\{f_{\lambda+x} b\} d(g)
+
b \{a_{\lambda+x}g\}
d\big(\big|_{x=\partial}f\big)
\,.
\end{equation}
\begin{proof}
By construction, $\theta:\,\Omega(\mc V)\to\CDer(\mc V)$ is a left $\mc V[\partial]$-module homomorphism,
and the $\lambda$-bracket $[\cdot\,_\lambda\,\cdot]:\,\Omega(\mc V)^{\otimes2}\to\Omega(\mc V)[\lambda]$
satisfies conditions (i) and (i') of Definition \ref{def:LCAd}.
Then, thanks to Lemma \ref{lem:gener}, we only need to verify the 
LCAd axioms on the generators $d(f),\,f\in\mc V$, of $\Omega(\mc V)$
(as left $\mc V[\partial]$-module).
The skewsymmetry of the $\lambda$-bracket, for a pair of elements of the form $df, dg$,
with $f,g\in\mc V$, reduces to the skewsymmetry of the $\lambda$-bracket 
$\{\cdot\,_\lambda\,\cdot\}$ of $\mc V$, since $d$ and $\partial$ commute.
Condition (ii) of Definition \ref{def:LCAd}, for $\theta([df_\lambda dg])_\mu(h)$, with $f,g,h\in\mc V$,
reduces to the Jacobi identity of the $\lambda$-bracket $\{\cdot\,_\lambda\,\cdot\}$ of $\mc V$.
Moreover, the Jacobi identity of the $\lambda$-bracket, for a triple of elements of the form $df, dg, dh$,
with $f,g,h\in\mc V$, reduces again to the Jacobi identity of the $\lambda$-bracket 
$\{\cdot\,_\lambda\,\cdot\}$ of $\mc V$.
The claim follows.
\end{proof}
\begin{remark}\label{rem:kahlerPA}
A similar result holds for Poisson algebras and Lie algebroids:
if $P$ is a Poisson algebra,
then the space of K\"{a}hler differentials $\Omega(P)$ has a natural structure of a Lie algebroid, \cite{H90}.
In this case the formulas for the anchor map and for the Lie bracket are similar
to \eqref{eq:kahler-anchor} and \eqref{eq:kahler-tot-form}, except that there is no $\lambda$ (and $x$).
\end{remark}

\begin{remark}\label{rem:jetPVA}
Let $P$ be a Poisson algebra, and consider the LAd structure 
on the space $\Omega(P)$ of K\"{a}hler differentials.
Recall from Section \ref{sec:2.2b} that 
$\widehat{\Omega(P)}=\widehat{P}[\partial]\otimes_P\Omega(P)$
is an LCAd over the differential algebra $\widehat{P}=J_\infty P$.
On the other hand, the jet algebra $\widehat{P}=J_\infty P$ is naturally a PVA, \cite{A12},
and we can consider the corresponding LCAd structure on the space $\Omega(\widehat P)$
of K\"{a}hler differentials, given by Theorem \ref{thm:kahlerLCAd}.
It is not hard to check that the resulting LCAd's are isomorphic:
$\widehat{\Omega(P)}\simeq\Omega(\widehat{P})$.
\end{remark}

\subsection{PVA modules over $\mc V$ and LCAd modules over $\Omega(\mc V)$}\label{sec:5.3}

\begin{proposition}\label{prop:PVA-LCAd-modules}
Let $\mc V$ be a PVA, and consider the corresponding LCAd $\Omega(\mc V)$.
Let $M$ be a left $\mc V[\partial]$-module.
There is a bijective correspondence between:
\begin{enumerate}[(a)]
\item
the $\lambda$-actions of $\mc V$ on $M$ making $M$ a PVA module over $\mc V$;
\item
the $\lambda$-actions of $\Omega(\mc V)$ on $M$ making $M$ an LCAd module over $\Omega(\mc V)$.
\end{enumerate}
This correspondence associates to a PVA $\lambda$-action $f_\lambda m$ of $\mc V$ on $M$
the LCAd $\lambda$-action of $\Omega(\mc V)$ on $M$ given by
\begin{equation}\label{eq:action4}
(fdg)_\lambda m = f\cdot (g_\lambda m)
\,.
\end{equation}
\end{proposition}
\begin{proof}
By Proposition \ref{prop:PVAmod}, to give a PVA $\lambda$-action of $\mc V$ on $M$ is the same as to give a map
$\rho:\mc V\to \mc G(\mc V,M)$ of the form $\rho=(\phi,\sigma)$ with $\phi:\,\mc V\to\CEnd(M)$ and $\sigma(f)=\{f_\lambda\,\cdot\}$, 
which is both a derivation from $\mc V$ to $\mc G(\mc V,M)$ commuting with $\partial$, and an LCA homomorphism 
(In the case when the differential algebra $\mc V$ or the $\mc V[\partial]$-module $M$
are not finitely generated this is not strictly speaking an LCA homomorphism, but the same arguments still holds, see Remark \ref{20251103:rem1}.)
By Proposition \ref{prop:univ-prop1} there exists a unique homomorphism of left $\mc V[\partial]$-modules 
$\widetilde{\rho}:\Omega(\mc V)\to\mc G(\mc V,M)$
such that $\widetilde{\rho}(gdf)=g\rho(f)$, $f,g\in\mc V$. 
Recalling the definition of the LCA structure of $\Omega(\mc V)$ given by \eqref{eq:kahler-lambda} 
and using the fact that $\rho$ is an LCA homomorphism we then get ($f,g\in\mc V$)
$$
\widetilde{\rho}([df_\lambda dg])=\widetilde{\rho}(d\{f_{\lambda} g\})=\rho(\{f_\lambda g\})
=[\rho(f)_\lambda\rho(g)]=[\widetilde{\rho}(df)_\lambda\widetilde{\rho}(dg)]\,.
$$
Moreover, recalling the anchor maps \eqref{eq:gauge-anchor} of $\mc G(\mc V,M)$ and \eqref{eq:kahler-theta} of $\Omega(\mc V)$, 
as well as the definition of $\sigma$ above, we have ($f,g\in\mc V)$
$$
\theta_{\mc G(\mc V,M)}\left(\widetilde{\rho}(df)\right)_\lambda g
=
\theta_{\mc G(\mc V,M)}(\phi(f),\sigma(f))_\lambda g
=
\sigma(f)_\lambda g=\{f_\lambda g\}=\theta_{\Omega(\mc V)}(df)_\lambda g\,.
$$
Hence, the map $\widetilde{\rho}:\Omega(\mc V)\to \mc G(\mc V,M)$ is a homomorphism of LCAd, which, 
by Proposition \ref{prop:LCAd-mod}  corresponds to a $\lambda$-action of $\Omega(\mc V)$ on $M$.
Conversely, given a $\lambda$-action of $\Omega(\mc V)$ on $M$,
by the same argument as above, we get a $\lambda$-action of $\mc V$ on $M$
defined by $f_\lambda m:=(df)_{\lambda}m$, $f\in\mc V$, $m\in M$.
\end{proof}

\section{Relation between Poisson vertex algebra cohomology
and Lie conformal algebroid cohomology}\label{sec:6}

We review in Section \ref{sec:6.1} the definition of the so called variational Poisson cohomology complex
and, in the following Section \ref{sec:6.2}, 
the definitions of the basic and reduced PVA cohomology complexes,
following \cite{DK11a,DK11b,BDHKV,BDK}.
Then, in Section \ref{sec:6.3}, we relate all these complexes to the corresponding LCAd cohomology
complexes defined in Section \ref{sec:3}.

\subsection{Variational PVA cohomology complex}\label{sec:6.1}

The \emph{variational PVA cohomology complex} $(\Gamma^\bullet(\mc V,M),d)$ associated to a PVA $\mc V$
and a PVA module $M$ is defined as follows.
We let 
\begin{equation}\label{eq:var-compl}
\Gamma^\bullet(\mc V,M)=\bigoplus_{k\geq0}\Gamma^k(\mc V,M)
\,,
\end{equation}
where $\Gamma^k(\mc V,M)$ is the space of $k$-\emph{cochains},
i.e. linear maps 
\begin{equation}\label{eq:k-cochainsPVA}
\gamma_{\lambda_1,\dots,\lambda_k}:\,\mc V^{\otimes k}\to
M[\lambda_1,\dots,\lambda_k]/\langle\partial+\lambda_1+\dots+\lambda_k\rangle
\,,
\end{equation}
mapping $a_1\otimes\dots\otimes a_k\mapsto\gamma_{\lambda_1,\dots,\lambda_k}(a_1,\dots,a_k)$,
satisfying the sesquilinearity conditions
\begin{equation}\label{eq:sesqPVA}
\gamma_{\lambda_1,\dots,\lambda_k}(a_1,\dots,(\partial a_i),\dots,a_k)
=
-\lambda_i\gamma_{\lambda_1,\dots,\lambda_k}(a_1,\dots, a_k)
\,,
\end{equation}
for all $i=1,\dots,k$,
the Leibniz rules
\begin{align}
& 
\gamma_{\lambda_1,\dots,\lambda_k} (a_1,\dots,b_ic_i,\dots,a_k) 
\label{eq:leibPVA}
=
\big(\big|_{x=\partial}b_i\big)
\gamma_{\lambda_1,\dots,\lambda_i+x,\dots,\lambda_k}
(a_1,\cdots,c_i,\dots,a_k) \\
\notag
&\qquad\qquad +
\big(\big|_{x=\partial}c_i\big)
\gamma_{\lambda_1,\dots,\lambda_i+x,\dots,\lambda_k}
(a_1,\dots,b_i,\dots,a_k)
\,,
\end{align}
for all $i=1,\dots,k$, and the skewsymmetry
\begin{equation}\label{eq:poly-skewPVA}
\gamma_{\lambda_1,\dots,\lambda_k}\big(a_1,\dots,a_k)
=
\sign(\pi)
\gamma_{\lambda_{\pi(1)},\dots,\lambda_{\pi(k)}}(a_{\pi(1)},\dots,a_{\pi(k)})
\,,
\end{equation}
for every permutation $\pi\in S_k$.
For example
$\Gamma^0(\mc V,M)=M/\partial M$,
while 
$\Gamma^1(\mc V,M)=\Der^\partial(\mc V,M)$ is the space of derivations from $\mc V$ to $M$ commuting with $\partial$.

The space $\Gamma^\bullet(\mc V,M)$ of cochains 
is a cohomology complex with the differential 
$d:\,\Gamma^k(\mc V,M)\to\Gamma^{k+1}(\mc V,M),\,k\geq0$, defined by
(cf. \eqref{eq:differential})
\begin{equation}\label{eq:differentialPVA}
\begin{split}
& 
(d\gamma)_{\lambda_1,\dots,\lambda_{k+1}}(a_1,\dots,a_{k+1})
=
\sum_{i=1}^{k+1}(-1)^{i+1}
{a_i}_{\lambda_i}
\big(
\gamma_{\lambda_1,\stackrel{i}{\check\dots},\lambda_{k+1}}(a_1,\stackrel{i}{\check\dots},a_{k+1})
\big) \\
&\qquad +
\sum_{\substack{i,j=1 \\ i<j}}^{k+1}(-1)^{i+j}
\gamma_{\lambda_i+\lambda_j,\lambda_1,\stackrel{i}{\check\dots}\stackrel{j}{\check\dots},\lambda_{k+1}}
\big(\{{a_i}_{\lambda_i}{a_j}\},a_1,\stackrel{i}{\check\dots}\stackrel{j}{\check\dots},a_{k+1}\big)
\,.
\end{split}
\end{equation}

It is proved in \cite{DK11a} that 
formula \eqref{eq:differentialPVA} defines a cohomology differential of $\Gamma^\bullet(\mc V,M)$.
We thus get the corresponding
\emph{variational PVA cohomology} of $\mc V$ with coefficients in $M$:
$$
H^k(\mc V,M)
:=
\ker\big(d\big|_{\Gamma^k(\mc V,M)}\big)\big/d\big(\Gamma^{k-1}(\mc V,M)\big)
\,.
$$
We refer to \cite[Sec.3.4]{BDK} for an interpretation of the low degree cohomology spaces.

\subsection{Basic and reduced PVA cohomologies}\label{sec:6.2}

If we do not quotient $M[\lambda_1,\dots,\lambda_k]$ 
by $\langle\partial+\lambda_1+\dots+\lambda_k\rangle$
in the definition \eqref{eq:k-cochainsPVA} of $k$-cochains,
we obtain the so-called \emph{basic} PVA cohomology complex of $\mc V$ with coefficients in $M$.

We thus let $\widetilde{\Gamma}^k(\mc V,M)$ be the vector space
consisting of all linear maps
\begin{equation}\label{eq:basicY}
\widetilde\gamma\colon \mc V^{\otimes k}\to M[\lambda_1,\dots,\lambda_k] \,,
\end{equation}
satisfying the sesquilinearity conditions \eqref{eq:sesqPVA},
the Leibniz rules \eqref{eq:leibPVA}
and the skewsymmetry conditions \eqref{eq:poly-skewPVA}
(where all the equations are now in the space $M[\lambda_1,\dots,\lambda_k]$).
Elements of $\widetilde{\Gamma}^k(\mc V,M)$ are called \emph{basic} $k$-\emph{cochains}.

It is proved in \cite{DK11a} that we have a cohomology differential
$\widetilde{d}:\,\widetilde{\Gamma}^k(\mc V,M)\to\widetilde{\Gamma}^{k+1}(\mc V,M)$
defined by the same formula as \eqref{eq:differentialPVA},
with $\widetilde{d}$ in place of $d$,
$\widetilde{\gamma}$ in place of $\gamma$,
where both sides are understood in $M[\lambda_1,\dots,\lambda_{k+1}]$.
We thus get the corresponding \emph{basic PVA cohomology} of $\mc V$ with coefficients in $M$
$$
\widetilde{H}^k(\mc V,M)
:=
\ker\big(\widetilde{d}\big|_{\widetilde{\Gamma}^k(\mc V,M)}\big)
\big/\widetilde{d}\big(\widetilde{\Gamma}^{k-1}(\mc V,M)\big)
\,.
$$

Moreover, we have an action of $\mb F[\partial]$ on $\widetilde{\Gamma}^k(\mc V,M)$ by letting
(cf. \eqref{eq:partial-basic})
\begin{equation}\label{eq:partial-basicPVA}
(\partial\widetilde\gamma)_{\lambda_1,\dots,\lambda_k}(a_1,\dots,a_k)
=
(\lambda_1+\dots+\lambda_k+\partial)
\widetilde\gamma_{\lambda_1,\dots,\lambda_k}(a_1,\dots,a_k)
\,.
\end{equation}
It is proved in \cite{DK11a} that
this action of $\partial$ commutes with the action of the differential:
$\widetilde d\circ\partial=\partial\circ\widetilde d$.
Hence, quotienting by the action of $\partial$ we get an induced cohomology complex,
$$
\overline{\Gamma}^\bullet(\mc V,M)=\widetilde{\Gamma}^\bullet(\mc V,M)/\partial\widetilde{\Gamma}^\bullet(\mc V,M)
\,,
$$
with the differential $\overline{d}:\,\overline{\Gamma}^k(\mc V,M)\to\overline{\Gamma}^{k+1}(\mc V,M)$
induced by $\widetilde{d}$, i.e. $\overline{d}\circ\pi=\pi\circ\widetilde{d}$,
where $\pi:\,\widetilde{\Gamma}^k(\mc V,M)\twoheadrightarrow
\widetilde{\Gamma}^k(\mc V,M)/\partial\widetilde{\Gamma}^k(\mc V,M)$
is the canonical quotient map.
Passing to cohomology we get the so called \emph{reduced PVA cohomology}
$$
\overline{H}^k(\mc V,M)
:=
\ker\big(\overline{d}\big|_{\overline{\Gamma}^k(\mc V,M)}\big)
\big/\overline{d}\big(\overline{\Gamma}^{k-1}(\mc V,M)\big)
\,.
$$

The following result relates basic, reduced and variational PVA complexes:
\begin{proposition}[{\cite{DK11a}, see also \cite[Lem.3.17]{BDK}}]\label{lem:Ytilde}
\begin{enumerate}[(a)]
\item
We have a well-defined homomoprphism of complexes
\begin{equation}\label{eq:piYtilde3}
\big(\widetilde{\Gamma}^\bullet(\mc V,M),\widetilde{d}\big)
\to 
\big({\Gamma}^\bullet(\mc V,M),d\big)
\,,\qquad
\widetilde \gamma\mapsto p\circ\widetilde \gamma
\,,
\end{equation}
where $p:\,M[\lambda_1,\dots,\lambda_k]\twoheadrightarrow
M[\lambda_1,\dots,\lambda_k]/\langle\partial+\lambda_1+\dots+\lambda_k\rangle$
is the canonical quotient map.
\item
The map \eqref{eq:piYtilde3} has kernel $\partial\widetilde{\Gamma}^\bullet(\mc V,M)$,
hence \eqref{eq:piYtilde3} induces an injective homomoprphism of complexes
$$
\big(\overline{\Gamma}^\bullet(\mc V,M),\overline{d}\big)\,\hookrightarrow\, \big({\Gamma}^\bullet(\mc V,M),d\big)
\,.
$$
\item
Suppose that, as a differential algebra, $\mc V$
is an algebra of differential polynomials in finitely many variables.
Then \eqref{eq:piYtilde3} is surjective at each degree.
Hence, we get an isomorphism of complexes
$$
\big(\overline{\Gamma}^\bullet(\mc V,M),\overline{d}\big)
\,\stackrel{\sim}{\longrightarrow}\,
\big({\Gamma}^\bullet(\mc V,M),d\big)
\,.
$$
\end{enumerate}
\end{proposition}

\subsection{Relation between PVA cohomology and LCAd cohomology}\label{sec:6.3}
Recall that Theorem \ref{thm:kahlerLCAd} associates to any PVA $\mc V$ an LCAd structure on the space of
K\"{a}hler differentials $\Omega(\mc V)$,
and, given a left $\mc V[\partial]$-module $M$,
Proposition \ref{prop:PVA-LCAd-modules} associates to any structure of a PVA module on $M$ (over $\mc V$),
a structure of an LCAd module on $M$ (over $\Omega(\mc V)$).
\begin{theorem}\label{thm:main}
Given a PVA $\mc V$ and a PVA module $M$ over $\mc V$, we have an isomorphism of complexes
between the PVA cohomology complex $(\Gamma^\bullet(\mc V,M),d_{\textrm{PVA}})$ defined in Section \ref{sec:6.1}
and the LCAd cohomology complex $(C^\bullet(\Omega(\mc V),M),d_{\textrm{LCAd}})$ defined in Section \ref{sec:3.1}:
$$
\Phi:(\Gamma^\bullet(\mc V,M),d_{\textrm{PVA}})\stackrel{\sim}{\longrightarrow} (C^\bullet(\Omega(\mc V),M),d_{\textrm{LCAd}})
\,.
$$
Explicitly, the isomorphism $\Phi$ maps the $k$-cochain $\gamma_{\lambda_1,\dots,\lambda_k}:\,\mc V^{\otimes k}\to
M[\lambda_1,\dots,\lambda_k]/\langle\partial+\lambda_1+\dots,\lambda_k\rangle$ to
the $k$-cochain $\Phi(\gamma)_{\lambda_1,\dots,\lambda_k}:\,\Omega(\mc V)^{\otimes k}\to 
M[\lambda_1,\dots,\lambda_k]/\langle\partial+\lambda_1+\dots+\lambda_k\rangle$ 
defined by
\begin{equation}\label{eq:Phi}
\Phi(\gamma)_{\lambda_1,\dots,\lambda_k}(a_1du_1,\dots,a_kdu_k)
=(|_{x_1=\partial}a_1)\dots(|_{x_k=\partial}a_k)\gamma_{\lambda_1+x_1,\dots,\lambda_k+x_k}(u_1,\dots,u_k)
\,,
\end{equation}
for every $a_1,u_1,\dots,a_k,u_k\in\mc V$.
\end{theorem}
\begin{proof}
First, note that
$\Phi(\gamma)$ defined by \eqref{eq:Phi} is compatible with the defining relations
\eqref{eq:defOmega} of
$\Omega(\mc V)$: for the first two relations in
\eqref{eq:defOmega} it follows by linearity, while for the third relation it is a consequence of the fact that $\gamma$ satisfies the Leibniz rules \eqref{eq:leibPVA}.
Note also that if 
$\gamma_{\lambda_1,\dots,\lambda_k}\big(u_1,\dots,u_k)
\in\langle\partial+\lambda_1+\dots+\lambda_k\rangle$,
then the right-hand side of \eqref{eq:Phi}
lies in $\langle\partial+\lambda_1+\dots+\lambda_{k+1}\rangle$ as well,
hence it vanishes in the quotient
$M[\lambda_1,\dots,\lambda_k]/\langle\partial+\lambda_1+\dots+\lambda_k\rangle$.
Hence, $\Phi(\gamma)$ is a well defined map $\Omega(\mc V)^{\otimes k}\to M[\lambda_1,\dots,\lambda_k]/\langle\partial+\lambda_1+\dots+\lambda_k\rangle$.

To show that $\Phi(\gamma)$ lies in $C^k(\Omega(\mc V),M)$
we need to check that it satisfies conditions \eqref{eq:poly-lambda} and \eqref{eq:poly-skew}.
Condition \eqref{eq:poly-lambda} follows from the sesquilinearity \eqref{eq:sesqPVA} of $\gamma$, while \eqref{eq:poly-skew} follows from the skewsymmetry
\eqref{eq:poly-skewPVA} of $\gamma$.

Injectivity of $\Phi$ is obvious: if $\Phi(\gamma)=0$, then
$$
\gamma_{\lambda_1,\dots,\lambda_k}(u_1,\dots,u_k)=\Phi(\gamma)_{\lambda_1,\dots,\lambda_k}(d u_1,\dots, d u_k)=0\,.
$$
As for surjectivity, it suffices to exhibit explicitly the  inverse map
$C^k(\Omega(\mc V),M)\to \Gamma^k(\mc V,M)$, given by
$$
\Phi^{-1}(\varphi)_{\lambda_1,\dots,\lambda_k}(u_1,\dots,u_k)
=\varphi_{\lambda_1,\dots,\lambda_k}(d u_1,\dots,d u_k)
\,.
$$

In order to complete the proof we need to show that $\Phi$ commutes with the differential: $\Phi(d_{\textrm{PVA}}\gamma)=d_{\textrm{LCAd}}\Phi(\gamma)$.
Since both sides of this equation lie in $C^{k+1}(\Omega(\mc V),M)$ and since the $\mc V[\partial]$-module $\Omega(\mc V)$ is generated by
elements of the form $du$, $u\in\mc V$,
it suffices to prove that
$\Phi(d_{\textrm{PVA}}\gamma)$ and $d_{\textrm{LCAd}}\Phi(\gamma)$
coincide when they are evaluated at elements
of the form $du_1,\dots,du_{k+1}$.
For this we have, by definitions \eqref{eq:differential} and \eqref{eq:differentialPVA}
of $d_{\textrm{LCAd}}$ and $d_{\textrm{PVA}}$ respectively, and the LCAd structure of $\Omega(\mc V)$ given by
\eqref{eq:kahler-theta} and \eqref{eq:kahler-lambda}, and
the LCAd module structure on $M$ given by \eqref{eq:action4}
\begin{align*}
&(d_{\textrm{LCAd}}\Phi(\gamma))_{\lambda_1,\dots,\lambda_{k+1}}(du_1,\dots, du_{k+1})
\\
&=
\sum_{i=1}^{k+1}(-1)^{i+1}
(du_i)_{\lambda_i}\Big(
\Phi(\gamma)_{\lambda_1,\stackrel{i}{\check\dots},\lambda_{k+1}}(du_1,\stackrel{i}{\check\dots},du_{k+1})
\Big) \\
&+
\sum_{\substack{i,j=1 \\ i<j}}^{k+1}(-1)^{i+j}
\Phi(\gamma)_{\lambda_i+\lambda_j,\lambda_1,\stackrel{i}{\check\dots}\stackrel{j}{\check\dots},\lambda_{k+1}}
\big([{du_i}_{\lambda_i}{du_j}],du_1,\stackrel{i}{\check\dots}\stackrel{j}{\check\dots},du_{k+1}\big)
\\
&=
\sum_{i=1}^{k+1}(-1)^{i+1}
{u_i}_{\lambda_i}\Big(
\gamma_{\lambda_1,\stackrel{i}{\check\dots},\lambda_{k+1}}(u_1,\stackrel{i}{\check\dots},u_{k+1})\Big)
\\
&+
\sum_{\substack{i,j=1 \\ i<j}}^{k+1}(-1)^{i+j}
\gamma_{\lambda_i+\lambda_j,\lambda_1,\stackrel{i}{\check\dots}\stackrel{j}{\check\dots},\lambda_{k+1}}
\big(\{{u_i}_{\lambda_i}{u_j}\},u_1,\stackrel{i}{\check\dots}\stackrel{j}{\check\dots},u_{k+1}\big)
\\
&=(d_{\textrm{PVA}}\gamma)_{\lambda_1,\dots,\lambda_{k+1}}(u_1,\dots,u_{k+1})
=\Phi(d_{\textrm{PVA}}\gamma)_{\lambda_1,\dots,\lambda_{k+1}}(du_1,\dots,du_{k+1})
\,.
\end{align*}
\end{proof}

\begin{remark}\label{rem:5.3}
The same formula \eqref{eq:Phi} defines an isomorphism between the basic complexes  
$$
(\widetilde{\Gamma}^\bullet(\mc V,M),\widetilde{d}_{\textrm{PVA}})\stackrel{\sim}{\longrightarrow} (\widetilde{C}^\bullet(\Omega(\mc V),M),\widetilde{d}_{\textrm{LCAd}})
\,,
$$
as well as an isomorphism between the reduced complexes
$$
(\overline{\Gamma}^\bullet(\mc V,M),\overline{d}_{\textrm{PVA}})\stackrel{\sim}{\longrightarrow} (\overline{C}^\bullet(\Omega(\mc V),M),\overline{d}_{\textrm{LCAd}})
\,.
$$
\end{remark}
\begin{remark}
Recall \cite{BDK} that if $R$ is an LCA and $M$ is an $R$-module, then we have an isomorphism of complexes $C^k(R,M)\stackrel{\sim}{\longrightarrow}\Gamma^k(S(R),M)$ from the LCA cohomology complex of $R$ to the variational PVA cohomology complex of $S(R)$ with coefficients in the same module $M$.
The analogous result does not seem to hold in the case of an LCAd $E$, if we consider the associated PVA $S_A(E)$ defined in Theorem \ref{thm:SE}. In fact,
we only have an injective map $\Phi:C^k(E,S_A(E))\hookrightarrow \Gamma^k(S_A(E),S_A(E))$ mapping $\varphi\in C^k(E,S_A(E))$ to the unique $k$-cochain $\Phi(\varphi)\in\Gamma^k(S_A(E),S_A(E))$ defined by setting
$\Phi(\varphi)|_{E^{\otimes k}}=\varphi$,
$\Phi(\varphi)_{\lambda_1,\dots,\lambda_k}(u_1,\dots, u_k)$ vanishes if one of the entries $u_i$ lies in $A$
and it is extended to $S_A(E)^{\otimes k}$ by the Leibniz rules \eqref{eq:leibPVA}.
This map is injective but not surjective, indeed the $\lambda$-bracket \eqref{eq:lambda-vvv} on $S_A(E)$ lies in $\Gamma^2(S_A(E),S_A(E))$, but  it does not belong to the image of $\Phi$ unless $\theta=0$.
\end{remark}
\begin{example}
Let $\mf g$ be a Lie algebra and consider the corresponding current LCA $R=\mb F[\partial]\mf g$ with $\lambda$-bracket $[a_\lambda b]=[a,b]$, for 
$a,b\in\mf g$ and extended to $R$ by sesquilinearity. The symmetric algebra $\mc V=S(R)$ is naturally a PVA.
The LCAd of K\"ahler differentials can be explicitly described as the free $\mc V[\partial]$-module generated by $d\mf g\simeq\mf g$, namely
 $\Omega(\mc V)=\mc V[\partial]d\mf g\simeq\mc V[\partial]\otimes \mf g$.
 In particular, we have the identifications
 $$
 R\simeq\mb F[\partial]\otimes\mf g\subset\mc V[\partial]\otimes\mf g\simeq\Omega(\mc V)
 \,,
 $$
 which provide
an injective LCA homomorphism from $R$ to
 $\Omega(\mc V)$ given by $a\mapsto da$, $a\in R$.
 The $\lambda$-bracket on $\Omega(\mc V)$ is uniquely extended from the $\lambda$-bracket on $R$ by \eqref{eq:tot-form}, using the anchor map
 $\theta:\Omega(\mc V)\to\CDer(\mc V)$, which is the left $\mc V[\partial]$-module homomorphism defined on generators by $\theta(da)_\lambda(f)=\{a_\lambda f\}$, $a\in\mf g$, $f\in S(R)=\mc V$, where $\{\cdot_\lambda\cdot\}$ denotes the $\lambda$-bracket of the PVA $\mc V$. By Theorem \ref{thm:main} we can compute the cohomology of the LCAd $\Omega(\mc V)$ with coefficients in a module $M$ by computing the
 variational PVA cohomology of $S(R)$ with coefficients in $M$, see \cite{BDK}.
 \end{example}

\begin{remark}
There is another relation between the variational PVA and the LCAd cohomology complexes.
Indeed, let $\mc V$ be a PVA which is finitely generated as a differential algebra,
and consider the LCAd $\CDer(\mc V)$ over $\mc V$ defined in Example \ref{ex:CDer-LCAd}.
It is straightforward to check that we have a morphism of complexes 
$$
\Psi:C^k(\CDer(\mc V),\mc V)\to \Gamma^k(\mc V,\mc V)\,,
$$
defined by ($\varphi\in C^k(\CDer(\mc V),\mc V)$, $f_1,\dots,f_k\in\mc V$)
$$
(\Psi\varphi)_{\lambda_1,\dots,\lambda_k}(f_1,\dots,f_k)
=\varphi_{\lambda_1,\dots,\lambda_k}(X_{f_1},\dots,X_{f_k})
\,,
$$
where $X_f=\{f_\lambda\,\cdot\}$ as defined in Example \ref{exa:Btilde}. 
\end{remark}


\end{document}